\documentclass[letterpaper, 11pt,  reqno]{amsart}

\usepackage{amsmath,amssymb,amscd,amsthm,amsxtra, esint}
\usepackage{tikz} 
\usepackage{color}
\usepackage{hyperref}

\usepackage{cases}

\usepackage[left=32mm, right=32mm, 
bottom=27mm]{geometry}

\makeatletter
\renewcommand{\subjclassname}{%
  \textup{1991} Mathematics Subject Classification}
\@xp\let\csname subjclassname@1991\endcsname \subjclassname
\@namedef{subjclassname@2000}{%
  \textup{2000} Mathematics Subject Classification}
\@namedef{subjclassname@2010}{%
  \textup{2010} Mathematics Subject Classification}
  \@namedef{subjclassname@2020}{%
  \textup{2020} Mathematics Subject Classification}
\makeatother

\allowdisplaybreaks[2]

\usepackage{color}

\definecolor{gr}{rgb}   {0.,   0.69,   0.23 }
\definecolor{bl}{rgb}   {0.,   0.5,   1. }
\definecolor{mg}{rgb}   {0.85,  0.,    0.85}
\definecolor{yl}{rgb}   {0.8,  0.7,   0.}
\definecolor{or}{rgb}  {0.7,0.2,0.2}

\usepackage{tikz}
\usepackage{xcolor}
\usetikzlibrary{calc,intersections,through,backgrounds}
\usetikzlibrary{decorations.pathreplacing,decorations.pathmorphing,decorations.markings,decorations.shapes}
\usetikzlibrary{shapes}
\usetikzlibrary{external}

\newtheorem{theorem}{Theorem} [section]

\newtheorem{lemma}[theorem]{Lemma}
\newtheorem{proposition}[theorem]{Proposition}
\newtheorem{remark}[theorem]{Remark}

\newtheorem{corollary}[theorem]{Corollary}

\DeclareMathOperator*{\supp}{supp}

\newcommand{\1}{\hspace{0.5mm}\text{I}\hspace{0.5mm}}
\newcommand{\II}{\text{I \hspace{-2.8mm} I} }
\newcommand{\III}{\text{I \hspace{-2.9mm} I \hspace{-2.9mm} I}}

\renewcommand{\i}{\iota}
\newcommand{\noi}{\noindent}
\newcommand{\Z}{\mathbb{Z}}
\newcommand{\R}{\mathbb{R}}
\newcommand{\C}{\mathbb{C}}

\newcommand{\T}{\mathbb{T}}

\newcommand{\deff}{\stackrel{\textup{def}}{=}}

\let\Re=\undefined\DeclareMathOperator*{\Re}{Re}

\let\P= \undefined
\newcommand{\P}{\mathbf{P}}

\newcommand{\E}{\mathbb{E}}

\renewcommand{\L}{\mathcal{L}}

\newcommand{\RR}{\mathcal{R}}

\newcommand{\F}{\mathcal{F}}

\newcommand{\al}{\alpha}
\newcommand{\be}{\beta}
\newcommand{\dl}{\delta}

\newcommand{\nb}{\nabla}

\newcommand{\too}{\longrightarrow}

\newcommand{\Dl}{\Delta}
\newcommand{\eps}{\varepsilon}
\newcommand{\kk}{\kappa}

\newcommand{\G}{\Gamma}
\newcommand{\ld}{\lambda}
\newcommand{\Ld}{\Lambda}
\newcommand{\s}{\sigma}
\newcommand{\Si}{\Sigma}
\newcommand{\ft}{\widehat}

\newcommand{\wt}{\widetilde}
\newcommand{\cj}{\overline}
\newcommand{\dx}{\partial_x}
\newcommand{\dt}{\partial_t}
\newcommand{\dd}{\partial}

\newcommand{\Dlg}{\Delta_\gm}

\newcommand{\om}{\omega}
\renewcommand{\O}{\Omega}

\newcommand{\A}{\mathcal{A}}
\newcommand{\les}{\lesssim}

\newcommand{\jb}[1]
{\langle #1 \rangle}

\newcommand{\ind}{\mathbf 1}

\renewcommand{\S}{\mathcal{S}}

\newcommand{\J}{\mathcal{J}}
\newcommand{\gm}{\mathbf{\mathrm{g}}}

\newcommand{\gf}{\mathfrak{g}}

\newcommand{\PP}{\mathbb{P}}

\newcommand{\M}{\mathcal{M}}

\newcommand{\N}{\mathbb{N}}

\newcommand{\Y}{\mathcal{Y}}
\newcommand{\NN}{\mathcal{N}}
\newcommand{\ZZ}{\mathcal{Z}}
\newcommand{\GG}{\mathcal{G}}
\newcommand{\QQ}{\mathcal{Q}}

\renewcommand{\H}{\mathcal{H}}
\newcommand{\D}{\mathcal{D}}

\newcommand{\Id}{\textup{Id}}

\makeatletter
\def\DeclareSymbol#1#2#3{\expandafter\gdef\csname MH@symb@#1\endcsname{\tikz[baseline=#2, scale=.18]{#3}}}
\def\<#1>{\ensuremath{\mathchoice{\tikzsetnextfilename{macros#1}{\color{black}\csname MH@symb@#1\endcsname}}{\tikzsetnextfilename{macros#1}{\color{black}\csname MH@symb@#1\endcsname}}{\tikzsetnextfilename{macros#1}\scalebox{.7}{\color{black}\csname MH@symb@#1\endcsname}}
{\tikzsetnextfilename{macros#1}\scalebox{.5}{\color{black}\csname MH@symb@#1\endcsname}}}} 
\makeatother 
\DeclareSymbol{1}{0}{\path (-.4,0) -- (.4,0); \draw[line width=0.2mm] (0,0) -- (0,1.2) 
node[circle, fill, draw, solid, inner sep=0pt, minimum size=1.8pt] {};}

\newcommand{\ds}{\displaystyle}

\newcommand{\V}{\mathbb{V}}

\newcommand{\Law}{\mathrm{Law}}
\newcommand{\Ha}{\mathbb{H}_a}
\newcommand{\dr}{\theta}
\newcommand{\Dr}{\Theta}
\newcommand{\Vol}{\mathrm{Vol}}
\newcommand{\GF}{\mathfrak{G}}
\newcommand{\DF}{\mathfrak{D}}
\newcommand{\zz}{\mathfrak{z}}
\newcommand{\GGG}{\mathbb{G}}

\newtheorem*{ackno}{Acknowledgements}

\numberwithin{equation}{section}
\numberwithin{theorem}{section}

\begin{document}
\baselineskip = 15pt

\title[Invariant Gibbs measure for ExpNLS]
{Invariant Gibbs measure for a Schr\"odinger equation with exponential nonlinearity}

\author[T.~Robert]
{Tristan Robert}

\address{
Tristan Robert\\
Universit\'e Rennes 1\\ CNRS, IRMAR - UMR 6625\\ F-35000 Rennes\\ France}

\email{tristan.robert@ens-rennes.fr}



\begin{abstract}
We investigate the invariance of the Gibbs measure for the fractional Schr\"odinger equation of exponential type (expNLS) $i\dt u + (-\Dl)^{\frac{\al}2} u = 2\gamma\be e^{\be|u|^2}u$ on $d$-dimensional compact Riemannian manifolds $\M_d$, for  a dispersion parameter $\al>d$, some coupling constant $\be>0$, and $\gamma\ne 0$. (i) We first study the construction of the Gibbs measure for (expNLS). We prove that in the defocusing case $\gamma>0$, the measure is well-defined in the whole regime $\al>d$ and $\be>0$ (Theorem~\ref{THM:Gibbs1}~(i)), while in the focusing case $\gamma<0$ its partition function is always infinite for any $\al>d$ and $\be>0$, even with a mass cut-off of arbitrary small size (Theorem~\ref{THM:Gibbs1}~(ii)). (ii) We then study the dynamics (expNLS) with random initial data of low regularity. We first use a compactness argument to prove weak invariance of the Gibbs measure in the whole regime $\al>d$ and $0<\be < \be^\star_\al$ for some natural parameter $0<\be^\star_\al\sim (\al-d)$ (Theorem~\ref{THM:GWP1}~(i)). In the large dispersion regime $\al>2d$, we can improve this result by constructing a local deterministic flow for (expNLS) for any $\be>0$. Using the Gibbs measure, we prove that solutions are almost surely global for $0<\be \ll\be^\star_\al$, and that the Gibbs measure is invariant (Theorem~\ref{THM:GWP1}~(ii)). (iii) Finally, in the particular case $d=1$ and $\M=\T$, we are able to exploit some probabilistic multilinear smoothing effects to build a probabilistic flow for (expNLS) for $1+\frac{\sqrt{2}}2<\al\le 2$, locally for arbitrary $\be>0$ and globally for $0<\be \ll \be^\star_\al$ (Theorem~\ref{THM:GWP2}).

\end{abstract}



\maketitle
%


\tableofcontents

\baselineskip = 14pt

\section{Introduction}
\subsection{The exponential NLS}
The purpose of this work is to investigate the construction and the invariance of the Gibbs measure for a nonlinear fractional Schr\"odinger equation of exponential type
\begin{align}\label{expNLS}
i\dt u  + (-\Dl)^{\frac{\al}2} u = 2\gamma \be e^{\be|u|^2}u, \hspace*{5mm} (t,x)\in\R\times\M_d,
\end{align}
where $(\M_d,\gm)$ is a closed (compact, boundaryless) Riemannian manifold of dimension $d\ge 1$, $\al>0$ is the dispersion parameter, $\be>0$ is the coupling constant, and $\gamma\in\R$, $\gamma\neq0$ encodes the nature of the nonlinear interaction (defocusing for $\gamma>0$, focusing when $\gamma<0$).

\noi The model \eqref{expNLS} appears naturally in the context of nonlinear optics \cite{optics}. From a mathematical perspective, the exponential nonlinearity was proposed as the natural energy critical nonlinearity for the Schr\"odinger equation in dimension $d=2$. Indeed, recall that in the Euclidean space $\R^d$, the pure power nonlinear fractional Schr\"odinger equation
\begin{align}\label{NLS}
i\dt u + (-\Dl)^{\frac{\al}2}u = \pm|u|^{p-1}u
\end{align} 
has scaling critical regularity 
\begin{align}\label{scaling}
s_c(p,\al) = \frac{d}2-\frac{\al}{p-1}.
\end{align}
 In particular, rewriting \eqref{expNLS} as
\begin{align*}
i\dt u + (-\Dlg)^{\frac{\al}2}u =2\gamma \be\sum_{k\ge 0}\frac{\be^k}{k!}|u|^{2k}u, 
\end{align*}
we see from \eqref{scaling} that for \eqref{expNLS} the scaling critical regularity is formally 
\begin{align}\label{scaling2}
s_c(\infty,\al)=\frac{d}2
\end{align}
  for any $\al$. The energy (or Hamiltonian) of \eqref{expNLS} being given by
\begin{align}\label{hamiltonian}
\H_{\al,\be}(u) = \frac12\int_{\M_d}\big|\jb{\nabla}^{\frac{\al}2} u\big|^2dx + \gamma \int_{\M_d} e^{\be|u|^2}dx,
\end{align}
we thus see that $d=\al=2$ makes \eqref{expNLS} energy critical. This motivated the study of \eqref{expNLS} for $\al=2$ in $\R^2$ \cite{Cazenave,NaOz,Coll,Na,IMMN,Wang1,Wang2}. In particular, building on the previous global well-posedness result for small data for the\footnote{In \cite{Coll}, the Schr\"odinger equation with nonlinearity $e^{4\pi|u|^2}u-u$ is considered, which is equivalent to \eqref{expNLS} via a gauge transform. In \cite{IMMN} the nonlinearity $e^{4\pi|u|^2}u-u-4\pi|u|^2u$ is treated.} defocusing (expNLS) of \cite{Coll}, Ibrahim, Majdoub, Masmoudi, and Nakanishi proved in \cite{IMMN} that scattering occurs for the small data global solutions of \eqref{expNLS} when $\al=d=2$ on $\R^2$, and with $\be=4\pi$ related to the sharp constant in Moser-Trudinger's inequality \cite{Trudinger,Moser}.

In the case of a compact domain, however, scattering cannot occur, so that one can ask what the long-time behaviour of the flow looks like. The pioneering works \cite{LRS,BO94,BO96} on the standard NLS \eqref{NLS} with $\al=2$ on $\T^d$, $d=1,2$, paved the way to a probabilistic approach to this question. Namely, in view of Poincar\'e's recurrence theorem, one can look for the construction of a probability measure invariant under the flow of \eqref{expNLS}, which then indicates that the flow is recurrent, at least at the level of regularity of the support of the invariant measure. The probabilistic approach to nonlinear dispersive PDEs has since then attracted a tremendous amount of interest. We refer e.g. to \cite{BOP3,BO94,BO96,BS,BTT2,BTT1,CO,LRS,NORS,STz1,STz2,TT,Tz0,Tz,Zhidkov} and references therein for further examples and discussion on the construction and invariance of Gibbs measure for NLS type equations. In the following, we study this approach for the model \eqref{expNLS}.

\subsection{Gibbs measure associated to the Hamiltonian}
The exponential NLS equation \eqref{expNLS} is Hamiltonian, with Hamiltonian given by $\H_{\al,\be}$ \eqref{hamiltonian}. In particular $\H_{\al,\be}$ is conserved along the flow, and so is the mass
\begin{align}\label{mass}
\J(u) = \frac12\int_{\M_d}|u|^2dx.
\end{align}

In view of the conservation of these two quantities, we expect that any measure given formally by
\begin{align*}
``F(\J,\H_{\al,\be})du"
\end{align*}
for some function $F : \R^2 \to \R$ will be conserved by the flow of \eqref{expNLS}. In particular, we aim at defining the Gibbs type measure
\begin{align}\label{Gibbs1}
``d\rho_{\al,\be}(u) = \exp\Big(-\J(u)-\H_{\al,\be}(u)\Big)du".
\end{align}
The usual way is to write $\J+\H_{\al,\be}$ as a quadratic part in $u$ (the kinetic energy) plus a potential term
\begin{align}\label{V}
V_\be(u) = \int_{\M_d}e^{\be|u|^2}dx,
\end{align}
and then to interpret
\begin{align*}
d\rho_{\al,\be}(u) = \exp(-\gamma V_\be(u))\times \exp\Big\{-\frac12\int_{\M_d}\big|\jb{\nabla}^{\frac{\al}2}u\big|^2dx\Big\}du,
\end{align*}
where $\jb{\cdot} = \sqrt{1+|\cdot|^2}$.

The last part
\begin{align}\label{mu}
``d\mu_\al(u)=\exp\Big\{-\int_{\M_d}\big|\jb{\nabla}^{\frac{\al}2}u\big|^2dx\Big\}du"
\end{align}
can be interpreted as a Gaussian measure on $L^2(\M_d)$. Indeed, expanding $u$ on an eigenbasis $\{\varphi_n\}_{n\ge 0}$ of $L^2(\M_d)$ associated with the eigenvalues $\{-\ld_n^2\}_{n\ge 0}$ of the Laplace-Beltrami operator $\Dlg$ on $(\M_d,\gm)$, we can formally write
\begin{align*}
d\mu_\al(u) = \exp\Big\{-\frac12\sum_{n\in\N}\jb{\ld_n}^\al|\ft u_n|^2\Big\}\Big(\prod_{n\ge 0}d\ft u_n\Big) = \prod_{n\ge 0}e^{-\frac{\jb{\ld_n}^\al|\ft u_n|^2}{2}}d\ft u_n,
\end{align*}
where $\ft u_n$ stands for the $n$-th Fourier coefficient of $u\in L^2(\M_d)$, $n\ge 0$, i.e. $\ft u_n = \langle u,\varphi_n\rangle$ where $\langle \cdot,\cdot\rangle$ stands for the usual inner product in $L^2(\M_d)$. In particular, up to a constant factor we can recognize $e^{-\frac{\jb{\ld_n}^\al|\ft u_n|^2}{2}}d\ft u_n$ as a standard centered (complex-valued) Gaussian measure with variance $\frac1{\jb{\ld_n}^\al}$. This allows us to define properly the measure $\mu_\al$ as the law of the random series
\begin{align}\label{init}
u_\al^\om = \sum_{n\ge 0}\frac{g_n(\om)}{\jb{\ld_n}^{\frac{\al}2}}\varphi_n,
\end{align}
where $\{g_n\}_{n\ge 0}$ is a family of independent standard complex-valued Gaussian random variables on some probability space $(\O,\PP)$. Here the convergence of the series in \eqref{init} holds in $L^2(\O;H^s(\M_d))$ if and only if 
\begin{align}\label{support}
s<\frac{\al-d}2;
\end{align}
see Lemma~\ref{LEM:init} below.


In view of the previous discussion, our first result is concerned with the rigorous definition of the Gibbs measure $\rho_{\al,\be}$ in \eqref{Gibbs1}. For $N\ge 1$, we first define the truncated version
\begin{align}\label{GibbsM}
d\rho_{\al,\be,N}(u)\deff \ZZ_{\al,\be,N}^{-1}e^{-\gamma V_\be(\Pi_{\le N}u)}d\mu_\al(u),
\end{align}
where $\Pi_{\le N}$ is as above, and $\ZZ_{\al,\be,N} = \int e^{-\gamma V_\be(\Pi_{\le N}u)}d\mu_\al(u)$ is the partition function of $\rho_{\al,\be,N}$.
\begin{theorem}\label{THM:Gibbs1}
The following hold for any $d\ge 1$, $\al>d$ and $\be >0$:\\
\textup{(i) (Defocusing case)} If $\gamma>0$, then ${\displaystyle e^{-\gamma V_\be(u)}\in L^1(\mu_\al)}$. In particular the Gibbs measure
\begin{align}\label{Gibbs}
d\rho_{\al,\be}(u)= \frac{e^{-\gamma V_\be(u)}}{\int e^{-\gamma V_\be(u)}d\mu_\al(u)}d\mu_\al(u)
\end{align}
is a well-defined probability measure on $H^s(\M_d)$, $s<\frac{\al-d}{2}$, which is absolutely continuous with respect to $\mu_\al$. Moreover, $\rho_{\al,\be,N}\to \rho_{\al,\be}$ in total variation, as $N\to\infty$.\\
\textup{(ii) (Focusing case)} If $\gamma <0$, then for any $K>0$, ${\displaystyle e^{-\gamma V_\be(u)}\mathbf{1}_{\|u\|_{L^2}\le K}\not\in L^1(\mu_\al)}$. In particular the Gibbs measure, even with an arbitrarily small mass cut-off, cannot be defined as a probability measure.
\end{theorem}

\noi Let us comment on Theorem~\ref{THM:Gibbs1}. In the non-singular regime $\al>d$ which Theorem~\ref{THM:Gibbs1} is concerned with, the potential $V_\be$ \eqref{V} is almost surely well-defined on the support of the Gaussian measure $\mu_\al$ \eqref{mu} (see Lemma~\ref{LEM:init} below). Since it is bounded below, the construction of the Gibbs measure in the defocusing case in Theorem~\ref{THM:Gibbs1}~(i) is rather straightforward; see Subsection~\ref{SUBSEC:defoc}. In comparison, in the singular case $\al=d$ and for polynomial interactions, much effort has been made in constructive quantum field theory \cite{GF,Simon} to construct the Gibbs measure. Indeed in this latter case the potential needs to be renormalized through an approximation procedure and Wick ordering, which in particular makes it non bounded below. We refer to \cite{OT} for a pedestrian construction in the case $\al=d=2$ and polynomial interaction in place of $V_\be$.

The focusing case of Theorem~\ref{THM:Gibbs1}~(ii) is more subtle. In the seminal paper \cite{LRS}, Lebowitz, Rose, and Speer studied the construction of the Gibbs measure for the standard (NLS) \eqref{NLS} with $\al=2$ in dimension $d=1$ on $\M=\T$ for $p\le 6$, and showed that the measure is finite for arbitrary mass cut-off in the case $p<5$, and in the case $p=6$ for small enough mass cut-off depending on the size of the ground state associated to \eqref{NLS}. The same result as Theorem~\ref{THM:Gibbs1}~(ii) holds in particular for \eqref{NLS} with $d=1$, $\al=2$ and $p>6$. See also \cite{TW} for a simple proof of this result. Since $V_\be(u) \ge \frac{\be^k}{k!}|u|^{2k}$ for any $k\in\N$, the case $\al=2$, $d=1$ of Theorem~\ref{THM:Gibbs1}~(ii) is thus straightforward from \cite{LRS,TW}, at least for large $K$. Similarly, \cite{BS} proves that the partition function for the focusing Gibbs measure associated with the quartic potential $|u|^4$ in place of $V_\be$ in the case $\al=d=2$ is infinite. However, one could have hoped that taking $\al>d$ large enough and $K,\be>0$ small enough might allow the construction of the mass-truncated focusing Gibbs measure. Theorem~\ref{THM:Gibbs1}~(ii) shows that this is not the case. In the case of large $K$, a simple argument similar to that of \cite{LRS}, based on Cameron-Martin's theorem and a choice of deterministic drift\footnote{Here we use properly scaled functions supported at high frequency instead of concentrating solitons as in \cite{LRS}. See Lemma~\ref{LEM:fM} below.} with energy blowing-up to $-\infty$ suffices to prove Theorem~\ref{THM:Gibbs1}~(ii) for any $\al>d$ and any $\be>0$: see e.g. Proposition A.1 in \cite{ORTW} in the context of an exponential interaction. To show that Theorem~\ref{THM:Gibbs1}~(ii) also holds for arbitrarily small $K$, we use a more subtle construction introduced in \cite{OOT} and refined in \cite{OST}, based on a variational formulation to express the partition function of the Gibbs measure; see \cite{BD,Ust} and \eqref{var} below. This latter formulation proved to be an effective tool in the construction of singular Gibbs measure with polynomial interactions in the defocusing case mentioned above \cite{BG,OOT,Bringmann,ORSW2}. It allows to replace the deterministic drift in Cameron-Martin's formula with one built on the random variable $u_\al$ \eqref{init} with law $\mu_\al$. Taking formally $-u_\al$ plus a perturbation with large negative energy as a drift then allows to get the result.

\begin{remark}\label{REM:critical}
\rm~\\
\noi\textup{(i)} In the critical case $\al=d$, the measure $\mu_\al$ becomes supported on distributions rather than functions. This can be easily seen by computing for any $N\ge 0$ and $x\in\M_d$,
\begin{align}\label{sN}
\s_{\al,N}(x)&\deff \E\big|\Pi_{\le N}u_\al^\om(x)\big|^2 = \E\Big|\sum_{\ld_n\le N}\frac{g_n}{\jb{\lambda_n}^{\frac{\al}2}}\varphi_n(x)\Big|^2\notag \\&= \sum_{\ld_n\le N}\frac{|\varphi_n(x)|^2}{\jb{\lambda_n}^\al}\sim \sum_{\ld_n\le N}\frac1{\jb{\lambda_n}^\al},
\end{align} 
uniformly in $N\ge 1$, where $\Pi_{\le N}$ stands for the spectral projector on the set of frequencies $\{\ld_n\le N\}$, and where the last estimate follows from Corollary~\ref{LEM:avg eig} below. From Weyl's law
\begin{align}\label{Weyl}
\ld_n \sim n^{\frac1d}
\end{align}
we get that the partial sums $\s_{\al,N}$ in \eqref{sN} then converges in $L^\infty(\M)$ to
\begin{align}\label{sal}
\s_\al(x) = \sum_{n\ge 0}\frac{|\varphi_n(x)|^2}{\jb{\ld_n}^\al}
\end{align}
 if and only if $\al>d$. 
 
As a result, any nonlinear potential (and in particular $V_\be$) would need a renormalization to be able to define the Gibbs measure. As mentioned above, in the defocusing case, for $\al=d=2$ and a polynomial interaction, this is done in \cite{OT}. For the same regime and replacing the Gaussian measure $\mu_\al$ by the law $\mu_\al^\R$ of $\Re(u_\al)$, the defocusing Gibbs measure with (Wick renormalized) exponential interaction $e^{\be u}$ in place of $V_\be$ was proved to be constructible if and only if $\be^2< 8\pi$ in \cite{HK}; see also\cite{DKRV,DRV,GRV,HKK1,HKK2,ORW,ORTW} for recent results on this model, in particular the invariance of the Gibbs measure for the stochastic heat and wave dynamics. The point is that when $\al=d$, the random variable $u_\al$ (or $\Re(u_\al)$) is log-correlated, namely $G(x,y)\deff \E[u_\al(x)\cj u_\al(y)]\sim \frac1{2\pi}\log d(x,y)$. By property of the Wick renormalization, the renormalized exponential $\,:\!e^{\be u}\!:\,$ has then correlation function $\exp\big(\be G(x,y)\big)\sim d(x,y)^\frac{\be^2}{4\pi}$ which is locally integrable if and only if $\be^2<8\pi$. Note however that there is no Schr\"odinger analogue of the dynamics studied in \cite{HKK1,HKK2,ORW,ORTW} for the potential $e^{\be u}$, since this latter does not generate a Hamiltonian dynamics for NLS. \\

\noi \textup{(ii)} In order to get existence of the measure $\rho_{\al,\be}$ \eqref{Gibbs1} for some $\be>0$ in the critical case $\al=d$, a renormalization procedure would thus also be needed. This amounts to defining the renormalized truncated potential for $N\ge 1$ as
\begin{align}\label{renorm}
\V_{\be,N}(u) &\deff \int_\M\sum_{k\ge 0}\frac{\be^k}{k!}\,:\!|\Pi_{\le N}u|^{2k}\!:\,dx\notag\\
& = \int_\M\frac1{1+\be\s_{\al,N}(x)}\exp\Big(\frac{\be}{1+\be\s_{\al,N}(x)}|\Pi_{\le N} u(x)|^2\Big)dx.
\end{align}
Here $\,:\!|\Pi_{\le N}u|^{2k}\!:\, = (-1)^k(k!)L_k\big(|\Pi_{\le N}u(x)|^2;\s_N(x)\big)$ denotes the Wick renormalization of the potential $|u|^{2k}$, where $L_k(y;\s)$ are the generalized Laguerre polynomials. The last formula then follows from the expression for the generating function of the $L_k$, see e.g. \cite{OT}. Then one should prove convergence in measure of $\V_{\be,N}$ for some $\be>0$. One major difficulty here is that one deals with the exponential of $|u_\al|^2$ instead of $\Re(u_\al)$ as discussed above, which is not Gaussian anymore, and with correlation function behaving like $\big(\log d(x,y)\big)^2$. Thus it is not clear at the moment whether the Gibbs measure with renormalized potential \eqref{renorm} is well-defined, even in the defocusing case. We plan to pursue this construction in a follow-up work.\\

\noi \textup{(iii)} On the other hand, for $\al=d$, Oh, Seong, and Tolomeo showed in \cite{OST} that the Gibbs measure has again infinite mass in the focusing case for the renormalized quartic potential $\,:\!|u|^4\!:\,$ in place of $V_\be(u)$, even with a (Wick-ordered) mass cut-off of arbitrarily small size. Thus in the focusing case, we expect that the Gibbs measure associated with the renormalized potential \eqref{renorm} is also ill-defined.
\end{remark}

\subsection{Invariance of the Gibbs measure}

We now move on to the construction of solutions to \eqref{expNLS} leaving $\rho_{\al,\be}$ invariant.

\begin{theorem}\label{THM:GWP1}
Let $d\ge 1$ and $\al>d$. Then the following hold:\\
\textup{(i) (Weak invariance for subcritical dispersion)} For any $0<\be<\|\s_\al\|_{L^\infty}^{-1}$ and $0<s<\frac{\al-d}2$, there is a random variable $u$ in $C(\R;H^s(\M_d))$ which solves \eqref{expNLS} in the distributional sense. Moreover, $\rho_{\al,\be}$ is invariant under the map $u(0)\mapsto u(t)$, $t\in\R$.\\
\textup{(ii) (Strong invariance for higher dispersion)} If $\al>2d$, for any $0<\be\ll\|\s_\al\|_{L^\infty}^{-1}$, the exponential NLS \eqref{expNLS} on $(\M,\gm)$ is $\rho_{\al,\be}$-almost surely globally well-posed, and the Gibbs measure $\rho_{\al,\be}$ is invariant under the flow. More precisely: for $\frac{d}2<s<\frac{\al-d}2$ , any $\be>0$ and $u_0\in H^s(\M)$, there exists $T>0$ and a unique solution to the Cauchy problem for \eqref{expNLS} with data $u_0$ in $C([-T,T];H^s(\M))$. Moreover, if $\Phi(t) : u_0 \mapsto u(t)$ denotes the local flow map of \eqref{expNLS} on $H^s(\M)$, in the case $0<\be\ll\|\s_\al\|_{L^\infty}^{-1}$ there exists $\Si\subset H^s(\M)$ of full $\rho_{\al,\be}$-measure such that for any $u_0\in \Si$, $\Phi(t,u_0)$ is globally defined, and $\rho_{\al,\be}$ is invariant under $\Phi(t)$, in the sense that for any $\rho_{\al,\be}$-measurable $A\subset \Si$ it holds $\rho_{\al,\be}\big(\Phi(t)(A)\big)=\rho_{\al,\be}(A)$ for any $t\in\R$.
\end{theorem}
\noi The first part of Theorem~\ref{THM:GWP1} deals with the full sub-critical regime $\al>d$, but we can only claim existence of weak solutions to \eqref{expNLS} having $\rho_{\al,\be}$ as an invariant measure, but not the construction of a flow (in particular, no uniqueness) on the support of $\mu_\al$. On the other hand, Theorem~\ref{THM:GWP1}~(ii) is concerned with strong invariance of the Gibbs measure, in the sense that the equation \eqref{expNLS} is (unconditionally) globally well-posed on the support of $\rho_{\al,\be}$, and in particular there is uniqueness of the solution. Theorem~\ref{THM:GWP1}~(ii) then improves this latter result to get the uniqueness of the weak solutions constructed in Theorem~\ref{THM:GWP1}~(i). In the general setting of the above theorem we can however only deal with the high dispersion regime $\al>2d$.

Theorem~\ref{THM:GWP1}~(i) is proved using a very flexible compactness argument first introduced in \cite{AC,DPD02} and then developed in the context of dispersive PDEs in \cite{BTT1}. This is based on the study of the approximate dynamics
\begin{align}\label{expNLSN}
(i\dt + (-\Dlg)^{\frac{\al}2})u_N =2\gamma \be\Pi_{\le N}\Big[e^{\be |\Pi_{\le N} u|^2}\Pi_{\le N}u\Big].
\end{align}
Since $\rho_{\al,\be}=\lim_{N\to\infty}\rho_{\al,\be,N}$, it is natural to look at the Hamiltonian dynamics \eqref{expNLSN} which leaves $\rho_{\al,\be,N}$ invariant. Then one proves that there is a subsequence $N_k\to\infty$ and a sequence $\wt u_{N_k}$ of $C(\R;H^s(\M))$-valued random variables with the same law $\rho_{\al,\be,N}$ as $u_{N_k}$ such that $\wt u_{N_k}$ converges almost surely in $C(\R;H^s(\M))$ to some random variable $u$, which is proved to be a distributional solution to \eqref{expNLS}. The convergence $\rho_{\al,\be,N}\to\rho_{\al,\be}$ then ensures that $u$ has $\rho_{\al,\be}$ has an invariant measure.

For Theorem~\ref{THM:GWP1}~(ii), the dispersion $\al>2d$ is large enough so that the support of the Gibbs measure lies at sub-critical regularity in view of \eqref{scaling2} and \eqref{support}. Thus local well-posedness can be proved in a purely deterministic manner. However, note that the support $H^s(\M)$, $s<\frac{\al-d}2$ of $\rho_{\al,\be}$ and $\mu_\al$ (note that $\rho_{\al,\be}\ll\mu_\al$) is always larger than the energy space $H^{\frac\al2}(\M)$, which makes global well-posedness for deterministic initial data non trivial since there are no conservation laws at this level of regularity. Here we rely on Bourgain's invariant measure argument \cite{BO94}, which exploits the invariance of the Gibbs measure as a replacement for a conservation law at low regularity, to get almost sure globalization of the local flow map. This argument allows us to construct a set of large $\rho_{\al,\be}$-measure for which we can iterate a great number of times the deterministic local Cauchy theory without the solution to grow too much. The restriction on $0<\be\ll \|\s_\al\|_{L^\infty}^{-1}$ thus comes from the smallness of the local time compared to the $\rho_{\al,\be}$-measure of balls of $H^s(\M)$ (see the tail estimate in Lemma~\ref{LEM:init} below).

\begin{remark}\label{Rk:BB}\rm~\\
In \cite{BB1,BB2,BB3}, Bourgain and Bulut developped a flexible argument leading to a stronger result than that provided by the compactness method. Namely, their method provides convergence of the whole sequence $u_N$ for the truncated dynamics and without changing the base probability space as in the compactness method. However this usually does not provide us with a flow map. See the discussion in \cite{STz1}. In particular, in \cite[Theorem 7]{STz1}, Sun and Tzvetkov proved the convergence of the solutions $u_N$ to the truncated dynamics associated with the defocusing fractional NLS \eqref{NLS}, in the whole range $\al>d$, and for any $p>1$. This argument relies on the observation that the nonlinear flow has, with high probability, the same enhanced integrability property than that of the linear flow (see \eqref{linear} below). This allows to obtain, with high probability, good controls of the truncated dynamics $u_N$ for times $\tau_N =O\big((\log N)^{-\frac1p}\big)$, related to the local time of existence in the Cauchy theory. However in our case \eqref{expNLSN} we have formally $p\to\infty$, and thus we should replace $\tau_N$ by $e^{-c(\log N)^2}$. This makes the number of iterations needed to control $u_N$ on times $O(1)$ too large so that we cannot ensure the summability of the probabilities to have bad controls for each $N$, even along dyadic sequences. So it seems that this argument fails in our context.
\end{remark}

When $\M_d=\T$, we can improve on the general result of Theorem~\ref{THM:GWP1} by using probabilistic multilinear smoothing effects to get existence and uniqueness of solutions for some $\al\le 2d$.
\begin{theorem}\label{THM:GWP2}
Let $d=1$ and $\M_d=\T$ with the standard metric. For any $1+\frac{\sqrt{2}}2<\al\le 2$, the following hold:\\
\textup{(i) (Almost sure local well-posedness)} For any $\gamma\in\R$ and $\be>0$, the Cauchy problem for \eqref{expNLS} is $\mu_\al$-almost surely locally well-posed. More precisely: for $0<s<\frac{\al-1}2$, there exists $\Si\subset H^{s}(\T)$ of full $\mu_\al$-measure such that for any $u_0\in \Si$, there exists $T>0$ and a solution to the Cauchy problem for \eqref{expNLS}  on $[-T,T]$ with data $u_0$, unique in the class 
\begin{align}\label{class}
\GGG^{-1}\big(e^{it(-\dx^2)^{\frac{\al}2}}u_0+X^{s_1,b}(T)\big),
\end{align} 
for some $s_1,b>\frac12$.\\
\textup{(ii) (Almost sure global well-posedness and invariance of the Gibbs measure)} If $\gamma>0$ and $0<\be\ll\|\s_\al\|_{L^\infty}^{-1}$, then the Cauchy problem for \eqref{expNLS} is $\rho_{\al,\be}$-almost surely globally well-posed, and the Gibbs measure $\rho_{\al,\be}$ is invariant under the flow. More precisely, if $\Si$ is as in \textup{(i)}, and if $\Phi(t) : u_0 \mapsto u(t)$ denotes the flow map of \eqref{expNLS} on $\Si$, then $\Phi(\Si)=\Si$, and $\rho_{\al,\be}$ is invariant under $\Phi(t)$ in the sense that for any $\rho_{\al,\be}$-measurable $A\subset \Si$ it holds $\rho_{\al,\be}\big(\Phi(t)(A)\big)=\rho_{\al,\be}(A)$ for any $t\in\R$.
\end{theorem}

\noi Here the gauge transform $\GGG$ and the Bourgain space $X^{s_1,b}(T)$ are defined in Subsection~\ref{SUBSEC:gauge} below.

 When $\al\le 2d$, the support of the Gibbs measure \eqref{support} then lies in spaces of super-critical regularity, and thus a purely deterministic approach to the Cauchy problem for \eqref{expNLS} on the support of $\rho_{\al,\be}$ fails, so that the solutions obtained in Theorem~\ref{THM:GWP2} are a manifestation of a probabilistic smoothing effect.
 
 The proof of Theorem~\ref{THM:GWP2} relies on the now classical probabilistic approach of Bourgain \cite{BO96} and consists at looking for a solution to \eqref{expNLS} under the form
\begin{align}\label{ansatz}
u = \<1>(u_0) + \wt w \qquad \text{where} \qquad \<1>(t,u_0)=e^{it(-\dx^2)^{\frac{\al}2}}u_0.
\end{align}
On the one hand, the randomization provides the linear evolution $\<1>(u_0)$ with better integrability property than for a deterministic data $u_0\in H^{\frac{\al-1}2-}(\T)$. Indeed, we have that $\mu_\al$-almost surely,
\begin{align}\label{linear}
\<1>(u_0) \in C_b(\R;W^{\frac{\al-1}2-,\infty}(\T));
\end{align}
see \eqref{Ysb}-\eqref{Zb} and Lemma~\ref{LEM:devY} below. On the other hand, the remainder $\wt w$ in \eqref{ansatz} is assumed to be smoother, belonging to the space $X^{s_1,b}$ for some $s_1,b>\frac12$ (see \eqref{Xb} below for the definition of Bourgain's spaces $X^{s_1,b}$). Then $\wt w$ solves the perturbed nonlinear Schr\"odinger equation
\begin{align}\label{wt}
i\dt \wt w + (-\dx^2)^{\frac{\al}2}\wt w = \be\sum_{k\ge 0}\frac{\be^k}{k!}|\<1>+\wt w|^{2k}(\<1>+\wt w)
\end{align}
starting from $\wt w(0)=0$. Thus we need to control uniformly all the multilinear forms in the right-hand side of \eqref{wt}, at spatial regularity $s_1>\frac12$. The most problematic terms are the high-low interactions involving a high frequency piece of the random linear evolution $\<1>$ since this latter only has regularity $\frac{\al-1}2-<\frac12$ for $\al\le 2$. To be able to systematically exploit some multilinear smoothing effect uniformly in all the multilinear forms, we make a gauge transform
\begin{align*}
v(t)=\exp\Big\{-i\int_0^t\GG\big(|u(t')|\big)dt'\Big\}u(t)
\end{align*}
for some choice of gauge $\GG$; see \eqref{gauge2} below. In particular, note that this gauge transform does not affect the initial data. Then the new remainder $w=v-\<1>$ will solve the forced nonlinear Schr\"odinger equation
\begin{align}\label{w}
(i\dt + (-\dx^2)^{\frac{\al}2})w &= 2\gamma\be\Big[ e^{\be|\<1>+w|^2}-\GG\big(|\<1>+w|\big)\Big](\<1>+w)\notag\\
&=2\gamma\be \sum_{k\ge 1}\frac{\be^k}{k!}\QQ_{2k+1}(\<1>+w).
\end{align}
The point now is that in the expansion \eqref{w}, some resonances have been removed in the $(2k+1)$-linear forms $\QQ_{2k+1}$ compared to the original power nonlinearities $|u|^{2k}u$. This allows us to get some nonlinear smoothing effect and to run a fixed point argument in $X^{s_1,b}$ to construct $w$. This provides local well-posedness of the remainder $w$ in \eqref{w}, and thus of the solution $u$ to \eqref{expNLS} after inverting the gauge transform $u(t)=\exp\Big(i\int_0^t\GG\big(|v(t')|\big)dt'\Big)v(t)$, with $v=\<1>+w$. The global well-posedness and the invariance of the Gibbs measure are then established through Bourgain's invariant measure argument \cite{BO94,BO96}. Here, compared to the implementation of this argument for Theorem~\ref{THM:GWP1}~(ii), it is more involved due to the space of initial data $Y^s(\T)\not\subset H^{s_1}(\T)$ used to build the probabilistic local flow; see the definition of $Y^s(\T)$ in \eqref{Ysb} below. In order to iterate the local probabilistic Cauchy theory, we thus need to control the flow on the sum space $Y^s(\T)+H^{s_1}(\T)$ as in \cite{STz1}.

\bigskip
\begin{remark}\rm
Observe that even the softest argument of Theorem~\ref{THM:GWP1}~(i) requires $0<\be< \|\s_\al\|_{L^\infty}^{-1}\sim \al-d \too 0$ as $\al\searrow d$. This prevents us for example from proving that the family of measures $\{\rho_{\al,\be}\}_\al$ is tight for some fixed $\be>0$ in order to try and treat the critical case $\al=d$. It is not clear if this issue persists after renormalizing the potential as in Remark~\ref{REM:critical}~(ii).
\end{remark}
\bigskip
\begin{remark}[Exponential interaction in the wave case]\rm~\\
In \cite{STz0}, Sun and Tzvetkov considered the invariance of Gibbs measure for the fractional wave equation with exponential nonlinearity
\begin{align}\label{wave}
(\dt^2+(-\Dlg)^{\frac{\al}2}+1)u+\be e^{\be u}=0
\end{align}
in the whole range $\al>d$ and $\be\in\R$. Writing the Duhamel formulation
\begin{align*}
u(t)=\cos(t\jb{\nabla}^\al)u(0)+\frac{\sin(t\jb{\nabla}^\al)}{\jb{\nabla}^\al}\dt u(0) + \be\int_0^t\frac{\sin\big((t-t')\jb{\nabla}^\al\big)}{\jb{\nabla}^\al}e^{\be u(t')}dt',
\end{align*}
we see that the smoothing property of the propagator for the wave equation (i.e. the smoothing term $\jb{\nabla}^{-\al}$ in the Duhamel formula) allows for a direct gain of smoothness and thus there is no need for exploiting multilinear smoothing in this case. 

In \cite{ORW}, thanks to this gain of smoothness, the author with Oh and Wang was able to treat \eqref{wave} in the case $\al=d=2$ and some small $|\be|$. The argument relied on a refinement of the linear ansatz \eqref{ansatz}, in the spirit of the theory of para-controlled distributions used in the setting of singular parabolic stochastic equations \cite{GIP}, but did not exploit any dispersive multilinear smoothing effect either. 
\end{remark}
\bigskip
\begin{remark}[Invariant Gibbs measure for the cubic fractional NLS]\rm~\\
In \cite{STz1}, Sun and Tzvetkov considered the invariance of Gibbs measure for the cubic fractional NLS \eqref{NLS} for $d=1$, $p=3$ and some $\al<2$. They used the same ansatz \eqref{ansatz} to build a probabilistic flow map as in Theorem~\ref{THM:GWP2}, but in their case the multilinear estimates required to close the fixed point argument in $X^{s_1,b}$ for the remainder $w$ are more involved than the one used for Theorem~\ref{THM:GWP2}. Indeed in our case the sub-critical space in which we build $w$ contains $L^{\infty}(\R\times\T)$, which allows for more straightforward multilinear estimates. They also established analogue of Theorem~\ref{THM:GWP1}~(i) even for some $\al<1$, and proved a stronger convergence statement in the case $\al>1$; see Remark~\ref{Rk:BB} above. Our study is thus similar in spirit to that of \cite{STz1}, namely using $\al$ as a way to test the roughness of initial data allowed in each of the methods describes above. In our case we also use the coupling constant $\be>0$ as a way to quantify the moment bounds with respect to $\mu_\al$ needed in the arguments. Indeed, a major difference between the polynomial and the exponential cases is that in the former case, the various nonlinear expressions of the random initial data belong to Wiener chaoses of finite degree, from which exponential tail estimates follow directly thanks to the Wiener chaos estimate. In our case, even the integrability of $V_\be(u)$ requires limitation on the range of $\be>0$. See Lemma~\ref{LEM:V} below. 
\end{remark}

\medskip
\begin{remark}[Probabilistic scaling criticality and random averaging operators/random tensors]\label{REM:DNY}\rm~~\\
It might be possible to refine the argument below for the multilinear estimates involved in the proof of Theorem~\ref{THM:GWP2} in the case of High$\times$High$\to$High interactions to show that we can take $\al>\frac32$ in Theorem~\ref{THM:GWP2}, and that this threshold for the dispersion solely comes from the control on High$\times$Low$\to$High interactions with a linear evolution of the random initial data as a high regularity input. This is precisely the bad interaction that is amenable to the more sophisticated method of random averaging operators and random tensors developed recently in \cite{DNY1,DNY2}. Thus we believe that there is a chance that Theorem~\ref{THM:GWP2} can be improved all the way down to the range $\al>1$. See also the discussions in a similar context in \cite{STz1,STz2}. However, note that in the critical case $\al=d$, the exponential NLS appears to be probabilistic critical in the sense of \cite{DNY1,DNY2}, and it is not clear at all that the methods developed in these papers could handle this case due to combinatorics losses. Namely, after decomposing the (renormalized) nonlinearity in \eqref{renorm} as 
${\ds\sum_{k\ge 0}\big(\frac{-\be}{1-\be\s_\al}\big)^k\NN_{2k+1}(v)}$ where the $(2k+1)$-linear forms $\NN_{2k+1}$ are simple (see \cite{DNY1,DNY2}), it seems that the estimates in \cite{DNY1} for $\al=d=2$ are of the form
\[\|\NN_{2k+1}(v)\|_{X^{s,b-1}}\le C^k(k!)\|v\|_{X^{s,b}}^{2k+1}\]
which is not summable.
\end{remark}

\begin{ackno}\rm
The author is very thankful to Laurent Thomann for pointing out this problem and for interesting discussions. Part of this work was supported by the German Research Foundation (DFG) through the CRC1283 ``Taming uncertainty and profiting from randomness and low regularity in analysis, stochastics and their applications".
\end{ackno}

\section{Study of the Gibbs measure}\label{SEC:Gibbs}
In this section, we introduce the basic tools needed for the proofs of Theorems~\ref{THM:Gibbs1},~\ref{THM:GWP1} and ~\ref{THM:GWP2}. We start with recalling some basic tools from analysis on manifolds in Subsection~\ref{SUBSEC:mani}, and establish some standard large deviation estimates for $\mu_\al$ as well as the proof of Theorem~\ref{THM:Gibbs1}~(i) in Subsection~\ref{SUBSEC:defoc}. Subsections~\ref{SUBSEC:var},~\ref{SUBSEC:focus} and~\ref{SUBSEC:LEM} are then devoted to the proof of Theorem~\ref{THM:Gibbs1}~(ii).

\subsection{Background on analysis on manifolds}\label{SUBSEC:mani}

Let $(\M,\gm)$ be a closed (compact, boundaryless) smooth Riemannian manifold of dimension $d\ge 1$. In local coordinates, the metric $\gm$ is given by a smooth function $x\mapsto(\gm_{j,k}(x))_{j,k}$ taking value in the set of positive symmetric definite matrices. In particular $(\gm_{j,k})$ is invertible and its inverse is denoted by 
\begin{align*}
\big(\gm^{j,k}(x) \big) = \big(\gm_{j,k}(x)\big)^{-1}.
\end{align*}
We also write
\begin{align*}
|\gm(x)| = \det \big(\gm_{j,k}(x)\big)>0.
\end{align*} 
The volume form $\Vol$ can then be written locally as
\begin{align*}
d\Vol(x) = |\gm(x)|^{\frac12}dx,
\end{align*}
and abusing notations we will write $dx$ in place of $d\Vol$.

The Laplace-Beltrami operator is given in local coordinates by
\begin{align*}
\Dlg f = |\gm|^{-\frac12}\dd_j(|\gm|^{\frac12}\gm^{j,k}\dd_k f\big)
\end{align*}
for any smooth function $f\in C^{\infty}(\M)$. Here summation on repeated indices is implicit.

It is well-known (see e.g. \cite[Chapter 2]{Jost}) that there exists a basis $\{\varphi_n\}_{n\ge 0}\subset C^{\infty}(\M)$ of $L^2(\M)$ consisting of eigenfunctions of $\Dlg$ associated with the eigenvalue $-\ld_n^2$, assumed to be arranged in increasing order: $0=\ld_0 <\ld_1\le\ld_2\le...$ and satisfying $\ld_n\to\infty$. Moreover, any $f\in\D'(\M)$ can be expanded (in the sense of distributions) as
\begin{align*}
f = \sum_{n\ge 0}\langle f,\varphi_n\rangle \varphi_n,
\end{align*}
where $\langle\cdot,\cdot\rangle$ denotes the usual duality pairing between $\D'(\M)$ and $\D(\M)$ which coincides with the $L^2(\M)$-inner product when $f\in L^2(\M)$. This allows to define, for any $s\in\R$ and $1\le p\le \infty$, the Sobolev spaces
\begin{align*}
W^{s,p}(\M,\gm) \deff\Big\{u\in\D'(\M),~\|u\|_{W^{s,p}(\M)} <\infty\Big\}
\end{align*}
where
\begin{align*}
\|u\|_{W^{s,p}(\M)}  \deff \big\|(1-\Dlg)^{\frac{s}2}u\big\|_{L^p(\M)} = \Big\|\sum_{n\ge 0}\jb{\ld_n}^s\langle u,\varphi_n\rangle_\gm \varphi_n\Big\|_{L^p(\M)}.
\end{align*}
When $p=2$ we write $H^s(\M)\deff W^{s,2}(\M)$.

To get some information on the localization of the eigenvalues $\ld_n$ and bounds on the eigenfunctions $\varphi_n$, it is usual to consider the spectral function, which is defined as 
\begin{align}\label{eig}
e(x,y,\Ld^2) \deff \sum_{\ld_n^2 \leq \Ld^2} \varphi_n(x)\varphi(y),
\end{align}
for $(x,y)\in\M\times\M$ and $\Ld\in\R$. We have the following behaviour of the spectral function on the diagonal (see \cite[Theorem 1.1]{Hormander}; note that here the $\ld_n$'s are the square roots of the eigenvalues of $-\Dlg$).
\begin{lemma} \label{LEM:loc eig} 
The following asymptotics hold uniformly in $x\in\M$ as $\Ld\to\infty$:
\begin{equation} \label{loc eig}
e(x,x,\Ld^2) = c_d \Ld^d + O(\Ld^{d-1}),
\end{equation}
for some constant $c_d>0$ only depending on the dimension.
\end{lemma}
This implies in particular Weyl's law \eqref{Weyl} after integrating \eqref{loc eig}. As a direct consequence we obtain the boundedness in average of the eigenfunctions $\varphi_n$.
\begin{corollary}\label{LEM:avg eig}
For any $A\in\R$, there exists $C>0$ such that for any $\Ld>0$ and any $x\in\M$, it holds
\begin{align*}
\sum_{\ld_n\in (\Ld,\Ld+1]}\frac{|\varphi_n(x)|^2}{\jb{\ld_n}^A} \le C\sum_{\ld_n\in (\Ld,\Ld+1]}\frac1{\jb{\ld_n}^A}.
\end{align*}
\end{corollary}
\begin{proof}
We just have to use Lemma \ref{LEM:loc eig} to bound
\begin{align*}
\sum_{\ld_n\in (\Ld,\Ld+1]}\frac{\varphi_n(x)^2}{\jb{\ld_n}^A} &\sim \jb{\Ld}^{-A}\Big[e\big(x,x,(\Ld+1)^2\big)-e(x,x,\Ld^2)\Big]\\
&\les \jb{\Ld}^{-A}\Ld \les \jb{\Ld}^{-A}\#\big\{k: \Ld<\ld_k \le \Ld+1\big\} \sim \sum_{\ld_n\in (\Ld,\Ld+1]}\jb{\ld_n}^{-A}.
\end{align*}
This proves Corollary~\ref{LEM:avg eig}.
\end{proof}

\subsection{Properties of the Gaussian and Gibbs measures}\label{SUBSEC:defoc}
In this subsection we state the basic properties of the Gaussian measure $\mu_\al$ \eqref{mu}, which allow us to prove Theorem~\ref{THM:Gibbs1}~(i).

\noi We start by recalling Khinchin's inequality.
\begin{lemma}\label{LEM:Khinchin}
Let $(a_n)\in\ell^2(\Z)$ and $g_n\sim\NN_\C(0,1)$ be standard iid complex Gaussian random variables on a probability space $(\O,\F,\PP)$. Then there exists $C>0$ such that for any $p\ge 2$ it holds
\begin{align*}
\big\|\sum_{n\in \Z}a_ng_n\big\|_{L^p(\O)} \le C\sqrt{p} \|a_n\|_{\ell^2}.
\end{align*}
\end{lemma}

\noi Next, we turn to the properties of the random variable $u_\al$ in \eqref{init}.
\begin{lemma}\label{LEM:init}
For $\al\in\R$, the series defining $u_\al$ in \eqref{init} converges in $L^p(\Omega;W^{s,\infty}(\M))$ and almost surely in $W^{s,\infty}(\M)$, for any $p\ge 1$ and $s<\frac{\al-d}2$. Moreover, for any $1\le r<+\infty$, there exists $c,C,\delta>0$ such that for any any $\al\in\R$, $R>0$ and any $1\le N_1\le N_2\le\infty$ we have the tail estimates
\begin{align}\label{tail}
\mu_\al(\|\Pi_{\le N_2}\Pi_{>N_1}u\|_{W^{s,r}}>R)\le Ce^{-c\|\s_{\al-2s,N_2}-\s_{\al-2s,N_1}\|_{L^\infty}^{-1} R^2}.
\end{align} 
\end{lemma}
\noi In particular, Lemma~\ref{LEM:init} and the trivial estimate
\begin{align*}
\int_\M e^{\be |u|^2}dx \le \Vol(\M)e^{\be\|u\|_{L^\infty}^2}
\end{align*} 
show that for $\al>d$, the potential $V_\be(u)$  in \eqref{V} is finite $\mu_\al$-a.s. and $V_\be(\P_Nu)\to V_\be(u)$ a.s. in $C(\M)$.
\begin{proof}
Let $\al\in\R$, $s<\frac{\al-d}2$, and $\eps>0$ such that $s+\eps<\frac{\al-d}2$. Note that it suffices to prove that the series in \eqref{init} converges in $L^p(\O;W^{s,\infty}(\M))$ for sufficiently large $p$. If $p\ge p_\eps > \frac{d}{\eps}$, we have by Sobolev embedding, Minkowki's inequality, Khinchin's inequality (Lemma~\ref{LEM:Khinchin}) and Corollary~\ref{LEM:avg eig}, that
\begin{align}
\E\big\|\Pi_{\le N_2}\Pi_{>N_1}u_\al\big\|_{W^{s,\infty}}^p &\le C\E\big\|\Pi_{\le N_2}\Pi_{>N_1}u_\al\big\|_{W^{s+\eps,p_\eps}}^p \le C\big\|\jb{\nabla}^{s+\eps}\Pi_{\le N_2}\Pi_{>N_1}u_\al\big\|_{L^{p_\eps}L^p_\omega}^p\notag\\
&\le C^p p^\frac{p}2\big\|\jb{\nabla}^{s+\eps}\Pi_{\le N_2}\Pi_{>N_1}u_\al\big\|_{L^{p_\eps}L^2_\omega}^p\notag\\
&\le \sup_x C^p p^\frac{p}2 \Big(\sum_{N_1<\ld_n\le N_2}\jb{\ld_n}^{2s+2\eps-\al}|\varphi_n(x)|^2\Big)^\frac{p}2\notag\\
& = C^p\|\s_{\al-2s-2\eps,N_2}-\s_{\al-2s-2\eps,N_1}\|_{L^\infty}^pp^\frac{p}2.\label{Lp}
\end{align}
This proves that $u_\al \in L^p(\O;W^{s,\infty}(\M))$. The tail estimate \eqref{tail} then follows from Chebychev's inequality and optimizing \eqref{Lp} by taking $$p\sim \log\big(c\|\s_{\al-2s-2\eps,N_2}-\s_{\al-2s-2\eps,N_1}\|_{L^\infty}^{-1}R^2\big).$$
\end{proof}

The results above are enough for the construction of the Gibbs measure in the defocusing case.

\begin{proof}[Proof of Theorem~\ref{THM:Gibbs1}~(i)]
Let us consider the defocusing case $\gamma>0$. Since $\al>d$, by Lemma~\ref{LEM:init}, $u_\al\in W^{s,\infty}(\M)$ a.s. for any $0\le s<\frac{\al-d}2$. In particular $u_\al\in C(\M)$ a.s. and so ${\displaystyle V_\be(u_\al) = \int_\M e^{\be |u_\al|^2}dx}$ is almost surely finite. Thus ${\displaystyle e^{-\gamma V_\be(u_\al)}}$ is almost surely positive, and since $V_\be(u_\al)\ge 0$ a.s. and $\gamma>0$ it also holds $e^{-\gamma V_\be(u)}\in L^{\infty}(\Omega)\subset L^1(\Omega)$. This shows that $\rho_{\al,\be}$ is a well-defined probability measure, absolutely continuous with respect to $\mu_\al$. As for the almost sure convergence of $V_\be(\Pi_{\le N}u)$, it follows from the mean value inequality through the estimate
\begin{align*}
\Big|V_\be(\Pi_{\le N}u)-V_\be(u)\Big| & \le \int_\M\Big|e^{\be|\Pi_{\le N}u|^2}-e^{\be|u|^2}\Big|dx \les \int_\M\sum_{k\ge 0}\frac{\be^k}{k!} \Big||\Pi_{\le N}u|^{2k}\Pi_{\le N}u-|u|^{2k}u\Big|dx\\
&\les \int_\M\sum_{k\ge 0}\frac{\be^k}{k!}(2k+1)C^k |\Pi_{\le N}u-u|\big(|\Pi_{\le N}u|^{2k}+|u|^{2k}\big)dx\\
&\les \|\Pi_{\le N}u-u\|_{L^2}e^{\be C(\|\Pi_{\le N}u\|_{L^\infty}^2+\|u\|_{L^\infty}^2)},
\end{align*}
combined with Lemma~\ref{LEM:init}, and that $\|\Pi_{\le N}u\|_{L^\infty}$ converges $\mu_\al$-a.s. and is thus $\mu_\al$-a.s. finite, uniformly in $N\in\N$. This last claim follows since $\Pi_{\le N}u(x)$ is a martingale in $N\in\N$ for all $x\in\M$, and so $\|\Pi_{\le N}u\|_{L^\infty}$ is a positive submartingale, uniformly integrable by \eqref{Lp}. Thus it converges $\mu_\al$-a.s.

The convergence of $\rho_{\al,\be,N}$ to $\rho_{\al,\be}$ in total variation thus follows from the convergence in probability of $V_\be(\Pi_{\le N}u_\al)$ to $V_\be(u_\al)$ along with the bound $e^{-\gamma V(\Pi_{\le N}u)}\le 1$, uniformly in $N\in\N$. See e.g. the discussion in \cite[Remark 3.8]{Tz}.
\end{proof}

\subsection{Variational formulation}
\label{SUBSEC:var}

In order to  prove Theorem~\ref{THM:Gibbs1}~(ii), 
we rely on a variational formula for the partition function
as in \cite{TW, OOT, OST}.
Let us first introduce some notations.
Let $W(t)$ be a cylindrical Brownian motion in $L^2(\M)$, given by
\begin{align}\label{W}
W(t) = \sum_{n \ge 0} B_n(t) \varphi,
\end{align}

\noi
where  
$\{B_n\}_{n \ge 0}$ is a sequence of mutually independent complex-valued\footnote{By convention, we normalize $B_n$ such that $\text{Var}(B_n(t)) = t$.} Brownian motions.
Then, define a centered Gaussian process $Y(t)$
by 
\begin{align}\label{Y}
Y(t) =  \jb{\nabla}^{-\frac\al2}W(t).
\end{align}

\noi
In particular we have $\Law(Y(1)) = \mu_\al$ where $\mu_\al$ is the Gaussian measure \eqref{mu}.

Next, let $\Ha$ denote the space of drifts, 
which are progressively measurable processes 
belonging to 
$L^2([0,1]; L^2(\T^d))$, $\PP$-almost surely. The following  Bou\'e-Dupuis formula \cite{BD} is proved in \cite[Theorem 7]{Ust}.

\begin{lemma}\label{LEM:var}
Let $\al>d$ and $Y$ be as in \eqref{Y}.
For any $0< s<\frac{\al-d}2$, let $V:W^{s,\infty}(\M) \to \R$ be measurable such that $\E\big[|V(Y(1))|^p\big] < \infty$
and $\E\big[e^{-qV(Y(1))}\big] < \infty$ for some $1 < p, q < \infty$ with $\frac 1p + \frac 1q = 1$.
Then, we have
\begin{align}\label{var}
- \log \E\Big[e^{-V(Y(1))}\Big]
= \inf_{\dr \in \mathbb H_a}
\E\bigg[ V(Y(1) + I(\dr)(1)) + \frac{1}{2} \int_0^1 \| \dr(t) \|_{L^2_x}^2 dt \bigg], 
\end{align}

\noi
where  $I(\dr)$ is  defined by 
\begin{align*}
 I(\dr)(t) = \int_0^t \jb{\nabla}^{-\frac\al2} \dr(t') dt'
\end{align*}

\noi
and the expectation $\E = \E_\PP$
is an 
expectation with respect to the underlying probability measure~$\PP$.

\end{lemma}

In the recent works constructing Gibbs measure based on the study of a similar variational formula \cite{BG,GOTW,ORSW2}, a slightly weaker version is used (see e.g. Proposition 4.4 in \cite{GOTW}) where the expectation in \eqref{var} is replaced by an expectation  with respect to a shifted measure depending on the drift. Since in our case we want to construct a drift depending on $Y$ (heuristically, we want $I(\dr)(1) = -Y(1)+$ perturbation), the formula \eqref{var} is better suited for our purpose.

\subsection{Proof of Theorem~\ref{THM:Gibbs1}~(ii)}
\label{SUBSEC:focus}

In this subsection, we present the proof of  Theorem~\ref{THM:Gibbs1}~(ii) regarding the non-normalizability of the Gibbs measure ${\displaystyle e^{-\gamma V_\be(u)}\ind_{\{\|u\|_{L^2}\le K\}}d\mu_\al(u)}$ with mass cut-off in the focusing case, for any $\gamma<0$, $\be>0$, $K>0$ and $\al>d$.

Note that it suffices to prove that
\begin{align*}
\lim_{L\to+\infty}\E_{\mu_\al}\Big[\exp\big( -\gamma \min(V_\be(u),L) \big) \ind_{\{\|u\|_{L^2} \le K\}} \Big] =  +\infty.
\end{align*}

Since
\begin{align}\label{ind}
\E_{\mu_\al}\Big[e^{-\gamma \min(V_\be(u),L)}\ind_{\{\|u\|_{L^2}\le K\}}\Big] &\ge \E_{\mu_\al}\Big[e^{-\gamma \min(V_\be(u),L)\ind_{\{\|u\|_{L^2}\le K\}}}\Big]-\mu_\al\big(\|u\|_{L^2}>K\big) \notag\\
&\ge \E_{\mu_\al}\Big[e^{-\gamma \min(V_\be(u),L)\ind_{\{\|u\|_{L^2}\le K\}}}\Big] -1,
\end{align}

\noi
the  divergence follows once we prove
\begin{align}\label{div}
\lim_{L\to+\infty}\E_{\mu_\al}\Big[e^{-\gamma \min(V_\be(u),L)\ind_{\{\|u\|_{L^2}\le K\}}}\Big] =  +\infty
\end{align}
for any $\al>d$, $\be>0$, $\gamma<0$, and $K>0$.

As for any $L>0$ we have $\min(V_\be(u),L)\in L^{\infty}(\mu_\al)$ and $\exp\big(-\gamma\min(V_\be(u),L)\big)\in L^{\infty}(\mu_\al)$, we can apply the  Bou\'e-Dupuis variational formula 
(Lemma~\ref{LEM:var}). Thus we have for any $L>0$
\begin{align}\label{var2}
&- \log \E_{\mu_\al}\Big[e^{-\gamma \min(V_\be(u),L)\ind_{\{\|u\|_{L^2}\le K\}}}\Big]\notag\\
&\qquad = \inf_{\dr \in \mathbb H_a} \E\bigg[ \gamma \min\big\{V_\be\big(Y(1) + I(\dr)(1)\big),L\big\}  \ind_{\{\| Y(1)+I(\dr)(1)\|_{L^2}\le K\}} + \frac 12 \int_0^1 \| \dr(t)\|_{L^2_x} ^2 dt \bigg],
\end{align}

\noi
where $Y(1)$ is as in \eqref{Y}.
Here,  $\E_\mu$ and $\E$ denote expectations
with respect to the Gaussian measure~$\mu_\al$ in \eqref{mu}
and the underlying probability measure $\PP$, respectively.
In the following, we show that the right-hand side 
of \eqref{var2} goes to $-\infty$ as $L\to+\infty$.
As mentioned above, the idea is to construct a drift $\dr$
such that 
$I(\dr)\approx - Y(1) + $ perturbation. In \cite{LRS}, in the case of the focusing nonlinear Schrödinger equation on the circle, corresponding to $d=1$, $\al=2$, $\gamma<0$ and $V_\be(u)$ replaced with the potential $|u|^p$, $p> 6$, the authors used concentrating solitons as a perturbation having bounded $L^2$ norm but exploding $L^p$ norm. Here, as in \cite{OOT,OST} we rely instead on properly scaled functions supported at high frequencies.

\begin{lemma}\label{LEM:fM}
Let   $\alpha > d$, and fix $x_0\in\M$. For $N\ge 1$, define the function
\begin{align}\label{fM}
f_N : x\in\M\mapsto N^{-\frac{d}2}\sum_{\ld_n\in (N,2N]}\varphi_n(x_0)\varphi_n(x).
\end{align} 
Then, there are constants $c_1,c_2,c_3>0$ depending only on $\M$ such that for any $N\gg 1$ we have
\begin{align}
\|f_N\|_{L^2}^2 &= c_1 + O( N^{-1}), \label{fM2} \\
\|f_N\|_{H^{\frac\al2}} &\les N^{\frac\al2}, \label{fMal} \\
\|f_N\|_{L^{\infty}}&\les N^{\frac{d}2}, \label{fMinf}
\end{align}
and for all $x\in\M$ with $d(x,x_0)\le c_2 N^{-1}$, it holds
\begin{align}\label{fMinf2}
f_N(x)\ge c_3 N^{\frac{d}2}.
\end{align}

\end{lemma}
 
We present the proof of Lemma~\ref{LEM:fM} in the next subsection. The next lemma, whose proof is also postponed to the next subsection, provides us with an appropriate approximation $Z_N$ to $Y$ in \eqref{Y}. One should think as $Z_N$ to be $\Pi_{\le N}Y$, this latter naive choice being not suitable since $I(\theta)\approx -\Pi_{\le N}Y(1)$ means $\dr \approx -\frac{d}{dt}\Pi_{\le N}Y \not\in \Ha$.
\begin{lemma} \label{LEM:Z}
Given $ N\gg 1$,  define $Z_N=\sum_{\ld_n\le N}Z_{N,n}\varphi_n$, where for $n\ge 0$, $Z_{N,n}$ solves the following  differential equation\textup{:}
\begin{align}
\begin{cases}
d Z_{N,n} = \jb{\ld_n}^{-\frac \al2} N^{\frac\al2} (\jb{\ld_n}^{-\frac\al2}B_n- Z_{N,n}) dt \\
Z_{N,n}|_{t = 0} =0, 
\end{cases}
\label{Z}
\end{align}

Then, $Z_N(t)$ is a centered Gaussian process in $L^2(\M)$ which satisfies
\begin{align}
&\E\int_0^1\big\|\frac{d}{dt}Z_N(t)\big\|_{H^s}^2dt \les N^\frac\al2+N^{d+2s}, \label{Z1}\\
&\E|Y(1,x)-Z_N(1,x)|^2\les N^{-\frac\al2}+N^{d-\al}\text{ uniformly in }x\in\M, \label{Z2}\\
&\E\|Y(1)-Z_N(1)\|_{L^{\infty}}^p\les p^{\frac{p}2}\big(N^{-\frac\al2}+N^{d+2\eps-\al}\big)^{\frac{p}2} \text{ for }p\gg 1\text{ and }0<\eps<\frac{\al-d}2, \label{Z3}
\end{align}
	
\noi
for any $N \gg 1$.

\end{lemma}

We are now ready to prove the non-normalizability of the focusing Gibbs measure.

\begin{proof}[Proof of Theorem~\ref{THM:Gibbs1}~(ii)]
In view of the previous discussion, fix $\al>d$, $\be>0$, $\gamma<0$ and $K>0$. For $N \gg 1$, let $f_N$ and $Z_N$ be given by Lemmas~\ref{LEM:fM} and~\ref{LEM:Z}, respectively. For the minimization problem \eqref{var2},
we take the drift 
\begin{align}\label{dr}
\dr_N (t)  = \jb{\nb}^{\frac \al2} \bigg( -\frac{d}{dt} Z_N(t) + \eta f_N \bigg)
\end{align}
for some constant $\eta=\eta(K)>0$ independent of $N$ to be chosen later, and
\begin{align}\label{Dr}
\Dr_N = I(\dr_N)(1) 
= \int_0^1 \jb{\nb}^{-\frac \al2} \dr(t) \, dt = - Z_N(1) + \eta f_N.
\end{align}
	
Then the right-hand side of \eqref{var2} is bounded by
\begin{align}\label{var3}
\E\bigg[ \gamma \min\big\{V_\be\big(Y(1) + \Dr_N\big),L\big\}  \ind_{\{\| Y(1)+\Dr_N\|_{L^2}\le K\}} + \frac 12 \int_0^1 \| \dr_N(t)\|_{L^2_x} ^2 dt \bigg].
\end{align}

Note that, as in \eqref{ind} we have for $L\ge 1$
\begin{align}
\E\Big[\min\big(V_\be(u),L\big)\ind_{\{\|u\|_{L^2}\le K\}}\Big] &= \int_\M\E\Big[\exp\big\{\min(\be|u(x)|^2,\log L)\big\}\ind_{\{\|u\|_{L^2}\le K\}}\Big]dx\label{ind2}\\
&\ge \int_\M\E\Big[\exp\big\{\min(\be|u(x)|^2,\log L)\ind_{\{\|u\|_{L^2}\le K\}}\big\}\Big]dx -\Vol(\M).\notag
\end{align}

Thus, using \eqref{ind2} and Jensen's inequality (with $\gamma<0$), we can bound
\begin{align}
\eqref{var3}&\le \gamma\int_\M\exp\Big[\E\big\{\min(\be|Y(1,x)+\Dr_N(x)|^2,\log L)\ind_{\{\|Y(1)+\Dr_N\|_{L^2}\le K\}}\big\}\Big]dx\notag\\
&\qquad+\Vol(\M)+\frac12\E\int_0^1\|\dr_N(t)\|_{L^2}^2dt.\label{var4}
\end{align}

In view of the definitions of $\dr_N$ \eqref{dr}, $Z_N$ \eqref{Z}, and $f_N$ \eqref{fM}, we have
\begin{align*}
\E\int_0^1\|\dr_N(t)\|_{L^2}^2dt &= \E\int_0^1\big\|-\frac{d}{dt}Z_N+\eta f_N\big\|_{H^{\frac\al2}}^2dt \notag\\
&\les \int_0^1\E\big\|\frac{d}{dt}Z_N\big\|_{H^\frac\al2}^2dt+\eta^2 \|f_N\|_{H^{\frac\al2}}^2.
\end{align*}
Using the bounds \eqref{fM2} and \eqref{Z1} of Lemmas~\ref{LEM:fM} and~\ref{LEM:Z}, we end up with
\begin{align}\label{dr2}
\E\int_0^1\|\dr_N(t)\|_{L^2}^2dt\les N^{d+\frac\al2}+N^{\al}.
\end{align}
As for the first term in the right-hand side of \eqref{var4}, we first note that
\begin{align*}
|Y(1,x)+\Dr_N(x)|^2&=|Y(1,x)-Z_N(1,x)+\eta f_N(x)|^2 \les N^d +|Y(1,x)-Z_N(1,x)|^2
\end{align*}
in view of \eqref{fMinf}. Thus provided that $L\gg N^d$ we have for any $0<\dl\ll1$
\begin{align}
&\E\big\{\min(\be|Y(1,x)+\Dr_N(x)|^2,\log L)\ind_{\{\|Y(1)+\Dr_N\|_{L^2}\le K\}}\big\}\notag\\
&\qquad\ge \E\big\{\be|Y(1,x)-Z_N(1,x)+\eta f_N(x)|^2\ind_{\{\|Y(1)-Z_N(1)+\eta f_N\|_{L^2}\le K,\|Y(1)-Z_N\|_{L^\infty}\les \log L\}}\big\}\notag\\
&\qquad\ge \be(1-\dl)\eta^2|f_N(x)|^2\PP\big(\|Y(1)-Z_N(1)+\eta f_N\|_{L^2}\le K,\|Y(1)-Z_N\|_{L^\infty}\les \log L\big)\notag\\
&\qquad\qquad-C(\dl)\E|Y(1,x)-Z_N(1,x)|^2,\label{var5}
\end{align}
for some constant $C(\dl)>0$.

Assume that for any $K>0$, there exists $\eta=\eta(K)>0$ such that
\begin{align}\label{prob}
\PP\big(\|Y(1)-Z_N(1)+\eta f_N\|_{L^2}\le K,\|Y(1)-Z_N\|_{L^\infty}\les \log L\big) \ge \frac12,
\end{align}
\noi
uniformly in in $L\gg N^d\gg 1$.

Then we can put \eqref{dr2}-\eqref{var5}-\eqref{prob} into \eqref{var4} to get (recall $\gamma<0$)
\begin{align}
\eqref{var3}&\le \gamma \int_\M\exp\Big[\frac{\be(1-\dl)\eta^2}2|f_N(x)|^2-C(\dl)\E|Y(1,x)-Z_N(1,x)|^2)\Big]dx+O(N^{d+\frac\al2}+N^{\al})\notag\\
&\le \gamma \int_{\{d(x,x_0)\le c_\M N^{-1}\}}\exp\Big[\frac{\be(1-\dl)\eta^2}2|f_N(x)|^2-C(\dl)\E|Y(1,x)-Z_N(1,x)|^2)\Big]dx\notag\\
&\qquad\qquad\qquad\qquad+O(N^{d+\frac\al2}+N^{\al})\notag\\
&\les \gamma N^{-d}\exp\Big[\frac{\be(1-\dl)\eta^2}4N^{\frac{d}2}-o(1))\Big]dx+O(N^{d+\frac\al2}+N^{\al}),\label{var6}
\end{align}
where $x_0$ is as in Lemma~\ref{LEM:fM}, and the last estimate comes from \eqref{fMinf2} in Lemma~\ref{LEM:fM} and \eqref{Z2} in Lemma~\ref{LEM:Z}.

Combining \eqref{var2}-\eqref{var3}-\eqref{var6}, we infer that there are constants $C_1,C_2,C_3>0$ such that
\begin{align*}
-\log\E_{\mu_\al}\Big[\exp\Big(\min\big(V_\be(u),L\big)\ind_{\{\|u\|_{L^2}\le K\}}\Big)\Big] \le \gamma C_1N^{-d}e^{C_2N^{\frac{d}2}}+C_3(N^{d+\frac\al2}+N^{\al}),
\end{align*}
for any $L\gg N^d\gg 1$. Therefore,
\begin{align*}
&\lim_{L\to+\infty}\E_{\mu_\al}\Big[-\gamma\min(V_\be(u),L)\ind_{\{\|u\|_{L^2}\le K\}}\Big]\\
&\qquad \ge \exp\Big(-\gamma C_1N^{-d}e^{C_2N^{\frac{d}2}}-C_3(N^{d+\frac\al2}+N^{\al})\Big)\to +\infty
\end{align*}
as $N\to+\infty$ since $\gamma<0$. This finally proves \eqref{div} by assuming \eqref{prob}.

To conclude, we prove that \eqref{prob} holds for any $K>0$ and $L\gg N^d\gg1$, provided that $\eta(K)$ is large enough. Indeed, it holds
\begin{align*}
&\PP\big(\|Y(1)-Z_N(1)+\eta f_N\|_{L^2}> K\text{ or }\|Y(1)-Z_N\|_{L^\infty}\gg \log L\big)\\
&\qquad \le \PP\big(\|Y(1)-Z_N(1)+\eta f_N\|_{L^2}> K)+\PP\big(\|Y(1)-Z_N(1)\|_{L^\infty}\gg \log L\big)\\
&\qquad \le K^{-2}\E\|Y(1)-Z_N(1)+\eta f_N\|_{L^2}+C(\log L)^{-p}\E\|Y(1)-Z_N(1)\|_{L^\infty}^p\intertext{by using Chebychev's inequality with some large $p\gg1$. Using then \eqref{Z1}-\eqref{Z3} in Lemma~\ref{LEM:Z} and \eqref{fM2} in Lemma~\ref{LEM:fM}, we get the bound}
&\qquad \le CK^{-2}\big(N^{-\frac\al2}+N^{d-\al} + c_\M\eta\big) + (\log L)^{-p}p^{\frac{p}2}(N^{-\frac\al2}+N^{d+2\eps-\al})^\frac{p}2,
\end{align*}
for any $p\gg1$ (independent of $L,N$) and $0<\eps<\frac{\al-d}2$, and some constant $C>0$. Since $L\gg N^d\gg 1$, taking $\eta\ll K^2$ ensures that
\begin{align*}
CK^{-2}\big(N^{-\frac\al2}+N^{d-\al} + c_\M\eta\big) + (\log L)^{-p}p^{\frac{p}2}(N^{-\frac\al2}+N^{d+2\eps-\al})^\frac{p}2<\frac12.
\end{align*} 
\noi
This finally proves \eqref{prob}.

\end{proof}

\subsection{Some technical lemmas}
\label{SUBSEC:LEM}

In this subsection, we present the proofs of 
Lemmas~\ref{LEM:fM} and \ref{LEM:Z}.

We start with the approximation lemma (Lemma~\ref{LEM:Z}).

\begin{proof}[Proof of Lemma \ref{LEM:Z}]
For $\ld_n\le N$, let
\begin{align}
X_n(t)=\jb{\ld_n}^{-\frac\al2}B_n(t)- Z_{N,n}(t), 
\label{X} 
\end{align}
Then,  from \eqref{Z}, 
we see that $X_n(t)$ satisfies 
the following stochastic differential equation:
\begin{align}\label{X2}
\begin{cases}
dX_n(t)=-\jb{\ld_n}^{-\frac \al 2}N^\frac \al2 X_n(t) dt +\jb{\ld_n}^{-\frac \al2}dB_n(t)\\
X_n(0)=0
\end{cases}
\end{align}	

\noi
for $\ld_n\le N$.
The solution to this stochastic differential equation is thus given by
\begin{align}
X_n(t)=\jb{\ld_n}^{-\frac\al2}\int_0^t e^{-\jb{\ld_n}^{-\frac \al 2}N^{\frac \al2}(t-s)}dB_n(s).
\label{X3}
\end{align}

\noi
for $\ld_n\le N$. 

To show \eqref{Z1}, observe that $\frac{d}{dt}Z_N(t)=N^{\frac\al2}\jb{\nabla}^{-\frac\al2}X_N$ due to \eqref{Z} and \eqref{X}, where
\begin{align*}
X_N(t)=\Pi_{\le N}Y(t)-Z_N(t)=\sum_{\ld_n\le N}X_{n}(t)\varphi_n.
\end{align*} 
Hence, from \eqref{X2}, the independence of $\{B_n \}_{n \ge 0}$, and Ito's isometry with \eqref{X3}, it holds for any $s\in\R$
\begin{align}
\int_0^1\E \Big\|\frac{d}{dt}Z_N(t)\Big\|_{H^{s}}^2dt & = N^\al\int_0^1\E\big\|X_N(t)\big\|_{H^{s-\frac\al2}}^2dt\\
&\le N^\al \sum_{\ld_n \le N}\jb{\ld_n}^{2s-\al}\int_0^1\int_0^te^{-2\jb{\ld_n}^{-\frac\al2}N^\frac\al2(t-t')}dt'dt\\
& \les N^\al\sum_{\ld_n\le N}\jb{\ld_n}^{2s-\al}N^{-\frac\al2}\jb{\ld_n}^\frac\al2\\
&\les N^\frac\al2\max\big(1,N^{d+2s-\frac\al2}\big),
\end{align}

\noi
for any $N\gg 1$, where the last estimate follows from Weyl's law \eqref{Weyl}. This is enough for \eqref{Z1}.

For \eqref{Z2}, we estimate similarly as above
\begin{align*}
\E\big|Y(1,x)-Z_N(1,x)\big|^2 &\les \E|X_N(1,x)|^2 +\E|(1-\Pi_{\le N}Y)(1,x)|^2 \\
& \les \sum_{\ld_n\le N}|\varphi_n(x)|^2\jb{\ld_n}^{-\al}\int_0^1e^{-2\jb{\ld_n}^{-\frac\al2}N^{\frac\al2}(1-t)}dt + \sum_{\ld_n>N}|\varphi_n(x)|^2\jb{\ld_n}^{-\al}\\
&\les N^{-\frac\al2}\sum_{m\le N}\sum_{\ld_n\in(m,m+1]}\frac{|\varphi_n(x)|^2}{\jb{\ld_n}^{\frac\al2}} + \sum_{m>N}\sum_{\ld_n\in(m,m+1]}\frac{|\varphi_n(x)|^2}{\jb{\ld_n}^\al}\intertext{Applying then Corollary~\ref{LEM:avg eig}, we can continue with}
&\les N^{-\frac\al2}\sum_{m\le N}\sum_{\ld_n\in(m,m+1]}\jb{\ld_n}^{-\frac\al2} + \sum_{m>N}\sum_{\ld_n\in(m,m+1]}\jb{\ld_n}^{-\al}\\
&\les N^{-\frac\al2}\max\big(1,N^{d-\frac\al2}\big)+ N^{d-\al},
\end{align*}
where the last estimate follows again from Weyl's law \eqref{Weyl}. This proves \eqref{Z2}.

Finally, to prove \eqref{Z3}, we use a standard combination of Sobolev and Minkowski inequalities along with Khinchin's inequality (Lemma~\ref{LEM:Khinchin}): thus for any $0<\eps<\frac{\al-d}2$ and $p\ge p_\eps\gg1$, we have
\begin{align*}
\E\big\|Y(1)-Z_N(1)\big\|_{L^{\infty}}^p &\les \E\big\|Y(1)-Z_N(1)\big\|_{W^{\eps,p_\eps}}^p \les \int_\M\E\big|\jb{\nabla}^\eps(Y(1)-Z_N(1))\big|^pdx\\
&\les \sup_{x\in\M}p^\frac{p}2\Big(\E\big|\jb{\nabla}^{\eps}(Y(1,x)-Z_N(1,x))\big|^2\Big)^\frac{p}2\intertext{A computation similar to the one made above for \eqref{Z2} finally yields}
&\les p^\frac{p}2\Big(N^{-\frac\al2}\sum_{\ld_n\le N}\jb{\ld_n}^{2\eps-\frac\al2} + \sum_{\ld_n>N}\jb{\ld_n}^{2\eps-\al}\Big)^\frac{p}2\\
&\les p^\frac{p}2\Big(N^{-\frac\al2}+N^{d+2\eps-\al}\Big)^\frac{p}2.
\end{align*}
This proves \eqref{Z3}.

\end{proof}	

To conclude this subsection, we present the proof of Lemma~\ref{LEM:fM}.

\begin{proof}[Proof of Lemma~\ref{LEM:fM}]
Recall that $f_N$ is given by \eqref{fM}. In particular a straightforward computation using Corollary~\ref{LEM:avg eig} and Weyl's law \eqref{Weyl} yields
\begin{align*}
\|f_N\|_{H^{\frac\al2}}^2 = N^{-d}\sum_{\ld\in(N,2N]}\jb{\ld_n}^\al|\varphi_n(x_0)|^2 \les N^{-d}\sum_{\ld_n\sim N}\jb{\ld_n}^{\al}\les N^\al,
\end{align*}
and, using Cauchy-Schwarz inequality and repeating the previous computation,
\begin{align*}
\|f_N\|_{L^\infty}&\le \sup_{x\in\M}N^{-\frac{d}2}\sum_{\ld\in(N,2N]}|\varphi_n(x)\varphi_n(x_0)|\\
& \le \sup_{x\in\M}N^{-\frac{d}2}\Big(\sum_{\ld_n\in(N,2N]}|\varphi_n(x)|^2\Big)^\frac12\Big(\sum_{\ld_n\in(N,2N]}|\varphi_n(x_0)|^2\Big)^\frac12\\
&\les N^{\frac{d}2}.
\end{align*}
This proves \eqref{fMal} and \eqref{fMinf}. 

For \eqref{fM2}, we observe that
\begin{align*}
\|f_N\|_{L^2}^2 &= N^{-d}\sum_{\ld_n\in (N,2N]}|\varphi_n(x_0)|^2 = N^{-d}\big(e(x_0,x_0,(2N)^2)-e(x_0,x_0,N^2)\big)\intertext{where $e(x,y,\Ld^2)$ is the spectral function of $\Dlg$, defined in \eqref{eig} above. In view of Lemma~\ref{LEM:loc eig}, we can thus continue with}
&=N^{-d}\big(c_d(2N)^d-c_dN^d+O(N^{d-1})\big) = (2^d-1)c_d+O(N^{-1}).
\end{align*}
This proves \eqref{fM2}.

To show the remaining bound \eqref{fMinf2}, observe as above that
\begin{align}\label{fMspec}
f_N(x) = N^{-\frac{d}2}\big(e(x,x_0,(2N)^2)-e(x,x_0,N^2)\big).
\end{align}
In particular with Lemma~\ref{LEM:loc eig} this implies
\begin{align}\label{fMspec2}
f_N(x_0) = N^{-\frac{d}2}\big(c_d(2N)^d-c_dN^d + O(N^{d-1})\big) = (2^d-1)c_dN^{\frac{d}2} + O(N^{\frac{d}2-1}).
\end{align}
In order to control the difference $f_N(x)-f_N(x_0)$, we recall the following refined version of Lemma~\ref{LEM:loc eig} proved in \cite[Theorem 4.4]{Hormander}.
\begin{lemma}\label{LEM:spec}
There are some constants $r,R,C>0$ such that for any $x\in B_r(x_0)=\{d(x,x_0)< r\}$, it holds
\begin{align}\label{spec}
\Big|e(x,x_0,\Ld^2)-(2\pi)^{-d}\int_{\{|\xi|\le\Ld\}}e^{i\psi(\kk(x),\xi)}d\xi\Big|\le C\jb{\Ld}^{d-1},
\end{align}
where $\kk : B_r(x_0) \to B_{R}(0)\subset \R^d$ are normal coordinates centred at $x_0$, and the phase $\psi\in C^\infty(B_R(0)\times\R^d)$ satisfies 
\begin{align}
\psi(z,\xi) &= z\cdot\xi +O(|z|^2|\xi|) \text{ as }z\to 0.\label{psi}
\end{align}
\end{lemma}
Thus, putting \eqref{spec} into \eqref{fMspec} and using the mean value inequality and the property \eqref{psi} of $\psi$, we get for any $d(x,x_0)\ll N^{-1}$:
\begin{align}\label{fMspec3}
\big|f_N(x)-f_N(x_0)\big|&\les N^{-\frac{d}2}\Big|\int_{\{N<|\xi|\le 2N\}}\big[e^{i\psi(\kk(x),\xi)}-1\big]d\xi\Big|+O(N^{\frac{d}2-1})\notag\\
&\les N^{-\frac{d}2}\int_{\{|\xi|\sim N\}}|\psi(\kk(x),\xi)|d\xi+O(N^{\frac{d}2-1})\notag\\
& \les N^{-\frac{d}2}d(x,x_0)\int_{\{|\xi|\sim N\}}|\xi|d\xi +O(N^{\frac{d}2-1}) \notag\\
&\les N^{\frac{d}2+1}d(x,x_0)+O(N^{\frac{d}2-1})\ll N^{\frac{d}2}.
\end{align}
Thus \eqref{fMinf2} follows from \eqref{fMspec2} and \eqref{fMspec3}.

\end{proof}

\section{Weak invariance of the Gibbs measure for lower dispersion}\label{SEC:weak}
In this section, we give the proof of Theorem~\ref{THM:GWP1}~(i). We mainly follow \cite{BTT1,OT}. In this section we thus fix $\gamma>0$, $\al>d$ and $0<s<\frac{\al-d}2$.

\subsection{Preliminary estimates}
We start by establishing the following moment bound for the potential $V_\be$.
\begin{lemma}\label{LEM:V}
Let $\be>0$. Then the following hold:\\
\textup{(i)} If $0<\be<\|\s_\al\|_{L^{\infty}}^{-1}$ and $1\le p < (\be\|\s_\al\|_{L^{\infty}})^{-1}$, then $V_\be\in L^p(\mu_\al)$. In particular $V_\be(\Pi_{\le N}u)\to V_\be(u)$ in $L^p(\mu_\al)$ as $N\to\infty$.\\
\textup{(ii)} If $\be>\|\s_\al\|_{L^{\infty}}$ then $V_\be\not\in L^1(\mu_\al)$.
\end{lemma}
\begin{proof}
We start by proving (i). We thus assume that $0<\be<\|\s_\al\|_{L^{\infty}}^{-1}$ and fix $1\le p < (\be\|\s_\al\|_{L^{\infty}})^{-1}$. In view of the almost sure convergence of $V_\be(\Pi_{\le N}u)$ (Lemma~\ref{LEM:init}) and Fatou's lemma, it suffices to show that
\begin{align}\label{V1}
\sup_{N\in\N}\E_{\mu_\al}\big[V_\be(\Pi_{\le N}u)^p\big] <+\infty.
\end{align}
\noi
Recall that $\Pi_{\le N}$ is the (sharp) spectral projector onto the frequencies $\{\ld_n\le N\}$. For any $N\in\N$, we then write 
\begin{align}\label{dN}
d_N=\#\Ld_N=\#\{n,~\ld_n\le N\},
\end{align}
 and using that $\mu_\al$ is the law of $u_\al$ in \eqref{init}, we estimate using Jensen's inequality in $x$ with the compactness of $\M$ and Fubini-Tonelli's theorem:
\begin{align}
\E_{\mu_\al}\big[V_\be(\Pi_{\le N}u)^p\big] &= \int_\Omega \Big[\int_\M\exp\Big(\be\big|\sum_{\ld_n\le N}\frac{g_n}{\jb{\ld_n}^\frac{\al}2}\varphi_n(x)\big|^2\Big)dx\Big]^pd\PP\notag\\
&\les \int_\M\int_\Omega\exp\Big(p\be\big|\sum_{\ld_n\le N}\frac{g_n}{\jb{\ld_n}^\frac{\al}2}\varphi_n(x)\big|^2\Big)d\PP dx\notag\\
& \les \int_\M\int_\Omega\exp\Big(p\be\sum_{n,m\in\Ld_N}\frac{g_n\cj g_m}{\jb{\ld_n}^\frac{\al}2\jb{\ld_m}^\frac{\al}2}\varphi_n(x)\varphi_m(x)\Big)d\PP dx\notag\intertext{Recalling that $g_n$ are iid standard complex Gaussian random variables, $g_n = \frac{h_n+i\ell_n}{\sqrt{2}}$ with $h_n,\ell_n$ iid standard real-valued Gaussian random variables, we can continue with}
& \les \int_\M\int_{\R^{2d_N}}\exp\Big(p\be\sum_{n,m=1}^{d_N}\frac{y_ny_m+z_nz_m +i(z_ny_m-y_nz_m)}{2\jb{\ld_n}^\frac{\al}2\jb{\ld_m}^\frac{\al}2}\varphi_n(x)\varphi_m(x)\Big)\notag\\
&\qquad\qquad\times\prod_{n=1}^{d_N}e^{-\frac12y_n^2}e^{-\frac12z_n^2}\frac{dy_ndz_n}{2\pi}dx\notag\\
&\les\int_\M\bigg(\int_{\R^{d_N}}\exp\Big(p\be\sum_{n,m=1}^{d_N}\frac{y_ny_m}{2\jb{\ld_n}^\frac{\al}2\jb{\ld_m}^\frac{\al}2}\varphi_n(x)\varphi_m(x)-\frac{y_n^2}2\Big)\prod_{n=1}^{d_N}\frac{dy_n}{\sqrt{2\pi}}\bigg)^2dx,\label{V2}
\end{align}
where we used the symmetry in $n$ and $m$. Defining the symmetric matrix 
\begin{align}\label{Ax}
A(x) = \Big(\frac{\varphi_n(x)\varphi_m(x)}{\jb{\ld_n}^\frac{\al}2\jb{\ld_m}^\frac{\al}2}\Big)_{n,m=1}^{d_N}\in\S_{d_N}(\R),
\end{align}
 we can write the inner integral as
\begin{align}
&\int_{\R^{d_N}}\exp\Big(p\be\sum_{n,m=1}^{d_N}\frac{y_ny_m}{2\jb{\ld_n}^\frac{\al}2\jb{\ld_m}^\frac{\al}2}\varphi_n(x)\varphi_m(x)-\frac{y_n^2}2\Big)\prod_{n=1}^{d_N}\frac{dy_n}{\sqrt{2\pi}}\notag\\
&\qquad= \int_{\R^{d_N}}e^{-\frac12y\cdot(I_{d_N}-p\be A(x))y}\frac{dy}{(2\pi)^{\frac{d_N}2}}.\label{V3}
\end{align}
Then, since $\mathrm{rank} A(x)\le 1$ in view of \eqref{Ax}, we have that the only non zero eigenvalue of $A(x)$ is
$$\mathrm{tr} A(x) = \sum_{\ld_n\le N}\frac{|\varphi_n(x)|^2}{\jb{\ld_n}^\al} = \s_{\al,N}(x).$$
Thus, since $A(x)$ is symmetric, we find that the integral in \eqref{V3} is finite if and only if $I_{d_N}-p\be A(x)$ is positive definite, which corresponds to 
\begin{align}\label{V4}
p\be \s_{\al,N}(x)<1.
\end{align}
In that case,
\begin{align}\label{V5}
\int_{\R^{d_N}}e^{-\frac12y\cdot(I_{d_N}-p\be A(x))y}\frac{dy}{(2\pi)^{\frac{d_N}2}} &= \mathrm{det}\big(I_{d_N}-p\be A(x)\big)^{-\frac12} = (1-p\be\s_{\al,N}(x))^{-\frac12}.
\end{align}
Since for any $x\in\M$ and $N\in\N$ it holds $\s_{\al,N}(x)\le \|\s_\al\|_{L^{\infty}}$, under the conditions $0<\be<\|\s_\al\|_{L^{\infty}}^{-1}$ and $1\le p < (\be\|\s_\al\|_{L^\infty})^{-1}$, we see that \eqref{V4} holds uniformly in $x\in\M$ and $N\in\N$, and putting \eqref{V2}-\eqref{V3}-\eqref{V5} together and using the compactness of $\M$ we get
\begin{align*}
\sup_{N\in\N}\E_{\mu_\al}\big[V_\be(\Pi_{\le N}u)^p\big] &\les \sup_{N\in\N}\frac1{1-p\be\|\s_{\al,N}\|_{L^\infty}} = \frac1{1-p\be\|\s_\al\|_{L^\infty}} <+\infty.
\end{align*}
This proves \eqref{V1}, hence Lemma~\ref{LEM:V}~(i). This even provides the slightly stronger bound
\begin{align}\label{V6}
\sup_{N\in\N}\Big\|\E_{\mu_\al}\big[e^{p\be|\Pi_{\le N}u|^2}\big]\Big\|_{L^\infty}<+\infty.
\end{align}

As for (ii) in the case $\be >\|\s_\al\|_{L^\infty}^{-1}$, since $\s_\al\in C(\M)$, there exists $x_0\in\M$ and $r>0$ such that $\be > \s_{\al,N}(x)$ for all $N\in\N$ and $x\in B_r(x_0)$. Thus the computation above shows that
\begin{align*}
\E_{\mu_\al}\big[V_\be(\Pi_{\le N}u)\big] \ge \Vol(B_r(x_0))\cdot \infty = +\infty,
\end{align*}
for all $N\in\N$. This proves Lemma~\ref{LEM:V}~(ii).
\end{proof}

In the proof of Theorem~\ref{THM:GWP1}~(i), we will also make use of the following result, which is a straightforward adaptation of \cite[Lemma 3.3]{BTT1}.

\begin{lemma}\label{LEM:BTT1}
Let $T > 0$ and $1< p,q < \infty$, $\theta_0\in [0,1]$ be such that $(1-\theta_0)(\frac1q-1)+\frac{\theta_0}p=0$.
Suppose that 
$u \in L^p_T H^{s_1}$ and $\dt u \in L^q_T H^{s_2}$
for some $s_2 \leq s_1$.
Then, for $ \eps > (1-\theta_0)(s_1 - s_2)$, we have 
\begin{align}\label{interp1}
 \| u \|_{L^\infty_TH^{s_1}} \les \| u \|_{L^p_T H^{s_1+\eps}}^{\theta_0}
\| u \|_{W^{1, q}_T H^{s_2}}^{ 1-\theta_0}.
\end{align}

\noi
Moreover, writing $\theta_1 = \frac{s_1-s_2}{s_1-s_2+\eps}\in[0,1]$, it holds for all $t_1, t_2 \in [-T, T]$
\begin{align}\label{interp2}
 \| u(t_2) - u(t_1)  \|_{H^{s_1}} \les |t_2 - t_1|^{(1-\theta_1)(1-\frac1q)}  \| u \|_{L^p_T H^{s_1+\eps}}^{\theta_1}
\| u \|_{W^{1, q}_T H^{s_2}}^{ 1-\theta_1}.
\end{align}

\end{lemma}

\begin{proof}
As in \cite[Lemmas 3.2 and 3.3]{BTT1}, we perform a Littlewood-Paley decomposition $$u = \sum_{M\in 2^{\N}}\psi_M(-\Dl)u=\sum_{M\in 2^\N}u_M,$$ where $\psi\in C^{\infty}_0(\R)$, $\supp\psi\subset [\frac14,4]$, $\psi\equiv 1$ on $[-\frac12,2]$, and $\psi_M(-\Dlg)=\psi(-M^{-2}\Dlg)$ is defined via the functional calculus:
\begin{align*}
u_M = \sum_{n\ge 0}\psi(M^{-2}\ld_n^2)\langle u,\varphi_n\rangle_{L^2}\varphi_n.
\end{align*} 
Then we have
\begin{align*}
\|u\|_{L^\infty_TH^{s_1-\eps}}& \les \sum_{M\in 2^\N}L^{s_1-\eps}\|u_M\|_{L^\infty_TL^2} = \sum_{M\in 2^\N}M^{s_1-\eps}\sup_{\|v\|_{L^2}\le 1}\sup_{t\in[0,T]}\langle u_M(t),v\rangle_{L^2}.
\end{align*}
Then we apply the Gagliardo-Nirenberg interpolation inequality (see for example Theorem 12.83 in \cite{Leoni})
\begin{align}\label{GaNi}
\|w\|_{L^\infty_T}\les \|w\|_{L^p_T}^{\theta_0}\|w\|_{W^{1,q}_T}^{1-\theta_0},
\end{align}
which holds for any $1\le p,q\le \infty$ and $\theta_0\in [0,1]$ such that $(1-\theta_0)(\frac1q-1)+\frac\theta{p}=0$, to the function $w(t)=\langle \psi_M(-\Dlg)u(t),v\rangle_{L^2}$. Thus
\begin{align*}
\sup_{t\in[0,T]}\langle u_L(t),v\rangle_{L^2}& \les \big\|\langle u_M(t),v\rangle_{L^2}\big\|_{L^p_T}^{\theta_0}\big\|\langle u_M(t),v\rangle_{L^2}\big\|_{W^{1,q}_T}^{1-\theta_0}\\
&\les \|u_M\|_{L^p_TL^2}^{\theta_0}\|u_M\|_{W^{1,q}_tL^2}^{1-\theta_0}.
\end{align*}
This yields
\begin{align*}
\|u\|_{L^\infty_TH^{s_1-\eps}}& \les \sum_{M\in 2^\N}M^{s_1-\eps}\|u_M\|_{L^q_TL^2}^{\theta_0}\|u_M\|_{W^{1,p}_tL^2}^{1-\theta_0}\\
& \les \sum_{M\in 2^\N}M^{s_1-\eps-\theta_0 s_1 -(1-\theta_0)s_2}\|u_M\|_{L^p_TH^{s_1}}^{\theta_0}\|u_M\|_{W^{1,q}_tH^{s_2}}^{1-\theta_0}\\
&\les \sum_{M\in 2^\N}M^{s_1-\eps-\theta_0 s_1 -(1-\theta_0)s_2}\|u\|_{L^p_TH^{s_1}}^{\theta_0}\|u\|_{W^{1,q}_tH^{s_2}}^{1-\theta_0}\\
&\les \|u\|_{L^p_TH^{s_1}}^{\theta_0}\|u\|_{W^{1,q}_tH^{s_2}}^{1-\theta_0}
\end{align*}
 since $\eps>(1-\theta_0)(s_1-s_2)$. This proves \eqref{interp1}. The proof of \eqref{interp2} follows similarly as in \cite[Lemma 3.3]{BTT1}: combine the interpolation inequality
 \begin{align*}
 \|u(t_1)-u(t_2)\|_{H^{s_1}}&\les \|u(t_1)-u(t_2)\|_{H^{s_1+\eps}}^{\theta_1}\|u(t_1)-u(t_2)\|_{H^{s_2}}^{1-\theta_1}
 \end{align*}
with the bound
 \begin{align*}
 \|u(t_1)-u(t_2)\|_{H^{s_2}} = \Big\|\int_{t_1}^{t_2}\dt u dt\Big\|_{H^{s_2}} \les |t_1-t_2|^{1-\frac1q}\|u\|_{W^{1,q}_TH^{s_2}}
 \end{align*}
Combining the two previous estimates finally proves \eqref{interp2}.
 \end{proof}

\subsection{From the Gibbs measure to measures on paths}\label{SUBSEC:5.1} 

The general strategy of \cite{BTT1,OT} to get weak invariance of the Gibbs measure is to look for stationary solutions with law $\rho_{\al,\be}$ to \eqref{expNLS} as limits of solutions to \eqref{expNLSN} having law $\rho_{\al,\be ,N}$. Thus we start by proving global well-posedness for the truncated equation \eqref{expNLSN}
\begin{align*}
i\dt u_N + (-\Dlg)^{\frac{\al}2}u_N =2\gamma \be\Pi_{\le N}\Big[e^{\be |\Pi_{\le N} u|^2}\Pi_{\le N}u\Big].
\end{align*}

\begin{lemma}\label{LEM:global}
Let  $N \in \N$.
Then, 
the truncated exponential NLS~\eqref{expNLSN}
is globally well-posed in $H^s(\M)$.
Moreover, the truncated Gibbs measure $\rho_{\al,\be,N}$  
is invariant under the dynamics of 
\eqref{expNLSN}.

\end{lemma}
\begin{proof}
We write \eqref{expNLSN} under the Duhamel formulation
\begin{align}\label{DuhamelN}
u_N(t) = e^{it(-\Dlg)^\frac\al2}u_N(0)-2i\gamma\be \int_0^te^{i(t-t')(-\Dlg)^\frac\al2}\Pi_{\le N}\Big[e^{\be |\Pi_{\le N} u|^2}\Pi_{\le N}u\Big](t')dt'.
\end{align}
Write $\G_N(u_N)$ for the right-hand side of \eqref{DuhamelN}. Then for $T>0$ and $u\in C([-T,T];L^2(\M))$, using the compactness of $\M$, it holds
\begin{align*}
\big\|\G_N(u)\big\|_{C_TL^2} &\le \|u(0)\|_{L^2}+\Big\|\Pi_{\le N}\Big[e^{\be |\Pi_{\le N} u|^2}\Pi_{\le N}u\Big]\Big\|_{L^1_TL^2}\\
&\le \|u(0)\|_{H^s}+CT\Big\|\Pi_{\le N}\Big[e^{\be |\Pi_{\le N} u|^2}\Pi_{\le N}u\Big]\Big\|_{L^\infty_TL^2}\\
&\le \|u(0)\|_{H^s}+CTe^{\be \|\Pi_{\le N} u\|_{L^\infty_TL^\infty}^2}\big\|\Pi_{\le N}u\big\|_{L^\infty_TL^2}\\
&\le \|u(0)\|_{H^s}+CTe^{\be N^{d}\|u\|_{C_TL^2}^2}\|u\|_{C_TL^2}.
\end{align*}
Here we used the Bernstein type inequality $$\|\Pi_{\le N}u\|_{L^\infty}\les N^\frac{d}2 \|u\|_{L^2},$$ which follows from writing ${\displaystyle \Pi_{\le N}u = \int_\M e(x,y,N^2)u(y)dy}$ with $e(x,y,\Ld)$ the spectral function \eqref{spec}, and from the estimate
\begin{align*}
\|e(x,y,N^2)\|_{L^\infty_xL^2_y} = \sup_{x\in\M}\Big(\sum_{\ld_n\le N}|\varphi_n(x)|^2\Big)^\frac12 = \sup_{x\in\M}e(x,x,N^2)^\frac12 \les N^\frac{d}2
\end{align*}
which is a consequence of Lemma~\ref{LEM:loc eig}.

Thus for $T=T_N(R)\sim e^{-\be N^{d}R^2}$, $\G_N$ maps the ball $B_{R,T}$ of radius $R=2\|u_0\|$ of $C([-T,T];L^2(\M))$ in itself. Similarly, using the mean value theorem, it holds for any $u_1,u_2\in B_{R,T}$
\begin{align*}
&\big\|\G_N(u_1)-\G_N(u_2)\big\|_{C_TL^2}\\
&\le \|u_1(0)-u_2(0)\|_{L^2} + CT\Big\|\Pi_{\le N}\Big[e^{\be |\Pi_{\le N} u_1|^2}\Pi_{\le N}u_1\Big]-\Pi_{\le N}\Big[e^{\be |\Pi_{\le N} u_2|^2}\Pi_{\le N}u_2\Big]\Big\|_{L^\infty_TL^2}\\
&\le \|u_1(0)-u_2(0)\|_{L^2} + CT\Big\|e^{\be |\Pi_{\le N} u_1|^2}\Pi_{\le N}(u_1-u_2)\Big\|_{L^\infty_TL^2}\\
&\qquad+CT\Big\|\big|\Pi_{\le N}(u_1-u_2)\big|\big(|\Pi_{\le N}u_1|+|\Pi_{\le N}u_2|\big)\big(e^{\be |\Pi_{\le N} u_1|^2}+e^{\be |\Pi_{\le N} u_2|^2}\big)\Pi_{\le N}u_2\Big\|_{L^\infty_TL^2}\\
&\le \|u_1(0)-u_2(0)\|_{L^2} + CT\|u_1-u_2\|_{C_TL^2}e^{\be CN^d\|u_1\|_{C_TL^2}^2}\\
&\qquad+CTN^{2d}\|u_1-u_2\|_{C_TL^2}\big(\|u_1\|_{C_TL^2}+\|u_2\|_{C_TL^2}\big)\\
&\qquad\qquad\qquad\qquad\times\big(e^{\be CN^d\|u_1\|_{C_TL^2}^2}+e^{\be CN^d\|u_2\|_{C_TL^2}^2}\big)\|u_2\|_{C_TL^2}.
\end{align*}
This shows that for $T_N\sim N^{-2d}R^{-2}e^{-\be CN^dR^2}$, $\G_N$ is also a contraction on $B_{R,T_N}$. Thus by Banach fixed point theorem, there is a solution $u_N$ to \eqref{DuhamelN}, unique in $B_{R,T_N}$. Standard iteration argument shows that uniqueness holds in the whole of $C([-T_N,T_N];L^2(\M))$. Moreover,
\begin{align*}
\frac{d}{dt}\|u_N(t)\|_{L^2}^2&=2\langle u_N(t),\dt u_N(t)\rangle \\
&= 2\Re\Big\{\int_\M \cj u_N\cdot i\Big((-\Dlg)^\frac\al2 u_N(t) -2\gamma\be \Pi_{\le N}\big[e^{\be |\Pi_{\le N} u|^2}\Pi_{\le N}u\big]\Big)\Big\}dx\\
&=2\int_\M\Re i\Big\{\big|(-\Dlg)^\frac\al2u_N(t)\big|^2-2\gamma\be|\Pi_{\le N}u|^2e^{\be|\Pi_{\le N}u|^2}\Big\}dx = 0,
\end{align*}
after integrating by parts. The conservation of the $L^2$ norm of $u_N$ thus shows that \eqref{expNLSN} is actually globally well-posed in $C(\R;L^2(\M))$. At last, since $\Pi_{\le N}^2=\Pi_{\le N}$, in view of the equation \eqref{expNLSN} it holds $$u_N = \Pi_{\le N}u_N + (1-\Pi_{\le N})e^{it(-\Dlg)^\frac\al2}u_N(0),$$ which shows that actually $u_N\in C(\R;H^s(\M))$ if $u_N(0)\in H^s(\M)$. 

In the following we write $\Phi_N : u_0\in H^s(\M)\mapsto u_N \in C(\R;H^s(\M))$ the global flow of the truncated equation \eqref{expNLSN} on $H^s(\M)$, and for $t \in \R$ we use $\Phi_N(t): H^s(\M) \to H^s(\M)$ to denote the map defined by $\Phi_N(t)(\phi) = \big(\Phi_N(\phi)\big)(t)$. From the previous remark we have 
\begin{align*}
\Phi_N(t) = \Pi_{\le N}\Phi_N(t)+(1-\Pi_{\le N})e^{it(-\Dlg)^\frac\al2}.
\end{align*}
We also have
\begin{align*}
\rho_{\al,\be,N} = \ZZ_{\al,\be,N}^{-1}e^{-\gamma V_\be(\Pi_{\le N}u)}(\Pi_{\le N})_\#\mu_\al\otimes(1-\Pi_{\le N})_\#\mu_\al
\end{align*}
by definition of $\rho_{\al,\be,N}$ \eqref{GibbsM}. Note that the linear flow $e^{it(-\Dlg)^\frac\al2}$ preserves the Gaussian measure $\mu_\al$, since $e^{it(-\Dlg)^\frac\al2}$ acts on Fourier modes by multiplication by a complex number of modulus one, and this operation does not change the law of the standard complex-valued Gaussian random variables. As for the nonlinear part $\Pi_{\le N}\Phi_N(t)$, let $\Ld_N$ and $d_N$ be as in \eqref{dN}. Then writing $$\Pi_{\le N}\Phi_N(t)=\sum_{n\in\Ld_N}U_n(t)\varphi_n,$$ we have that $\{U_n\}_{n\in\Ld_N}$ solves the $d_N$-dimensional system of ODEs
\begin{align}\label{ODE}
\frac{d}{dt}U_n(t) = i\ld_n^\al U_n(t) -i2\gamma\be \Big\langle e^{\be|\sum_{m\in\Ld_n}U_m\varphi_m|^2}\big(\sum_{m\in\Ld_N}U_m\varphi_m\big),\varphi_n\Big\rangle_{L^2},~~\text{for all }n\in\Ld_N.
\end{align}
The system \eqref{ODE} is Hamiltonian, with Hamiltonian given by $\H_{\al,\be}(\sum_{n\in\Ld_n}U_n\varphi_n)$, where the energy $\H_{\al,\be}$ is as in \eqref{hamiltonian}. By Liouville's theorem, the Lebesgue measure ${\displaystyle \prod_{n=0}^{d_N-1}dU_n}$ is invariant by the flow of \eqref{ODE}. The invariance of the Hamiltonian and of the mass $\J(\sum_{n\in\Ld_N}U_n\varphi_n)$ with $\J$ in \eqref{mass} then ensures that the finite dimensional measure
\begin{align*}
\ZZ_{\al,\be,N}^{-1}e^{-(\J+\H_{\al,\be})(\sum_{n\in\Ld_N}U_n\varphi_n)}\Big(\prod_{n=0}^{d_N-1}dU_n\Big)
\end{align*}
is invariant under \eqref{ODE}. This proves the invariance of $\ZZ_{\al,\be,N}^{-1}e^{-\gamma V_\be(\Pi_{\le N}u)}(\Pi_{\le N})_\#\mu_\al$ under $\Pi_{\le N}\Phi_N(t)$, hence that of $\rho_{\al,\be,N}$ under $\Phi_N(t)$ by the previous discussion.
\end{proof}

As in \cite{OT}, we endow $ C(\R; H^s(\M))$ with the compact-open topology, and from the local Lipschitz continuity of $\Phi_N(\cdot)$ provided by the fixed point argument, it follows that $\Phi_N$ is continuous from $H^s (\M)$ into $C(\R; H^s(\M))  $.

Next we extend $\rho_{\al,\be,N}$ on $H^s$ to a probability measure $\nu_N$
on $C(\R; H^s(\M))  $ by setting
\[ \nu_N = \rho_{\al,\be,N} \circ \Phi_N^{-1}.\]

\noi
Namely, $\nu_N$ is the induced probability measure of $\rho_{\al,\be,N}$
under the map $\Phi_N$.
In particular, we have
\begin{align}
 \int_{C(\R;H^s(\M))}  F(u) d\nu_N (u) = \int_{H^s} F(\Phi_N(\phi)) d \rho_{\al,\be,N}(\phi)
\label{Y1}
 \end{align}

\noi
for any measurable function $F :C(\R; H^s(\M))\to  \R$.

\subsection{Tightness of the measures $\nu_N$}
\label{SUBSEC:5.2}

In the following, we prove that the sequence
$\{\nu_N\}_{N\in \N}$ of probability measures on $C(\R; H^s(\M))$ has a convergent subsequence. This will follow from Prokhorov's theorem (see for example \cite{Bass}) with the following proposition.

\begin{proposition}\label{PROP:tight}
The family $\{ \nu_N\}_{N \in \N}$ of probability measures on $C(\R; H^s(\M))$ is tight.
\end{proposition}
To establish Proposition~\ref{PROP:tight}, we will mainly follow \cite{OT} (see also \cite{BTT1}). Thus we start with the following bounds.

\begin{lemma}\label{LEM:bd1}
Assume that $0<\be <\|\s_\al\|_{L^\infty}^{-1}$. Let  $p \ge 1$ and $1\le q <(\be\|\s_\al\|_{L^\infty})^{-1}$.
Then, there exists $C_p,C_q > 0$ such that 
\begin{align}
\big\| \| u\|_{L^p_T H^s} \big\|_{L^p(\nu_N)} & \leq C_p T^\frac{1}{p}, \label{Y2}\\
\big\| \| u\|_{W^{1, q}_T H^{s-\al}} \big\|_{L^q(\nu_N)} & \leq C_qT^\frac{1}{q}, 
\label{Y3}
\end{align}

\noi
uniformly in $N \in \N$.
	
\end{lemma}

\begin{proof}

By Fubini's theorem, the invariance of $\rho_{\al,\be,N}$ under $\Phi_N$ (Lemma~\ref{LEM:global}), 
and that $e^{-\gamma V_\be}\in L^\infty(\mu_\al)$ (recall $\gamma>0$), we have 
 \begin{align}
\big\| \| u\|_{L^p_T H^s} \big\|_{L^p(\nu_N)}  
& = \big\| \| \Phi_N(t) (\phi)  \|_{L^p_T H^s}\big\|_{L^p(\rho_{\al,\be,N})}  = \big\| \| \Phi_N(t) (\phi)  \|_{L^p(\rho_{\al,\be,N}) H^s}\big\|_{L^p_T}  \notag \\
& = (2T)^{\frac{1}{p}} \| \phi \|_{L^p(\rho_{\al,\be,N}) H^s}\leq (2T)^\frac{1}{p}\|e^{-\gamma V_\be(\Pi_{\le N}\phi)} \|_{L^{\infty}(\mu_\al)} \| \phi \|_{L^{p}(\mu_\al) H^s}\label{Y4}\\
& \le C(2T)^{\frac{1}{p}} \| \phi \|_{L^{p}(\mu_\al) H^s}.\notag
\end{align}

\noi
Then,~\eqref{Y2}
follows from~\eqref{Y4} with \eqref{Lp} in the proof of Lemma~\ref{LEM:init}.

From the truncated equation~\eqref{expNLSN} we also have
\begin{align}
\big\| \| \dt u\|_{L^q_T H^{s-\al}} \big\|_{L^q(\nu_N)} 
\les
\big\| \| u\|_{L^q_T H^{s}} \big\|_{L^q(\nu_N)} 
+ \bigg\| \Big\|\Pi_{\le N}\Big[e^{\be |\Pi_{\le N} u|^2}\Pi_{\le N}u\Big]\Big\|_{L^q_T H^{s-\al}} \bigg\|_{L^q(\nu_N)} .
\label{Y5}
\end{align}

\noi
The first term is estimated by~\eqref{Y2}.
Proceeding as in~\eqref{Y4}, we have for the second term
 \begin{align*}
&\bigg\| \Big\|\Pi_{\le N}\Big[e^{\be |\Pi_{\le N} u|^2}\Pi_{\le N}u\Big]\Big\|_{L^q_T H^{s-\al}} \bigg\|_{L^q(\nu_N)}\\
&\qquad \les T^\frac{1}{q} \Big\|\Pi_{\le N}\Big[e^{\be |\Pi_{\le N} \phi|^2}\Pi_{\le N}\phi\Big]\Big\|_{L^q(\mu_\al) H^{s-\al}}\intertext{Since $\al>d$ and $s<\frac{\al-d}2$ and since $\M$ is compact it holds $H^{s-\al}(\M)\hookrightarrow L^1(\M)$, so that with Jensen's inequality in $x$ by compactness of $\M$ and Young's inequality we can continue with}
&\qquad\les T^\frac1q \Big\|\int_\M e^{\be|\Pi_{\le N}\phi|^2}|\Pi_{\le N}\phi|dx\Big\|_{L^q(\mu_\al)}\\
&\qquad\les (CT)^\frac1q \Big(\E_{\mu_\al}\Big[\int_\M e^{q\be|\Pi_{\le N}u_\al|^2}|\Pi_{\le N}u_\al|^qdx\Big]\Big)^\frac1q\\
&\qquad\les (CT)^\frac1q \Big(\E_{\mu_\al}\big[\|\Pi_{\le N}u_\al\|_{L^\infty}^qV_{q\be}(\Pi_{\le N}u_\al)\big]\Big)^\frac1q\\
&\qquad\les (CT)^\frac1q \Big(\E_{\mu_\al}\big[\|\Pi_{\le N}u_\al\|_{L^\infty}^\frac{q_2q}{q_2-q}\big]+\E_{\mu_\al}\big[V_{q_2\be}(\Pi_{\le N}u_\al)\big]\Big)^\frac1q,
\end{align*}
where $q<q_2<(\be\|\s_\al\|_{L^\infty})^{-1}$. This together with Lemmas~\ref{LEM:init} and~\ref{LEM:V}~(i) proves~\eqref{Y3}.
\end{proof}

We can finally move on to the proof of Proposition~\ref{PROP:tight}.

\begin{proof}[Proof of Proposition~\ref{PROP:tight}]
Let $0<s < s_1 < \frac{\al-d}2$.
Recall that for $a \in (0, 1)$, the Lipschitz space $C^a_TH^{s_1} = C^a([-T, T]; H^{s_1}(\M))$ is defined by the norm
\[ \| u \|_{C^a_T H^{s_1}} = \sup_{\substack{t_1, t_2 \in [-T, T]\\t_1 \ne t_2}}
\frac{\| u(t_1) - u(t_2) \|_{H^{s_1}}}{|t_1 - t_2|^a} + \|u \|_{L^\infty_T H^{s_1}},
\]
\noi
and $C^a_T H^{s_1}$ is compactly embedded in $C_T H^{s}$ for each $T>0$.

Let $p\gg 1$ and $1\le q<(\be\|\s_\al\|_{L^\infty})^{-1}$. Take  ${\displaystyle \theta_0=\frac{1-\frac1q}{\frac1p+1-\frac1q}}$, $\eps=(1-\theta_0)\al >(1-\theta_0)\frac\al2$ and $\theta_1 = \frac{\al}{\al+\eps}$ as in Lemma~\ref{LEM:BTT1}. Note that $\theta_0\to 1-$ as $q\to 1+$, thus we can take $q$ close enough to 1 so that $s_1+\eps<\frac{\al-d}2$. Then applying \eqref{interp1}-\eqref{interp2} in Lemma~\ref{LEM:BTT1} and Young's inequality, we have for $0<a<(1-\theta_1)(1-\frac1q)$
\begin{align}
\| u \|_{C^a_T H^{s_1}}
&\les \|u \|_{L^p_TH^{s_2}}^{1-\theta_0} \|u \|_{W^{1, q}_TH^{s_2 - \al}}^{\theta_0}\notag\\
&\qquad\qquad + \sup_{\substack{t_1,t_2\in[-T,T]\\t_1\neq t_2}}|t_1-t_2|^{(1-\theta_1)(1-\frac1q)-a}\|u\|_{L^\infty_TH^{s_1+\eps}}^{\theta_1}\|u\|_{W^{1,q}_TH^{s_1-\al}}^{1-\theta_1}\notag\\
&\les \|u \|_{L^p_TH^{s_2}}+ T^{1-\frac1q-\frac{a}{1-\theta_1}}\|u \|_{W^{1, q}_TH^{s_2 - \al}}.
\label{Y6}
\end{align}
Then, it follows from~\eqref{Y2} and 
\eqref{Y3} in Lemma~\ref{LEM:bd1} that 
\begin{align}
\big\|\| u \|_{C^\al_T H^{s_1}}\big\|_{L^p(\nu_N)}
\leq C_p T^\frac{1}{p}+C_qT^{1-\frac{a}{1-\theta_1}} \le C_{p,q}T^{1-\frac{a}{1-\theta_1}}
\label{Y7}
\end{align}
	
From there, we can conclude as in \cite{OT} by letting $T_j = 2^j$ for $j\in\N$ and defining
\[ K_\dl = \big\{ u \in C(\R;H^s(\M)):\, \| u \|_{C^a_{T_j} H^{s_1}} \leq c_0 \dl^{-1} T_j 
\text{ for all }
j \in \N \big\}\]
for $\dl>0$. Markov's inequality and \eqref{Y7} ensure that
\[ \nu_N(K_\dl^c) 
\leq c^{-1}_0 C_1 \dl T_j^{-1}
\big\|\| u \|_{C^a_{T_j} H^{s_1}}\big\|_{L^p(\nu_N)}
\leq c^{-1}_0 C_p \dl \sum_{j = 1}^\infty T_j^{-\frac{a}{1-\theta_1}}
\le c^{-1}_0 C_p c \dl < \dl\]
by taking $c_0$ large enough (recall $a>0$ and $\theta_1\in (0,1)$). The compactness of $K_\dl$ then follows from the same argument as in \cite{OT}.
\end{proof}

\subsection{Proof of Theorem~\ref{THM:GWP1}~(i)}\label{SUBSEC:5.3}
We finally move on to the proof of Theorem~\ref{THM:GWP1}~(i). We will construct the random variable $u$ in Theorem~\ref{THM:GWP1}~(i) as a limit of solutions $(\wt u_{N_j})$ to the truncated equation \eqref{expNLSN}. For this purpose we apply Proposition~\ref{PROP:tight} with Prokhorov's theorem. This provides us with a subsequence
$\nu_{N_j}$ which converges weakly to some probability measure $\nu $
on $C(\R; H^s(\M))$ for any $s< \frac{\al-d}2$.
By Skorokhod's theorem (see \cite{Bass}), there is a new probability space $(\wt \O, \F, \wt \PP)$
and a sequence of new random variables $\wt {u^{N_j}}$ with the same distribution $\nu_{N_j}$ as $u^{N_j}$, 
such that $\wt {u^{N_j}}$ converges almost surely to $u$. 

Moreover, they have the following pointwise (in time) distribution.

\begin{lemma}\label{LEM:end1}
Let $\wt u_{N_j}$ and $u$ be as above.
Then, we have
\begin{align*}
\L\big(\wt {u^{N_j}}(t)\big)  = \rho_{\al,\be,N_j}
\quad \text{and} \quad
\L(u(t)) = \rho_{\al,\be} 
\end{align*}

\noi
for any $t \in \R$.
\end{lemma}

\begin{proof}
This follows exactly as in \cite[Proposition 4.11]{BTT1} or \cite[Lemma 5.8]{OT} so we omit details.
\end{proof}

To conclude the proof of Theorem~\ref{THM:GWP1} it remains to show that the random variable $u$ is indeed a global-in-time distributional solution to the exponential NLS \eqref{expNLS}.

\begin{lemma}\label{LEM:end2}
Let $\wt {u^{N_j}}$ and $u$ be as above.
Then, 
$\wt {u^{N_j}}$ and $u$
are global-in-time distributional 
solutions to 
the truncated exponential NLS~\eqref{expNLSN} for each $j \in \N$
and
to the exponential NLS~\eqref{expNLS}, respectively.
\end{lemma}

\begin{proof}
The only difference compared to the proof of Theorem 1.5 in \cite{OT} is in checking the almost sure convergence of 
\begin{align*}
F_{N_j}\big(\wt{u^{N_j}}\big)=\Pi_{\le N_j}\Big[e^{\be|\Pi_{\le N_j}\wt u^{N_j}|^2}\Pi_{\le N_j}\wt u^{N_j}\Big]
\end{align*}
to $F(u) = e^{\be|u|^2}u$.

In the following, we simply write $F_j = F_{N_j}$
and $u_j =  \wt{u^{N_j}}$.
For any $M \in \N$, we decompose
\begin{align}
F_j (u_j) - F(u) 
& = \big( F_j(u_j) - F(u_j) \big) +\big( F (u_j) - F_M(u_j) \big)\notag \\
& \hphantom{XXX}
+  \big( F_M (u_j) - F_M(u) \big)
+  \big( F_M (u) - F(u) \big).
\label{Y13}
\end{align}
	
\noi
Then, 
for each fixed $M \geq 1$, 
it follows from 
 the almost sure convergence
of  $\wt {u^{N_j}}$  to $u$ in $C(\R;H^s(\M))$
and the continuity of $F_M$
that the third term on the right-hand side of~\eqref{Y13}
converges to 0 in  $C(\R;H^s(\M))$ as $j \to \infty$, almost surely with respect to $\wt \PP$.

Fix $T>0$ and $1<p<(\be\|\s_\al\|_{L^\infty})^{-1}$, and let $\frac1{p'}=1-\frac1p$. Then similar computations as in~\eqref{Y4} yield
\begin{align}
\big\|\|F (u_j) - F_M(u_j)\|_{L^1_T H^{s-\al}}\big\|_{L^1(\nu_{N_j})}
& = 
\big\|\|F (\Phi_{N_j} \phi ) - F_M(\Phi_{N_j} \phi)\|_{L^1(\rho_{\al,\be,N_j}) H^{s-\al}}\big\|_{L^1_T} \notag \\
& = 2T
\|F ( \phi ) - F_M( \phi)\|_{L^1(\rho_{\al,\be,N_j}) H^{s-\al}} \notag \\
& \les T \|F ( \phi ) - F_M( \phi)\|_{L^1(\mu_\al) H^{s-\al}},
\label{Y14}
\end{align}
and with \eqref{V4},
\begin{align}
&\|F ( \phi ) - F_M( \phi)\|_{L^1(\mu_\al) H^{s-\al}}\\
& = \Big\|\Pi_M\Big[e^{\be|\Pi_M\phi|^2}\Pi_M\phi\Big]-e^{\be|\phi|^2}\phi\Big\|_{L^p(\mu_\al) H^{s-\al}}\notag\\
& \les \Big\|(1-\Pi_M)\Big[e^{\be|\phi|^2}\phi\Big]\Big\|_{L^1(\mu_\al) H^{s-\al}}+\Big\|e^{\be|\Pi_M\phi|^2}(1-\Pi_M)\phi\Big\|_{L^1(\mu_\al) L^1}\notag\\
&\qquad\qquad+\Big\|\Big[e^{\be|\Pi_M\phi|^2}-e^{\be|\phi|^2}\Big]\phi\Big\|_{L^1(\mu_\al) L^1}\notag\\
& \les M^{-\eps}\Big\|e^{\be|\phi|^2}\phi\Big\|_{L^1(\mu_\al) L^1}+\Big\|e^{\be|\Pi_M\phi|^2}\Big\|_{L^p(\mu_\al)L^p}\big\|(1-\Pi_M)\phi\big\|_{L^{p'}(\mu_\al)L^{p'}}\notag\\
&\qquad+\Big\|(1-\Pi_M)\phi(|\Pi_M\phi|+|\phi|)\big(e^{\be|\Pi_M\phi|^2}+e^{\be|\phi|^2}\big)\phi\Big\|_{L^1(\mu_\al) L^1}\notag\\
& \les M^{-\eps}\big\|V_{\be}(\phi)\big\|_{L^1(\mu_\al)}+M^{-\eps}\Big\|\E_{\mu_\al}\big[e^{p\be|\Pi_M\phi|^2}\big]\Big\|_{L^\infty}\notag\\
&\qquad+M^{-\eps}\big(\|\Pi_M\phi\|_{L^{3p'}(\mu_\al)L^{3p'}}+\|\phi\|_{L^{3p'}(\mu_\al)L^{3p'}}\big)\Big(\Big\|\E_{\mu_\al}\big[e^{p\be|\Pi_M\phi|^2}\big]\Big\|_{L^\infty}+\Big\|\E_{\mu_\al}\big[e^{p\be|\phi|^2}\big]\Big\|_{L^\infty}\Big)\notag\\
&\les M^{-\eps}\label{Y15}
\end{align}
\noi
uniformly in $j\in\N$, for some small $\eps > 0$ such that $s-\al+\eps<-d$ (recall that $\al>d$ and $s<\frac{\al-d}2$) so that $H^{s-\al+\eps}(\M)\hookrightarrow L^1(\M)$ by compactness of $\M$. We also used H\"older's inequality, the mean value theorem, and that for any $q\ge 1$
\begin{align*}
\big\|(1-\Pi_M)\phi\big\|_{L^{q}(\mu_\al)L^{q}}&\le C\sqrt{q}\big\|(1-\Pi_M)\phi\big\|_{L^{q}L^{2}(\mu_\al)} =C\sqrt{q}\big\|\s_\al-\s_{M,\al}\big\|_{L^{q}} \les M^{-\eps}
\end{align*}
by Minkowski inequality and Khinchin's inequality (Lemma~\ref{LEM:Khinchin}), since $0<\eps<\al-d$, the estimate on the tail $\s_\al-\s_{M,\al}$ being ensured again by Corollary~\ref{LEM:avg eig}.

Thus the first term in the right-hand side of \eqref{Y13} is controlled by \eqref{Y14}-\eqref{Y15}. The first and last terms are estimated similarly. leading to the convergence of $F_j(u_j)$ towards $F(u)$ in $L^1(\nu_{N_j};L^1([-T;T];H^{s-\al}(\M)))$. Thus after passing to a subsequence, $F_j(u_j)$ converges $\wt\PP$-almost surely to $F(u)$ in $L^1([-T, T]; H^{s-\al}(\M))$. The rest of the proof is identical to \cite[Lemma 5.9]{OT}.
\end{proof}

Finally, combining Lemma~\ref{LEM:end1} and~\ref{LEM:end2} through the same argument as in \cite{OT} finishes the proof of Theorem~\ref{THM:GWP1}.

\section{Strong invariance of the Gibbs measure for higher dispersion}\label{SEC:GWP1}
In this section, we give the proof of Theorem~\ref{THM:GWP1}~(ii). We thus fix $\gamma>0$, $\be>0$, $\al>2d$, and $\frac{d}2<s<\frac{\al-d}2$. We first establish a general local well-posedness result in Subsection~\ref{SUBSEC:LWP1} and then implement Bourgain's invariant measure argument in the context of \eqref{expNLS} in Subsection~\ref{SUBSEC:GWP}.

\subsection{Local well-posedness for high dispersion}\label{SUBSEC:LWP1}
We start with a deterministic local well-posedness result for the exponential NLS \eqref{expNLS}.
\begin{proposition}\label{PROP:LWP1}
There exists $c>0$ and $C_1=C_1(s,d)>0$ independent of $\al$ such that for any $\be>0$ and any $R\ge \theta>0$, if we set $T= c\theta (|\gamma|\be R)^{-1}e^{-C_1\be R^2}$, then for any $u_0\in H^s(\M)$ with $\|u_0\|_{H^s}\le R$, the Cauchy problem for \eqref{expNLS} with initial data $u_0$ has a unique solution $u$ in $C([-T;T];H^s(\M))$, which satisfies
\begin{align}\label{bdPhi}
\|u\|_{C_TH^s}\le R+\theta.
\end{align}
Moreover, the flow map $\Phi:u_0\in \{v_0\in H^s(\M),~\|v_0\|_{H^s}\le R\}\mapsto u\in C([-T,T];H^s(\M))$ is Lipschitz continuous.
\end{proposition}
The proof of Proposition~\ref{PROP:LWP1} is straightforward from the following fractional Leibniz rule, which implies the algebra property of $H^s(\M)$ for $s>\frac{d}2$ due to Sobolev inequality.
\begin{lemma}\label{LEM:Leibniz}
Let $s>\frac{d}2$. Then there exists $C_0>0$ such that for any $u,v\in H^s(\M)$, it holds
\begin{align*}
\|uv\|_{H^s}\le C_0\big(\|u\|_{H^s}\|v\|_{L^\infty}+\|u\|_{L^\infty}\|v\|_{H^s}\big).
\end{align*}
\end{lemma}
\begin{proof}
Lemma~\ref{LEM:Leibniz} follows for example from the equivalence of $H^s_\mathrm{loc}(\M)$ and $H^s_\mathrm{loc}(\R^d)$ (see e.g. \cite[Proposition 2.5]{ORTz}) and the corresponding fractional Leibniz rule in $\R^d$ (Corollary 2.86 in \cite{BCD}).
\end{proof}
\begin{proof}[Proof of Proposition~\ref{PROP:LWP1}]
Let $R\ge \theta>0$ and $u_0\in H^s(\M)$ satisfying $\|u_0\|_{H^s}\le R$. We start by writing the Duhamel formula for \eqref{expNLS} as
\begin{align}\label{Duhamel}
u(t) = e^{it(-\Dlg)^\frac\al2}u_0 -i2\gamma\be\int_0^te^{i(t-t')(-\Dlg)^\frac\al2}\big[e^{\be|u|^2}u\big](t')dt'.
\end{align}
Write $\Gamma(u)$ for the right-hand side of \eqref{Duhamel}, and for $T=c\theta|\gamma|\be R^{-1}e^{-\be 2C_0R^2}$ with $C_0=C_0(s,d)$ (independent of $\al$) as in Lemma~\ref{LEM:Leibniz} and $c$ to be chosen later, let $B_{R,T}$ be the ball of radius $R$ in $C([-T,T];H^s(\M))$. Then for any $u\in B_{R,T}$ it holds
\begin{align*}
\big\|\G(u)\big\|_{C_TH^s}&\le \|u_0\|_{H^s} + C|\gamma|\be\big\|e^{\be|u|^2}u\big\|_{L^1_TH^s} \le R + C|\gamma|\be T\sum_{k\ge 0}\frac{\be^k}{k!}\big\||u|^{2k}u\big\|_{C_TH^s}\\
&\le R + C|\gamma|\be T\sum_{k\ge 0}\frac{\be^k}{k!}C_0^k\|u\|_{C_TH^s}^{2k+1} \le R + C_2|\gamma|\be Te^{\be 4C_0R^2}R,
\end{align*}
where we used $(2k+1)$ times the fractional Leibniz rule (Lemma~\ref{LEM:Leibniz}) to estimate the $k$-th term in the sum, and that $\theta\le R$. Thus from $T=c\theta |\gamma|\be R^{-1}e^{-\be 8C_0R^2}$, we have that if $0<c<C_1$, $\G$ maps $B_{R,T}$ into itself. Proceeding similarly as above and with the mean value theorem as in the proof of Lemma~\ref{LEM:global}, we have for any $u_1,u_2\in B_{R,T}$
\begin{align}\label{diff}
&\big\|\G(u_1)-\G(u_2)\big\|_{C_TH^s}\notag\\
&\le \|u_1(0)-u_2(0)\|_{H^s} + C|\gamma|\be T\sum_{k\ge 0}\frac{\be^k}{k!}\big\||u_1|^{2k}u_1-|u_2|^{2k}u_2\|_{C_TH^s}\notag\\
&\le \|u_1(0)-u_2(0)\|_{H^s} + C|\gamma|\be T\sum_{k\ge 0}\frac{\be^k}{k!}(2k+1)C_0^k\|u_1-u_2\|_{C_TH^s}\big(\|u_1\|_{C_TH^s}^{2k}+\|u_2\|_{C_TH^s}^{2k}\big)\notag\\
&\le \|u_1(0)-u_2(0)\|_{H^s} + C_3|\gamma|\be T\|u_1-u_2\|_{C_TH^s}e^{\be 8C_0R^2}.
\end{align}
This shows that with $T$ as above and $c<\min(C_2,\frac{C_3}2)$, $\G$ is also a contraction on $B_{R,T}$. This proves existence and uniqueness of the solution in $B_{R,T}$. For two solutions $u_1,u_2\in C([-T,T];H^s(\M))$, iterating the previous estimate on $[mT_0,(m+1)T_0]$ for $T_0\sim e^{-\be C \min(\|u_1\|_{C_{T}H^s},\|u_2\|_{C_{T}H^s})^2}$ and $|m|\les \frac{T}{T_0}$ shows that $u_1\equiv u_2$ on $[-T,T]$, hence (unconditional) uniqueness in $C([-T,T];H^s(\M))$. The estimate \eqref{diff} also shows the continuity property of the flow with respect to the initial data. This concludes the proof of Proposition~\ref{PROP:LWP1}.
\end{proof}

Proposition~\ref{PROP:LWP1} applies in particular for initial data $u_0$ in the support of $\mu_\al$, i.e. almost surely for $u_0$ as in \eqref{init}. Since $\rho_{\al,\be}\ll\mu_\al$ by Theorem~\ref{THM:Gibbs1}, this also provides local well-posedness on the support of $\rho_{\al,\be}$.
\subsection{Almost sure global well-posedness and invariance of the Gibbs measure}\label{SUBSEC:GWP}
In this subsection we extend almost surely the flow $\Phi(t)$ defined locally on $H^s(\M)$ by Proposition~\ref{PROP:LWP1} and prove the invariance of the Gibbs measure $\rho_{\al,\be}$ under $\Phi$, thus completing the proof of Theorem~\ref{THM:GWP1}~(ii).

In order to prove almost sure global well-posedness, we rely on Bourgain's invariant measure argument \cite{BO94,BO96}; see also \cite{BTT2,ORTz}, or \cite{STz0,ORW} in the context of exponential nonlinearities. We thus study the approximate equation
\begin{align}\label{expNLSNN}
i\dt u_N + (-\Dlg)^\frac\al2 u_N = 2\gamma\be \P_{\le N}\Big[e^{\be|\P_{\le N}u|^2}\P_{\le N}u\Big],
\end{align}
where 
\begin{align}\label{PN}
\P_{\le N}u = \chi(-N^{-2}\Dlg)u = \sum_{n\ge 0}\chi(N^{-2}\ld_n^2)\langle u,\varphi_n\rangle_{L^2}\varphi_n 
\end{align}
for some $\chi\in C^{\infty}_0(\R)$, $\supp\chi\in [-1,1]$, $\chi\equiv 1$ on $[-\frac12,\frac12]$. As in \cite{BTT2}, the reason to consider $\P_{\le N}$ instead of $\Pi_{\le N}$ as in \eqref{expNLSN} is that it enjoys better mapping properties (see e.g. Corollary 2.2 in \cite{BGT}):
\begin{align}\label{bdPN}
\|\P_{\le N}u-u\|_{L^p} \to 0 
\end{align}
as $N\to+\infty$, for any $1<p<\infty$ and $u\in L^p(\M)$.

A straightforward adaptation of Lemma~\ref{LEM:global} shows that for any $N\in\N$ the Cauchy problem for \eqref{expNLSNN} is globally well-posed in $H^{s_1}(\M)$, for any $s_1\ge 0$, and that the truncated Gibbs measure 
\begin{align}\label{GibbsNN}
d\wt\rho_{\al,\be,N}(u)=\wt\ZZ_{\al,\be,N}^{-1}e^{-V_\be(\P_{\le N}u)}d\mu_\al(u)
\end{align}
is invariant under the global flow $\Phi_N$ of \eqref{expNLSNN} on $H^s(\M)$. Note that the convergence $\rho_{\al,\be,N}\to\rho_{\al,\be}$ in Theorem~\ref{THM:Gibbs1} also holds for $\wt\rho_{\al,\be,N}$ with minor\footnote{For $u_\al$ as in \eqref{init}, $\P_{\le N}u_\al$ is no longer a martingale contrary to $\Pi_{\le N}u_\al$, but this was only used to control $\|\Pi_Nu_\al\|_{L^\infty}$, which is easier for $\|\P_{\le N}u_\al\|_{L^\infty}$ due to Sobolev inequality and \eqref{bdPN}.} modifications.

As in \cite{BO94,BO96,BTT2,ORTz}, the following control on $\Phi_N$ is the key bound in globalizing the local flow $\Phi$.
\begin{proposition}\label{PROP:Si}
There exists $0<\be^\star\ll\|\s_\al\|_{L^\infty}^{-1}$ such that the following holds for any $\be\in (0,\be^\star]$: there exists $C>0$ such that for all $m,N\in\N$, there exists a measurable set $\Si_N^{m}\subset H^s(\M)$ such that\\
\textup{(i)} the bound
\begin{align}\label{measure-Si}
\wt\rho_{\al,\be,N}(H^s(\M)\setminus\Si_N^{m})\leq 2^{-m}
\end{align}
holds;\\
\textup{(ii)} for all $u_0\in\Si_N^{m}$ and $t\in\R$, the solution $\Phi_N(t)(u_0)$ to \eqref{expNLSNN} satisfies
\begin{align}
\big\|\Phi_N(t)(u_0)\big\|_{H^{s}} \le C\sqrt{m+\log(1+|t|)};
\label{estim-Si}
\end{align}
\textup{(iii)} there exists $c>0$ such that for every $t_0$, every $m\ge 1$ and $N\in\N$,
\begin{align}\label{estim-Si1}
\Phi_N(t_0)(\Si_N^m)\subset \Si_N^{m+[c\log(1+|t_0|)]+3}.
\end{align}
\end{proposition}
\begin{proof}
For $m,j\in \N$, we set 
\begin{align}\label{dlmj}
\tau = c(|\gamma|\be)^{-1}D^{-2}(m+j)^{-1}e^{-C_1\be D^2(m+j)}
\end{align}
which is the local time given by Proposition~\ref{PROP:LWP1} with $R=D\sqrt{m+j}$ and $\theta=R^{-1}$, for some $D\gg 1$ independent of $N,m,j$ to be fixed later, and $C_1=C_1(s,d)$ is independent of $\al$.

Next, as in \cite{BO94}, we can define
\begin{align*}
\Si_{N}^{m,j} &\deff \bigcap_{k=0}^{[2^j/\tau]}\Phi_N(k\tau)^{-1}\Big(B(m,j,D)\Big)
\end{align*}
where $[2^j/\tau]$ denotes the integer part of $2^j/\tau$, and
\begin{align}\label{Bmj}
B(m,j,D)&\deff \big\{u_0\in H^s(\M),~\|u_0\|_{H^s}\le D\sqrt{m+j}\big\}.
\end{align} 
By Proposition~\ref{PROP:LWP1}, for any $k=0,...,[2^j/\tau]$ and any $u_0\in \Phi_N(k\tau)^{-1}B(m,j,D)$, 
\begin{align}
\big\|\Phi_N(t+k\tau)(u_0)\big\|_{C_\tau H^{s}}\le D\sqrt{m+j}+D^{-1}(m+j)^{-\frac12} \le D\sqrt{m+j+1}
\label{bduN}
\end{align}
provided that $D$ is large enough.

Next, using that $\wt\rho_{\al,\be,N}$ is invariant under $\Phi_N$, we can estimate
\begin{align*}
&\wt\rho_{\al,\be,N}\big(H^s(\M)\setminus\Si_{N}^{m,j})\\&\leq \sum_{k=0}^{[2^j/\tau]}\wt\rho_{\al,\be,N}\Big\{\Phi_N(k\tau)^{-1}\Big(H^s(\M)\setminus B(m,j,D)\Big)\Big\}\\
&\le \big(\frac{2^j}{\tau}+1\big)\wt\rho_{\al,\be,N}\Big(H^s(\M)\setminus B(m,j,D)\Big)\le C\frac{2^j}{\tau}\mu_\al\Big(H^s(\M)\setminus B(m,j,D)\Big),
\end{align*}
where the last estimate comes from $d\rho_{\al,\be,N}=\ZZ_{\al,\be,N}^{-1}e^{-V_\be(\P_{\le N}u)}d\mu_\al$ with $e^{-V_\be(\P_{\le N}u)}$ in $L^{\infty}(\mu_\al)$ uniformly in $N\in\N$. Thus with the definition \eqref{Bmj} of $B(m,j,D)$ and the tail estimate \eqref{tail} we have
\begin{align}
\wt\rho_{\al,\be,N}\big(H^s(\M)\setminus\Si_{N}^{m,j})\le C2^jD^2(m+j)e^{\be C_1D^2(m+j)}e^{-c\|\s_{\al-2s}\|_{L^\infty}^{-1}D^2(m+j)}\le 2^{-(m+j)},
\label{Si1}
\end{align}
provided that we take $D\gg 1$ large enough depending on $c,C$ and $C_1$, and then $$0<\be\le \be^\star_\al = \frac{\|\s_{\al-2s}\|_{L^\infty}^{-1}}{2C_1D^2} \ll \|\s_\al\|_{L^\infty}^{-1},$$ independently of $N,m,j$, and since $\|\s_{\al-2s}\|_{L^\infty}\ge \|\s_\al\|_{L^\infty}$ for $s\ge 0$.
 
We then define
\begin{align*}
\Si_N^m = \bigcap_{j=1}^\infty \Si_N^{m,j}.
\end{align*}
With this definition, the bound \eqref{measure-Si} follows from \eqref{Si1}, while \eqref{estim-Si} and \eqref{estim-Si1} are a consequence of \eqref{bduN} and the definition of $\Si_N^m$ through the same argument as in \cite[Proposition 8.5]{BTT2}. This proves Proposition~\ref{PROP:Si}.
\end{proof}

As in \cite{BTT2,ORTz}, we now set
\begin{align*}
\Si^m = \limsup_{N\rightarrow\infty}\Si_{N}^{m}
\end{align*}
and
\begin{align*}
\Si = \bigcup_{m\in\N}\Si^m.
\end{align*}
By the same argument as in \cite{BTT2,ORTz} with the convergence in total variation of $\wt\rho_{\al,\be,N}$ to $\rho_{\al,\be}$, we have that $\Si$ has full $\rho_{\al,\be}$-measure, and for any $u_0\in \Si$, there exists $m\in\N$, $C>0$ and a sequence $N_p\rightarrow\infty$ such that for all $j,p\in\N$ and all $0\le t\le 2^j$,
\begin{align}
\|\Phi_{N_p}(t)(u_0)\|_{H^{s}} \le CD\sqrt{m+j+1}.
\label{global-bd1}
\end{align}
Similarly to Lemma 8.7 in \cite{BTT2}, the global well-posedness part of Theorem~\ref{THM:GWP1}~(ii) follows from the following key approximation result.

\begin{lemma}\label{LEM:approx}
There exist $R_0>0$, $c,C>0$ such that the following holds true.
Consider a sequence $u_{0,N_p}\in H^s(\M)$ and $u_0\in H^s(\M)$. 
Assume that there exists $R>R_0$ such that 
\begin{align*}
\|u_{0,N_p}\|_{H^s}\leq R,\quad \|u_0\|_{H^s} \leq R,\quad \lim_{p\rightarrow+ \infty}\|\P_{\le N_p}u_{0, N_p} -u_0\|_{H^s}=0.
\end{align*}
Then if we set $\tau= cR^{-1}e^{-C\be R^2}$ then 
$\Phi_{N_p}(t)( u_{0, N_p})$ and $\Phi(t)(u_0)$ exist for $t \in [0, \tau]$ and satisfy 
\begin{align*}
\|\Phi_{N_p}(t)( u_{0, N_p})\|_{C_{\tau}H^s}\leq 
R +1, \qquad \|\Phi(t)( u_{0})\|_{C_\tau H^s}\leq R +1.
\end{align*}
Furthermore
\begin{align*}
\lim_{p\rightarrow+ \infty} \|\P_{\le N_p} \Phi_{N_p}(t)( u_{0, N_p})- \Phi(t) (u_0)\|_{C_\tau H^s} =0.
\end{align*}
\end{lemma}

\begin{proof}
The first part of the lemma follows directly from the local 
well-posedness result of Proposition~\ref{PROP:LWP1} (with $\theta=1$), since this latter also holds unchanged for \eqref{expNLSNN} in place of \eqref{expNLS}, uniformly in $N\in\N$. For the convergence, we write $u_{N_p}$ and $u$ in place of $\Phi_{N_p}(t)( u_{0, N_p})$ and $\Phi(t)u_0$. Then we decompose again 
\begin{align}
&\P_{\le N_p} u_{N_p}(t)- u(t) \notag\\
&= e^{it(-\Dlg)^\frac\al2}\big(\P_{\le N_p}u_{0,N_p}-u_0)\notag\\
&\qquad -i2\gamma\be\int_0^te^{i(t-t')(-\Dlg)^\frac\al2}\Big[\P_{\le N_p}\Big(e^{\be |\P_{\le N_p}u_{N_p}|^2}\P_{\le N_p}u_{N_p}\Big)-e^{\be|u|^2}u\Big](t')dt'\notag\\
& =e^{it(-\Dlg)^\frac\al2}\big(\P_{\le N_p}u_{0,N_p}-u_0)+i2\gamma\be\int_0^te^{i(t-t')(-\Dlg)^\frac\al2}(1-\P_{\le N_p})\big(e^{\be|u|^2}u\big)(t')dt'\notag\\
&\qquad-i2\gamma\be\int_0^te^{i(t-t')(-\Dlg)^\frac\al2}\P_{\le N_p}\Big[e^{\be |\P_{\le N_p}u_{N_p}|^2}\P_{\le N_p}u_{N_p}-e^{\be|u|^2}u\Big](t')dt'\notag\\
&=\1 + \II + \III.\label{A0}
\end{align}
The first term can be estimated directly:
\begin{align}\label{A1}
\|\1\|_{C_\tau H^s}& \le \big\|\P_{\le N_p}u_{0,N_p}-u_0\big\|_{H^s} \to 0
\end{align}
as $p\to+\infty$ by assumption. 

\noi
For the second term, we have
\begin{align}\label{A2}
\|\II\|_{C_\tau H^s}& \le C\tau \big\|(1-\P_{\le N_p})\big(e^{\be|u|^2}u\big)\big\|_{C_\tau H^s} \too_{p\to\infty} 0,
\end{align}
by the mapping property of $\P_{\le N_p}$, since $e^{\be|u|^2}u\in C([0,\tau];H^s(\M))$ by the algebra property.

\noi
Finally, the third term is estimated by the mean value inequality and the Leibniz rule as before:
\begin{align}
\|\III\|& \le C\tau \big\|e^{\be |\P_{\le N_p}u_{N_p}|^2}\P_{\le N_p}u_{N_p}-e^{\be|u|^2}u\big\|_{C_\tau H^s}\notag\\
&\le C\tau \sum_{k\ge 0}\frac{\be^k}{k!}\big\||\P_{\le N_p}u_{N_p}|^{2k}\P_{\le N_p}u_{N_p}-|u|^{2k}u\big\|_{C_\tau H^s}\notag\\
&\le C\tau\sum_{k\ge 0}\frac{\be^k}{k!}(2k+1)C^k\|\P_{\le N_p}u_{N_p}-u\|_{C_\tau H^s}\big(\|\P_{\le N_p}u_{N_p}\|_{C_\tau H^s}^{2k}+\|u\|_{C_\tau H^s}^{2k}\big)\notag\\
&\le C\tau \|\P_{\le N_p}u_{N_p}-u\|_{C_\tau H^s}e^{\be C(\|\P_{\le N_p}u_{N_p}\|_{C_\tau H^s}^{2}+\|u\|_{C_\tau H^s}^2)}\notag\\
&\le \frac12\|\P_{\le N_p}u_{N_p}-u\|_{C_\tau H^s}\label{A3}
\end{align}
by choice of $\tau$ and the assumptions on $u_{N_p}$ and $u$. From \eqref{A0}-\eqref{A1}-\eqref{A2}-\eqref{A3}, we end up with
\begin{align*}
\big\|\P_{\le N_p} u_{N_p}(t)- u(t)\big\|_{C_\tau H^s} \le o_{p\to\infty}(1)+\frac12\big\|\P_{\le N_p} u_{N_p}(t)- u(t)\big\|_{C_\tau H^s},
\end{align*}
which proves the convergence part of Lemma~\ref{LEM:approx}.
\end{proof}

With Lemma~\ref{LEM:approx} at hand, the proof of Theorem~\ref{THM:GWP1}~(ii) is completed through the exact same argument as in the proof of Theorem 2.4 in \cite{BTT2}, so we omit details.

\section{Strong well-posedness on the circle}\label{SEC:GWP2}
In this last section, we give the proof of Theorem~\ref{THM:GWP2}. Thus from now on we fix $d=1$, $\M=\T$ with the standard metric, $\gamma\in\R$, $1+\frac{\sqrt{2}}2<\al\le 2$ and $0<\be\le \be^\star$, where $\be^\star>0$ is to be determined later. We also take $0<s<\frac{\al-d}2 < \frac12 < s_1<1$. We start by introducing our main function spaces in Subsection~\ref{SUBSEC:func} and the gauge function used to rewrite \eqref{expNLS} in a more favourable manner in Subsection~\ref{SUBSEC:gauge}. In Subsection~\ref{SUBSEC:LWP2} we present the proof of the local well-posedness result needed for Theorem~\ref{THM:GWP2} (see Proposition~\ref{PROP:LWP}), and then establish the multilinear estimates needed for Proposition~\ref{PROP:LWP} in Subsection~\ref{SUBSEC:multi}. Finally we conclude the proof of Theorem~\ref{THM:GWP2} by implementing Bourgain's invariant measure argument in Subsection~\ref{SUBSEC:GWP2}.
\subsection{Function spaces and large deviation bounds}\label{SUBSEC:func}
We begin by fixing some notations.

We will write
\begin{align*}
\ft u_n= \int_{\T}u(x)\cj e_n(x)dx
\end{align*}
for the Fourier coefficients of $u$, where $\{e_n\}_{n\in\Z}$ is the orthogonal basis of $L^2(\T)$ defined as $e_n(x)=\frac1{2\pi}e^{inx}$, $n\in\Z$. Note that $\{e_n\}$ is not normalized in $L^2(\T)$, but instead enjoys the property ${\ds\prod_{j=1}^{k}e_{n_j}(x) = e_{\sum_{j=1}^{k}n_j}(x)}$ for any $k\ge 1$, $n_1,...,n_k\in\Z$ and $x\in\T$.

\noi For a function of time and space, we write $\ft u_n(t)$ or $\ft u(t,n)$ for the spatial Fourier coefficients, and $\F u_n(\tau)$ or $\F u(\tau,n)$ for the space-time Fourier transform.

\noi
We write $\N_0 = \N\cup \{0\}$, and we define the set of dyadic integers as $2^{\N_0}$. For a dyadic integer $M\in 2^{\N_0}$, we define the Littlewood-Paley projections
\begin{align}\label{PM}
\P_M : u\mapsto \sum_{n\in\Z}\chi_M(n)\ft u_n e_n,
\end{align}
where
\begin{align*}
\chi_M(x)=\begin{cases}
\chi(M^{-1}x)-\chi(2M^{-1}x),~M\ge 2,\\
\chi(x),~M=1,
\end{cases}
\end{align*}
with $\chi\in C^\infty_0(\R)$, $\supp\chi\subset [-2,2]$, $\chi\equiv 1$ on $[-1,1]$.

\noi
Similarly, we use dyadic decompositions in $K\in 2^{\N_0}$ of the modulation variables $(\tau - |n|^\al)$, defined as
\begin{align*}
\wt\P_{K} : u\mapsto \F_{t,x}^{-1}\big[\chi_K(\tau - |n|^\al)\F u\big].
\end{align*}
We then simply write $\P_{K,M}u = \wt\P_K\P_Mu$.

We use the classical $X^{s,b}$ space, namely
\begin{align}\label{Xb}
\big\|u\big\|_{X^{s,b}} \deff \big\|\jb{\tau-|n|^\al}^b\jb{n}^s\ft u(\tau,n)\big\|_{L^2_\tau\ell^2_n}\sim \Big\{\sum_{K,M\in 2^{\N_0}}M^{2s}K^{2b}\big\|\P_{K,M}u\big\|_{L^2_{t,x}}^2\Big\}^{\frac12},
\end{align}
and their time localized version
\begin{align*}
\big\|u\big\|_{X^{s,b}(T)} = \inf_{\wt u\equiv u\text{ on }[-T,T]}\big\|\wt u\big\|_{X^{s,b}}.
\end{align*}

It is well-known that for $b>\frac12$ we have the embedding
\begin{align}\label{embeddings}
X^{s,b}(T)\subset C([-T,T];H^s(\T)).
\end{align}
If $\chi$ is as above $\chi_T(t) \deff \chi(t/T)$ for some $T>0$, then it holds for any $s\in\R$, any $-\frac12<b'\le b <\frac12$ and any $T\in (0,1]$
\begin{align}\label{XT}
\big\|\chi_Tu\big\|_{X^{s,b'}} \les T^{b-b'}\|u\|_{X^{s,b}}
\end{align}
for any $u\in X^{s,b}$; see \cite{Tao}. Moreover, the space $X^{s,b}(T)$ enjoys the following energy type estimate (see e.g. \cite[Proposition 2.12]{Tao}): for $\frac12<b<1$, and $T\in (0,1]$, if $w$ solves the non-homogeneous fractional Schr\"odinger equation
\begin{align*}
\begin{cases}
i\dt w +(-\dx^2)^{\frac{\al}2}w = F \text{ on }[-T,T],\\
w(0)=w_0
\end{cases}
\end{align*}
then it holds
\begin{align}\label{linearX}
\big\|w\big\|_{X^{s,b}(T)} \les \big\|w_0\big\|_{H^s} + \big\|F\big\|_{X^{s,b-1}(T)}.
\end{align}

\noi It is also well known that free solutions lie in $X^{s,b}$: for any $T>0$, $s,b\in\R$ and $u_0\in H^s(\T)$ it holds
\begin{align}\label{free}
\big\|\chi_T(t)e^{it(-\dx^2)^\frac\al2}u_0\big\|_{X^{s,b}}\les \|u_0\|_{H^s}.
\end{align}
Indeed this follows directly from ${\displaystyle \F\big\{\chi_T(t)e^{it(-\dx^2)^\frac\al2}u_0\big\}(\tau,n) = \F\big(\jb{\nabla}^b\chi_T\big)(\tau-|n|^\al)\ft u_0(n)}$ and the fact that $\ft\chi_T$ is a Schwartz function.

Let $\eta_0$ satisfy the same properties as $\chi$ in \eqref{PM}, and moreover
\begin{align*}
\sum_{q\in\Z}\eta_0(t-q)=1
\end{align*}
for all $t\in\R$. For $s,s_1$ as above, take $0<\eps_0\ll 1$ and $p_0\gg 2$ such that 
\begin{align}\label{delta}
\eps_0>\frac1{p_0}\text{ and }s+\eps_0<\frac{\al-1}2.
\end{align}
 In the same spirit as in \cite{STz1} (see also \cite{BT,STz0}), for $s\in\R$ we define the space $Y^{s}$ for the initial data through the norm
\begin{align}\label{Ysb}
\|u_0\|_{Y^{s}} = \|u_0\|_{H^s} + \|u_0\|_{Z^s},
\end{align}
where
\begin{align}\label{Zb}
\|u_0\|_{Z^{s}}=\sum_{q\in\Z}\jb{q}^{-2}\sum_{M\in 2^{\N_0}}M^{s+\eps_0}\Big\|\eta_0(t)\P_{M}e^{i(t+q)(-\dx^2)^{\frac{\al}2}}u_0\big]\Big\|_{L^\infty_tL^{p_0}_x}.
\end{align}

\noi Finally, we will follow\footnote{See also \cite{DNY1} or, in a different context, \cite{Hadac}.} \cite{STz1} and consider the sum space $Y^s(\T)+H^{s_1}(\T)$ endowed with the natural norm
\begin{align*}
\|u_0\|_{Y^s+H^{s_1}}=\inf\{\|\gf\|_{Y^s}+\|\phi\|_{H^{s_1}},~u_0=\gf+\phi\}.
\end{align*}

Clearly, from \eqref{free} and \eqref{Ysb}, it holds
\begin{align}\label{Ysb1}
\big\|\chi_T(t)e^{it(-\dx^2)^\frac\al2}u_0\big\|_{L^\infty_tH^s}\les \|u_0\|_{Y^{s}}
\end{align}
for any $s,b\in\R$. 

\noi Note that we also have the estimates\footnote{The last estimate is a consequence of \eqref{bdPN}, which also holds in this case due to the relation between $\Pi_{\le N}$ and the Hilbert transform on $\T$ along with the boundedness of the Hilbert transform on $L^p(\T)$; see e.g. \cite[Chapter 4]{Gra}.}
\begin{align}\label{normN}
\|\Pi_{\le N}w\|_{X^{s_1,b}}\le \|w\|_{X^{s_1,b}}\text{ and }\|\Pi_{\le N}\gf\|_{Y^{s}}\les \|\gf\|_{Y^{s}},
\end{align}
uniformly in $N\in\N$. 

The next lemma collects some further properties of the space $Y^s$.

\begin{lemma}\label{LEM:Y}
The following estimates hold: for any $b\in\R$, $s'\le s$, $T\in (0,1]$, and any $u_0\in Y^s$,
\begin{align}\label{Ysb2}
\sum_{K,M\in 2^{\N_0}}K^bM^{s'}\big\|\P_{K,M}\big\{\chi(t)e^{it(-\dx^2)^\frac\al2}u_0\big\}\big\|_{L^\infty_{T,x}}\les \|u_0\|_{Y^s},
\end{align}
and for any $|t_0|\le 1$ it holds
\begin{align}\label{Ysb3}
\big\|e^{it_0(-\dx^2)^\frac\al2}u_0\big\|_{Y^s} \les \|u_0\|_{Y^s}.
\end{align}
\end{lemma}
\begin{proof}
To prove \eqref{Ysb2}, we compute for $K,M\in 2^{\N_0}$
\begin{align*}
\big\|\P_{K,M}\big\{\chi(t)e^{it(-\dx^2)^\frac\al2}u_0\big\}\big\|_{L^\infty_{t,x}}& = \big\|(\P_K\chi)(t)e^{it(-\dx^2)^\frac\al2}\P_Mu_0(x)\big\|_{L^{\infty}_{T,x}}\\
&\les K^{-A}M^{\eps_0}\big\|\P_Me^{it(-\dx^2)^\frac\al2}u_0(x)\big\|_{L^\infty_{T}L^{p_0}_x}
\end{align*}
by Bernstein's inequality and using that $\ft\chi$ is a Schwartz function. Using then that  $T\in (0,1]$, that $\sum_{q\in\Z}\eta_0(t-q)=1$ with the support property of $\eta_0$, and summing the previous estimate in $M$ (with $s'\le s$), we obtain
\begin{align*}
&\sum_{K,M\in 2^{\N_0}}K^bM^{s'}\big\|\P_{K,M}\big\{\chi(t)e^{it(-\dx^2)^\frac\al2}u_0\big\}\big\|_{L^\infty_{T,x}}\\
&\qquad\qquad\les \sum_{M\in 2^{\N_0}}M^{s+\eps_0}\sum_{|q|\le 3}\big\|\eta_0(t-q)\P_M e^{it(-\dx^2)^\frac\al2}u_0\big\|_{L^\infty_tL^{p_0}_x} \\
&\qquad\qquad\les \sum_{|q|\le 3}\sum_{M\in 2^{\N_0}}M^{s+\eps_0}\big\|\eta_0(t)\P_M e^{i(t+q)(-\dx^2)^\frac\al2}u_0\big\|_{L^\infty_tL^{p_0}_x} .
\end{align*}
This is enough for \eqref{Ysb2}. Similarly, for \eqref{Ysb3} using again the properties of $\eta_0$ and that $|t_0|\le 1$ we have
\begin{align*}
\big\|e^{it_0(-\dx^2)^\frac\al2}u_0\big\|_{Y^s} &= \sum_{q\in\Z}\sum_{M\in 2^{\N_0}}\jb{q}^{-2}M^{s+\eps_0}\big\|\eta_0(t)\P_{M}e^{i(t+q+t_0)(-\dx^2)^\frac\al2}u_0\big\|_{L^\infty_tL^{p_0}_x}\\
&\les \sum_{q\in\Z}\sum_{M\in 2^{\N_0}}\jb{q}^{-2}M^{s+\eps_0}\sum_{|q_1|\le 3}\big\|\eta_0(t-t_0)\eta_0(t-q_1)\P_{M}e^{i(t+q)(-\dx^2)^\frac\al2}u_0\big\|_{L^\infty_tL^{p_0}_x}\\
&\les \sum_{|q_1|\le 3}\sum_{q_2\in\Z}\jb{q_2-q_1}^{-2}\sum_{M\in 2^{\N_0}}M^{s+\eps_0}\big\|\eta_0(t)\P_{M}e^{i(t+q_2)(-\dx^2)^\frac\al2}u_0\big\|_{L^\infty_tL^{p_0}_x}\\
&\les \sum_{q_2\in\Z}\jb{q_2}^{-2}\sum_{M\in 2^{\N_0}}M^{s+\eps_0}\big\|\eta_0(t)\P_{M}e^{i(t+q_2)(-\dx^2)^\frac\al2}u_0\big\|_{L^\infty_tL^{p_0}_x}.
\end{align*}
This proves \eqref{Ysb3}.
\end{proof}

The following lemma shows that the random initial data $u_\al$ in \eqref{init} belongs almost surely to $Y^{s}$ for any $s<\frac{\al-1}2$.
\begin{lemma}\label{LEM:devY}
For any $\al>d$, $s<\frac{\al-1}2$ and $b<1$, there exists $c,C,\dl>0$ independent of $\al$ such that for any $R>0$ and $1\le N_1\le N_2\le\infty$, it holds
\begin{align}\label{tailY}
\mu_\al\big(\|\Pi_{\le N_2}\Pi_{>N_1}u_0\|_{Y^{s}}>R\big) \le Ce^{-c (\s_{\al-2s-2\eps_0,N_2}-\s_{\al-2s-2\eps_0,N_1})^{-1}R^2}.
\end{align}
\end{lemma}
In particular from Borel-Cantelli's lemma, $u_\al\in Y^{s}$ almost surely. Here we simply write $\s_\al\equiv \|\s_\al\|_{L^\infty}$ since it is independent of $x$ for $\M=\T$.
\begin{proof}
First, note that in view of \eqref{Ysb} and \eqref{tail}, it suffices to obtain the tail estimate for the $Z^s$ part of the $Y^{s}$ norm. Then for $p\ge p_0$ we have by  Markov's inequality
\begin{align*}
\mu_\al\big(\|\Pi_{\le N_2}\Pi_{>N_1}u_0\|_{Z^s}>R\big) \le R^{-p}\E_{\mu_\al}\|\Pi_{\le N_2}\Pi_{>N_1}u_0\|_{Z^s}^p.
\end{align*} 
Then by definition \eqref{Zb} with the use of Minkowski's inequality along with Khinchin's inequality (Lemma~\ref{LEM:Khinchin}) we have
\begin{align*}
&\big\|\Pi_{\le N_2}\Pi_{>N_1}u_0\big\|_{L^p(\mu_\al)Z^s}\\
&=\Big\|\jb{q}^{-2}M^{s+\eps_0}\P_{M}\big[\eta_0(t)e^{i(t+q)(-\dx^2)^{\frac{\al}2}}\Pi_{\le N_2}\Pi_{>N_1}u_0\big]\Big\|_{L^p(\mu_\al)\ell^1_q\ell^1_ML^\infty_tL^{p_0}_{x}}\\
&\les \Big\|\jb{q}^{-2}K^{\eps_0}M^{s+\eps_0}\P_{K,M}\big[\eta_0(t)e^{i(t+q)(-\dx^2)^{\frac{\al}2}}\Pi_{\le N_2}\Pi_{>N_1}u_0\big]\Big\|_{\ell^1_{q,K,M}L^{p_0}_{t,x}L^p(\mu_\al)}\\
&\les p^\frac12\Big\|\jb{q}^{-2}K^{\eps_0}M^{s+\eps_0}\P_{K,M}\big[\eta_0(t)e^{i(t+q)(-\dx^2)^{\frac{\al}2}}\Pi_{\le N_2}\Pi_{>N_1}u_0\big]\Big\|_{\ell^1_{q,K,M}L^{p_0}_{t,x}L^2(\mu_\al)}.
\end{align*}
Next, we compute for any $t\in\R$ and $x\in\T$
\begin{align*}
&\E_{\mu_\al}\big|\P_{K,M}\big[\eta_0(t)e^{i(t+q)(-\dx^2)^{\frac{\al}2}}\Pi_{\le N_2}\Pi_{>N_1}u_0\big](t,x)\big|^2\\
&=\E_{\mu_\al}\Big|\int_\R\sum_{N_1<|n|\le N_2}\chi_K(\tau-|n|^\al)\chi_M(n)\ft \eta_0(\tau-|n|^\al)\frac{g_n}{\jb{n}^\frac\al2}e^{iq|n|^\al}e^{it\tau}e_n(x)d\tau\Big|^2\\
&= \int_\R\int_\R\sum_{N_1<|n|\le N_2}\chi_K(\tau_1-|n|^\al)\chi_K(\tau_2-|n|^\al)\chi_M(n)^2\ft \eta_0(\tau_1-|n|^\al)\cj{\ft \eta_0}(\tau_2-|n|^\al)\\
&\qquad\qquad\qquad\qquad \cdot e^{it(\tau_1-\tau_2)}\jb{n}^{-\al}d\tau_1d\tau_2\\
&=\Big(\sum_{N_1<|n|\le N_2}\chi_M(n)^2\jb{n}^{-\al}\Big)\big|\P_K\eta_0(t)\big|^2.
\end{align*}
Gathering the estimates above finally yields
\begin{align*}
\mu_\al\big(\|\Pi_{\le N_2}\Pi_{>N_1}u_0\|_{Y^{s}}>R\big) &\le C^pp^\frac{p}2R^{-p}(\s_{\al-2s-2\eps_0,N_2}-\s_{\al-2s-2\eps_0,N_1})\big\|\eta_0\big\|_{W^{2\eps_0,p_0}_t}^p,
\end{align*}
where the summation in $M$ is ensured by \eqref{delta}. Then the tail estimate \eqref{tailY} follows from optimizing in $p$ the previous estimate.
\end{proof}

\subsection{The gauge transform}\label{SUBSEC:gauge}
 We start by rewriting the equation for $u$ as
\begin{align}
i\dt u +(-\dx^2)^{\frac{\al}2}u = 2\gamma\be\sum_{k\ge 0}\frac{\be^k}{k!}|u|^{2k}u.
\end{align}
As in \cite{BO96,DNY1} for the pure power NLS with random initial data, in order to avoid some resonances\footnote{Here we only need to get rid of very few resonances. In the two-dimensional setting, note that for the cubic nonlinearity \cite{BO96} this gauge transform essentially amounts to Wick ordering, i.e. removing all \emph{non trivial} resonances, while in the case of higher power nonlinearity, Deng-Nahmod-Yue \cite{DNY1} perform a gauge transform to decompose the (renormalized) power $:|u|^{2k}u:$ in sums of \emph{simple} $(2k+1)$-linear forms, i.e. with only \emph{over-paired} pairings, which amounts to remove a lot more resonances. See also Remark~\ref{REM:DNY} above.} in the power nonlinearities $|u|^{2k}u$, we define the gauge
\begin{align}\label{gauge1}
\GG(u) = 2\gamma\be\sum_{k\ge 0}\frac{\be^k}{k!}(k+1)\A[|u|^{2k}],
\end{align}
where
\begin{align}\label{A}
\A[f] = \frac1{2\pi}\int_{\T}f(x)dx = \ft f(0).
\end{align}
The gauged function
\begin{align}\label{gauge2}
v(t) = \GGG(u)(t) \deff \exp\Big(i\int_0^t\GG(u(t'))dt'\Big)u(t)
\end{align}
then solves the truncated nonlinear Schr\"odinger equation
\begin{align}
i\dt v + (-\dx^2)^{\frac{\al}2}v &= 2\gamma\be \sum_{k\ge 0}\frac{\be^k}{k!}\Big(|v|^{2k} - (k+1)\A\big[|v|^{2k}\big]\Big)v\label{v1}\\
&=2\gamma\be \sum_{k\ge 1}\frac{\be^k}{k!}\big(\NN_{2k+1}(v)-\RR_{2k+1}(v)\big).\label{v}
\end{align}
In particular the sum in \eqref{v} runs on $k\ge 1$ since the linear term for $k=0$ in \eqref{v1} cancels out.

Note also that the gauge transform $\GGG$ in \eqref{gauge2} is invertible, with inverse
\begin{align}\label{gauge3}
\GGG^{-1}(v)(t) = \exp\Big(-i\int_0^t\GG(v(t'))dt'\Big)v(t),
\end{align} 
and in particular it does not change the initial data $v(0)=u(0)$. 

In \eqref{v} above, for $\QQ\in\{\NN,\RR\}$ and $k\ge 1$, we write 
\begin{align*}
\QQ_{2k+1}(v)=\QQ_{2k+1}(\underbrace{v,...,v}_{2k+1}),
\end{align*}
and the $(2k+1)$-linear forms $\NN_{2k+1}$ and $\RR_{2k+1}$ are defined as
\begin{align}\label{QQ}
\QQ_{2k+1}(v_1,...,v_{2k+1}) = \sum_{n_0\in\Z}\Big(\sum_{(n_1,...,n_{2k+1})\in\G_{\QQ_{2k+1}}(n_0)}\prod_{j=1}^{2k+1}\ft v_j(n_j)^{\i_j}\Big)e_{n_0},~~\QQ\in\{\NN,\RR\},
\end{align}
 where the signs are given by $\i_j=(-1)^{j+1}$ for $j=0,...,2k+1$.Here and in the following, for a complex number $z\in\C$ and a sign $\i\in\{\pm\}\simeq \{\pm1\}$ we use the convention $z^+=z$ and $z^-=\cj z$.

\noi The frequency regions $\G_{\QQ_{2k+1}}(n_0)$ are defined as
\begin{align}\label{GN}
\G_{\NN_{2k+1}}(n_0) = \Big\{(n_1,...,n_{2k+1})\in\Z^{2k+1},~~ \sum_{j=0}^{2k+1}\i_jn_j=0,~~\forall \ell\in\{0,...,k\},~ n_{2\ell+1}\neq n_0\Big\},
\end{align}
and
\begin{align}\label{GR}
\G_{\RR_{2k+1}}(n_0) &= \Big\{(n_1,...,n_{2k+1})\in\Z^{2k+1},~~ \sum_{j=0}^{2k+1}\i_jn_j=0,\notag\\
&\qquad\qquad\exists \ell_1\neq \ell_2\in \{0,...,k\},~~n_0=n_{2\ell_1+1}=n_{2\ell_2+1}\Big\}.
\end{align}
 With these definition, we have indeed that \eqref{v1} and \eqref{v} coincide, since for any $\ell=0,...,k$
\begin{align*}
\A[|v|^{2k}]v = \sum_{n_{0}\in\Z}e_{n_{0}}\ft v(n_{0})\sum_{\substack{n_1,...,n_{2k}\\ \sum_{j=1}^{2k}\i_jn_j=0}}\prod_{j=1}^{2k}\ft v(n_j)^{\i_j} = \sum_{n_0\in\Z}e_{n_0}\sum_{\substack{n_1,...,n_{2k+1}\\ \sum_{j=1}^{2k+1}\i_jn_j=n_0\\ n_{2\ell+1}=n_0}}\prod_{j=1}^{2k+1}\ft v(n_j)^{\i_j},
\end{align*}
due to the symmetry with respect to the frequencies with negative sign (i.e. the $n_j$'s for odd $j$) and also with respect to the ones with positive signs.

We will thus write
\begin{align}\label{GRL}
\G_{\RR_{2k+1}}(n_0) &= \bigsqcup_{\ell_1\neq \ell_2\in \{0,...,k\}}\G_{\RR_{2k+1}}(n_0,\ell_1,\ell_2) \notag\\
&\deff\bigsqcup_{\ell_1\neq \ell_2\in \{0,...,k\}}\Big\{(n_1,...,n_{2k+1})\in\Z^{2k+1},~~ \sum_{j=0}^{2k+1}\i_jn_j=0,~~n_0=n_{2\ell_1+1}=n_{2\ell_2+1}\Big\}.
\end{align}

Note that the gauge $\GG(u)$ a priori is ill-defined for a generic function $u\in H^{s}(\T)$ since $s<\frac{\al-d}2\le\frac12$. However,\footnote{This issue is somehow similar to what happens when dealing with singular Gibbs measure for the nonlinear wave or Schr\"odinger equations. For example, in the case of the cubic NLS in dimension $d=2$ considered in \cite{BO96}, the renormalized (Wick ordered) nonlinearity also only makes sense in a subclass of the form \eqref{class}. See e.g. \cite[Remark 1.3]{ORTz} for a more detailed discussion.} let us recall that we look for a solution $v$ to \eqref{v} under the special form $u=e^{it(-\dx^2)^{\frac\al2}}u_0+w$, with $u_0$ given by \eqref{init} and $w\in X^{s_1,b}$ for some $s_1,b>\frac12$. In particular,  $\GG(\<1>(u_0)+w)$ is almost surely well-defined, and so the inverse gauge transform \eqref{gauge3} is well-defined on the class \eqref{class} if $s_1,b>\frac12$ and $u_0\sim \mu_\alpha$. Indeed, in view of \eqref{linear} and \eqref{embeddings}, we can control for any $|t|\le T\le 1$
\begin{align}
\big|\GG(e^{it(-\dx^2)^{\frac\al2}}u_0+w(t))\big|&= \Big|\be\sum_{k\ge 0}\frac{\be^k}{k!}(k+1)\A\big[|e^{it(-\dx^2)^{\frac\al2}}u_0+w(t)|^{2k}\big]\Big|\notag\\
& \le \be\sum_{k\ge 0}\frac{\be^k}{k!}(k+1)\big(\|e^{it(-\dx^2)^{\frac\al2}}u_0\|_{L^{\infty}_{T,x}}+\|w\|_{X^{s_1,b}(T)}\big)^{2k}\notag\\
&\les  e^{\be C\big(\|u_0\|_{Y^s}+\|w\|_{X^{s_1,b}(T)}\big)^2}
\label{bd-gauge}
\end{align}
where we used \eqref{Ysb2}. In particular this last term is $\mu_\al$-almost surely finite in view of Lemma~\ref{LEM:devY}.

\subsection{A local well-posedness result}\label{SUBSEC:LWP2}
In this subsection we establish a general local well-posedness result, uniform in $N\in\N$, for the model
\begin{align}\label{wN}
\begin{cases}
i\dt w +  (-\dx^2)^{\frac{\al}2})w = F_N(\<1>(u_0)+w),\\
w(0)=0,
\end{cases}
\end{align}
where the nonlinearity is given by\footnote{Note that here we use the sharp projector $\Pi_{\le N}$ instead of $\P_{\le N}$ as in \eqref{expNLSNN}. This is because the relation $\Pi_{\le N}^2 = \Pi_{\le N}$ will allow us to avoid an extra term depending on $\GG$ when we will do the gauge transform for the truncated equation \eqref{expNLSN}.}
\begin{align}\label{FN}
F_N(\<1>(u_0)+w)=2\gamma\be\sum_{k\ge 1}\frac{\be^k}{k!}\Pi_{\le N}\Big\{\NN_{2k+1}\big(\Pi_{\le N}[\<1>(u_0)+w]\big)-\RR_{2k+1}\big(\Pi_{\le N}[\<1>(u_0)+w]\big)\Big\}
\end{align}
for $N\in\N\cup\{+\infty\}$, with the convention that $\Pi_{\le \infty} = \Id$.

Indeed, recall that, as in \cite{BO96} and subsequent works, for $T\in (0,1]$ we look for a solution to \eqref{v} on $[-T,T]$ under the form \eqref{ansatz}, i.e.
\begin{align*}
v= \<1>(u_0) + w \qquad \text{with}\qquad \<1>(u_0)=e^{it(-\dx^2)^{\frac{\al}2}}u_0,
\end{align*}
where $w\in X^{s_1,b}(T)$ for some $s_1,b>\frac12$ solves \eqref{wN} (with $N=\infty$).

Here the $(2k+1)$-linear forms $\NN_{2k+1}$ and $\RR_{2k+1}$ are defined in \eqref{QQ} above, and by analogy with the cubic case ($k=3$, see \cite{BO96}) we will refer to $\RR_{2k+1}$ as the resonant part, while $\NN_{2k+1}$ is the non-resonant part. 

The goal of this subsection is thus to prove the following uniform local well-posedness result for \eqref{wN}.
\begin{proposition}\label{PROP:LWP}
Let $1+\frac{\sqrt{2}}2<\al\le 2$,  $s=\frac{\al-1}2-\eps$ and $b=\frac12+\wt\eps$ for some $0<\eps,\wt\eps\ll 1$, and $\frac12< s_1<\al s$. Then there exists $C_1,c>0$ and $0<\theta\ll 1$ such that for any $\gamma\in\R$, any $\be>0$, and any $R>0$, if we set $T= c(|\gamma|\be R^{-2})^{\theta^{-1}}e^{-\be \theta^{-1} C_1 R^2}$, then for any $u_0=\gf+\phi\in Y^{s}(\T)+H^{s_1}(\T)$ with $\|\gf\|_{Y^{s}}+\|\phi\|_{H^{s_1}}\le R$, and any $N\in\N\cup\{\infty\}$, the Cauchy problem for \eqref{wN} with $\<1>=e^{it(-\dx^2)^{\frac{\al}2}}(\gf+\phi)$ has a unique solution $w$ in $X^{s_1,b}(T)$, which satisfies
\begin{align}\label{bdw}
\|w\|_{X^{s_1,b}(T)}\le R^{-1}.
\end{align}
Moreover, the flow map $\Phi:u_0\mapsto w$ is Lipschitz continuous.
\end{proposition}
\noi In particular the condition $\al>1+\frac{\sqrt{2}}2$ is there to ensure that it is possible to find $s<\frac{\al-1}2$ and $s_1>\frac12$ such that $s_1<\frac{2\al-1}{\al}s$.

Proposition \ref{PROP:LWP} classically follows from multilinear estimates on $\QQ_{2k+1}$, $\QQ\in\{\NN,\RR\}$. Since the regions $\G_{\QQ_{2k+1}}(n_0)$ in \eqref{GN}-\eqref{GR} are symmetric with respect to the frequencies with same sign, we can expand the multilinear terms appearing in \eqref{wN} as
\begin{align}\label{Qm}
\QQ_{2k+1}\big(\<1>(u_0)+w) = \sum_{m_+=0}^{k+1}\sum_{m_-=0}^k\binom{k+1}{m_+}\binom{k}{m_-}\QQ_{2k+1}(\{\<1>(\gf),\<1>(\phi)+w\}_{(m_+,m_-)})
\end{align}
where we use the shorthand notation for the $(2k+1)$-tuple $\{\<1>,w\}_{(m_+,m_1)}$ defined as
\begin{align}\label{m+-}
\big(\{\<1>(\gf),\<1>(\phi)+w\}_{(m_+,m_1)}\big)_j = \begin{cases} \<1>(\gf) \text{ if } j\in\mathfrak{G},\\
\<1>(\phi)+w \text{ if }j\in\mathfrak{D},\end{cases}
\end{align}
and the set of indices are given by
\begin{align}\label{indices}
\mathfrak{G} = \mathfrak{G}_{(m_+,m_-)} = \{2\ell-1,~~\ell = 1,...,m_+\}\cup\{2\ell,~~\ell=1,...,m_-\},
\end{align}
with the convention that $\GF=\emptyset$ if $m_+=m_-=0$, and $$\DF=\{1,...,2k+1\}\setminus\mathfrak{G}.$$ 

With these notations, we start with estimating the resonant parts.
\begin{lemma}\label{LEM:RR}
Let $\frac32<\al\le 2$, $0<s<\frac{\al-1}2$, $\frac12 <s_1,b<1$ such that $\frac{s_1}2<s$, and $0<\theta\ll1$. Then there exists $C>0$ and $0<\theta\ll 1$ such that the following holds. For $T\in (0,1]$, $k\ge 1$ and $(m_+,m_-)\in\{0,...,k+1\}\times\{0,...,k\}$, let $z_1,...,z_{2k+1}$ be such that $z_j(t)=e^{it(-\dx^2)^{\frac{\al}2}}\mathfrak{g}_j$ with $\mathfrak{g}_j\in Y^{s}$ if $j\in\mathfrak{G}$, and $z_j\in X^{s_1,b}(T)$ if $j\in \mathfrak{D}$. Then
\begin{align}\label{RR}
\Big\|\RR_{2k+1}(z_1,...,z_{2k+1})\Big\|_{X^{s_1,b-1}(T)} \le Ck(k+1)T^{\frac32-b-\theta}\prod_{j\in\mathfrak{G}}\big\|\mathfrak{g}_j\big\|_{Y^{s}}\prod_{j\in\mathfrak{D}}\big\|z_j\big\|_{X^{s_1,b}(T)}.
\end{align}
\end{lemma}

\noi The condition $\al>\frac32$ ensures that we can find $s<\frac{\al-1}2$ and $\frac12<s_1<2s$. We will prove Lemma~\ref{LEM:RR} in the next subsection. 

\noi We now state the multilinear estimate for the non-resonant terms, whose proof is also postponed to the next subsection.
\begin{lemma}\label{LEM:NN}
Let $1+\frac{\sqrt{2}}2<\al\le 2$,  $s=\frac{\al-1}2-\eps$ and $b=\frac12+\wt\eps$ for some $0<\eps,\wt\eps\ll 1$, and $\frac12< s_1<\frac{2\al-1}{\al}s$. Then there exists $C>0$ and $0<\theta\ll 1$ such that the following holds. For $T\in (0,1]$, $k\ge 1$ and $(m_+,m_-)\in\{0,...,k+1\}\times\{0,...,k\}$, let $z_1,...,z_{2k+1}$ be such that $z_j(t)=e^{it(-\dx^2)^{\frac{\al}2}}\mathfrak{g}_j$ with $\mathfrak{g}_j\in Y^{s}$ if $j\in\mathfrak{G}$, and $z_j\in X^{s_1,b}(T)$ if $j\in \mathfrak{D}$. Then
\begin{align}\label{NN}
\Big\|\NN_{2k+1}(z_1,...,z_{2k+1})\Big\|_{X^{s,b-1}(T)}\le C^kk^{1+s_1}(k+1)T^{\theta}\prod_{j\in\mathfrak{G}}\big\|\mathfrak{g}_j\big\|_{Y^{s}}\prod_{j\in\mathfrak{D}}\big\|z_j\big\|_{X^{s_1,b}(T)}.
\end{align}
\end{lemma}
The condition $1+\frac{\sqrt{2}}2<\al\le 2$ ensures that it is possible to find $s<\frac{\al-1}2$ and $\frac12< s_1<\frac{2\al-1}{\al}s$. With Lemma~\ref{LEM:RR} and Lemma~\ref{LEM:NN}, we can now give the proof of the local well-posedness result.
\begin{proof}[Proof of Proposition~\ref{PROP:LWP}] 
Recall that, for $u_0=\gf+\phi\in Y^s+H^{s_1}$, we look for a solution $w$ to 
\begin{align}
i\dt w + (-\dx^2)^\frac\al2 w&=F_N(\<1>(\gf)+\<1>(\phi)+w)\notag\\
& = 2\gamma\be \sum_{k\ge 1}\frac{\be^k}{k!}\sum_{m_+=0}^{k+1}\sum_{m_-=0}^k\binom{k+1}{m_+}\binom{k}{m_-}\notag\\
&\qquad\cdot\Pi_{\le N}\Big\{ \NN_{2k+1}(\{\Pi_{\le N}\<1>(\gf),\Pi_{\le N}(\<1>(\phi)+w)\}_{(m_+,m_-)})\notag\\
&\qquad\qquad\qquad\qquad+\RR_{2k+1}(\{\Pi_{\le N}\<1>(\gf),\Pi_{\le N}(\<1>(\phi)+w)\}_{(m_+,m_-)})\Big\}\label{wQ}
\end{align}
with $w(0)=0$, where we recall the notation $\{\<1>(\gf),\<1>(\phi)+w\}_{(m_+,m_-)}$ from \eqref{m+-}. Let $T\in (0,1]$ and $R\ge \|\gf\|_{Y^{s}}+\|\phi\|_{H^{s_1}}$. To solve \eqref{wQ} on $[-T,T]$, we solve in $X^{s_1,b}(T)$ the fixed point problem for the Duhamel formulation
\begin{align}
w(t)&=-i\int_0^te^{i(t-t')(-\dx^2)^\frac\al2}F_N(\<1>(u_0)+w)(t')dt'\notag\\
&=-2i\gamma\be \sum_{k\ge 1}\frac{\be^k}{k!}\sum_{m_+=0}^{k+1}\sum_{m_-=0}^k\binom{k+1}{m_+}\binom{k}{m_-}\notag\\
&\qquad\cdot\int_0^te^{i(t-t')(-\dx^2)^\frac\al2}\Pi_{\le N}\Big\{ \NN_{2k+1}(\{\Pi_{\le N}\<1>(\gf),\Pi_{\le N}(\<1>(\phi)+w)\}_{(m_+,m_-)})\notag\\
&\hspace{3cm}+\RR_{2k+1}(\{\Pi_{\le N}\<1>(\gf),\Pi_{\le N}(\<1>(\phi)+w)\}_{(m_+,m_-)})\Big\}(t')dt'.
\label{DuhamelQ}
\end{align}
Let $\G_{(\gf,\phi)}^N(w)$ denote the right-hand side of \eqref{DuhamelQ}, then from \eqref{linearX}, \eqref{normN}, \eqref{free}, Lemma~\ref{LEM:RR} and Lemma~\ref{LEM:NN}, we have
\begin{align*}
&\big\|\G_{(\gf,\phi)}^N(w)\big\|_{X^{s_1,b}(T)}\\
&\le C|\gamma|\be T^\theta \sum_{k\ge 1}\frac{\be^k}{k!}\sum_{m_+=0}^{k+1}\sum_{m_-=0}^k\binom{k+1}{m_+}\binom{k}{m_-}C^kk^{1+s_1}(k+1)\\
&\qquad\qquad\cdot\|\gf\|_{Y^{s}}^{m_++m_-}\|\<1>(\phi)+w\|_{X^{s_1,b}(T)}^{2k+1-m_+-m_-}\\
&\le C|\gamma|\be T^\theta \big(\|\gf\|_{Y^{s}}+\|\phi\|_{H^{s_1}}+\|w\|_{X^{s_1,b}(T)}\big)e^{\be C_0(\|\gf\|_{Y^{s}}+\|\phi\|_{H^{s_1}}+\|w\|_{X^{s_1,b}(T)})^2}.
\end{align*}
Therefore, if $T=c(|\gamma|\be  R^{-2})^{\theta^{-1}}e^{-\be \theta^{-1} C_1 R^2}$ for some small $c>0$ and large $C_1>0$, $\G_{(\gf,\phi)}^N(w)$ maps the ball $B_{R^{-1},T}$ of radius $R^{-1}$ of $X^{s_1,b}(T)$ into itself. Similarly, if $\gf_1,\gf_2\in Y^{s}(\T)$ and $\phi_1,\phi_2\in H^{s_1}(\T)$ with norm bounded by $R$, for any $w_1,w_2$ in $B_{R^{-1},T}$, we have by multilinearity
\begin{align}
&\big\|\G_{(\gf_1,\phi_1)}^N(w_1)-\G_{(\gf_1,\phi_1)}^N(w_1)\big\|_{X^{s_1,b}(T)}\notag\\
&\le C|\gamma|\be T^\theta \sum_{k\ge 1}\frac{\be^k}{k!}\sum_{m_+=0}^{k+1}\sum_{m_-=0}^k\binom{k+1}{m_+}\binom{k}{m_-}C^kk^{1+s_1}(k+1)\notag\\
&\qquad\cdot\Big\{\sum_{j=1}^{m_++m_-}\|\gf_1-\gf_2\|_{Y^{s}}\|\gf_1\|_{Y^{s}}^{m_++m_--j}\|\gf_2\|_{Y^{s}}^{j-1}\|\<1>(\phi_1)+w_1\|_{X^{s_1,b}(T)}^{2k+1-m_+-m_-}\notag\\
&\qquad\qquad + \sum_{j=m_++m_-+1}^{2k+1}\|\gf_2\|_{Y^{s}}^{m_++m_-}\|(\<1>(\phi_1)+w_1)-(\<1>(\phi_2)+w_2)\|_{X^{s_1,b}(T)}\notag\\
&\qquad\qquad\qquad\qquad\cdot\|\<1>(\phi_1)+w_1\|_{X^{s_1,b}(T)}^{2k+1-j}\|\<1>(\phi_2)+w_2\|_{X^{s_1,b}(T)}^{j-m_+-m_--1}\Big\}\notag\\
&\le C|\gamma|\be T^\theta \big(\|\gf_1-\gf_2\|_{Y^{s}(T)}+\|\phi_1-\phi_2\|_{H^{s_1}}+\|w_1-w_2\|_{X^{s_1,b}(T)}\big)\notag\\
&\qquad\qquad\cdot e^{\be C_2(\|\gf_1\|_{Y^{s}}+\|\gf_2\|_{Y^{s}}+\|\phi_1\|_{H^{s_1}}+\|\phi_2\|_{H^{s_1}}+\|w_1\|_{X^{s_1,b}(T)}+\|w_2\|_{X^{s_1,b}(T)})^2}\notag\\
&\le \frac12  \big(\|\gf_1-\gf_2\|_{Y^{s}(T)}+\|\phi_1-\phi_2\|_{H^{s_1}}+\|w_1-w_2\|_{X^{s_1,b}(T)}\big)
\label{diff-Gamma}
\end{align}
by taking $T\sim (|\gamma|\be R^{-2})^{\theta^{-1}}e^{-\be \theta^{-1} C_3 R^2}$ as above. This shows that $\G_{(\gf,\phi)}^N$ also defines a contraction on $B_{R^{-1},T}$. This proves existence of a solution by Banach fixed point theorem, unique in $B_{R^{-1},T}$, and the same iteration as before ensures uniqueness in the whole of $X^{s_1,b}(T)$. The previous estimate also shows the Lipschitz dependence on $\gf+\phi$, and by construction the fixed point $w$ of $\G_{(\gf,\phi)}^N$ in $B_{R^{-1},T}$ satisfies \eqref{bdPhi}.
\end{proof}

With the general local well-posedness result of Proposition~\ref{PROP:LWP} at  hand, we can now give the proof of Theorem~\ref{THM:GWP2}~(i).

\begin{proof}[Proof of Theorem~\ref{THM:GWP2}~(i)]
Let $\gamma\in\R$, $1+\frac{\sqrt{2}}2<\al\le 2$, and $\be>0$. We also take $0<s<\frac{\al-1}2$ and $s_1,b>\frac12$ such that $s_1<\al s$. For $m\in\N$, let us then define $\Si^m\subset H^s(\T)$ by
\begin{align*}
\Si^m = \big\{u_0\in Y^s(\T),~\|u_0\|_{Y^{s}}\le m\big\}.
\end{align*}
Then by Lemma~\ref{LEM:devY}, there exists $c_0,C_0>0$ such that 
\begin{align}\label{tailSi}
\mu_\al\big(H^s(\T)\setminus \Si^m\big)\le C_0e^{-c_0m^2}.
\end{align}
Applying Proposition~\ref{PROP:LWP} for $u_0\in\Si^m$ (note that $Y^s(\T)\subset Y^s(\T)+H^{s_1}(\T)$) gives constants $c_1,C_1,\theta>0$ independent of $m$ such that, if $u_0\in \Si^m$ and $T=c_1 (|\gamma|\be m^{-2})^{\theta^{-1}}e^{-\be \theta^{-1} C_1 m^2}$, there exists a unique solution $w\in X^{s_1,b}(T)$ to \eqref{wN} on $[-T,T]$ with $\<1> = e^{it(-\dx^2)^\frac\al2}u_0$. This in turn implies that
\begin{align*}
u(t) = e^{-it\int_0^t\GG(\<1>(u_0)+w)(t')dt'}(\<1>(u_0)+w)
\end{align*}
is the unique solution to \eqref{expNLS} in the class \eqref{class}, where the gauged $\GG$ given by \eqref{gauge1} is well-defined on $\<1>(Y^s)+X^{s_1,b}(T)$ from \eqref{bd-gauge}. Finally, taking 
\begin{align*}
\Si = \liminf_{m\to\infty}\Si^m,
\end{align*}
we see from \eqref{tailSi} and Borel-Cantelli's lemma that $\Si$ is of full $\mu_\al$-measure. This concludes the proof of Theorem~\ref{THM:GWP2}~(i).
\end{proof}

\subsection{Proof of the multilinear estimates}\label{SUBSEC:multi}
In this subsection we prove the multilinear estimates of Lemmas~\ref{LEM:RR} and~\ref{LEM:NN} used to establish Proposition~\ref{PROP:LWP}. We start with the estimate on the resonant terms.

\begin{proof}[Proof of Lemma~\ref{LEM:RR}]
Let $T\in (0,1]$ and $z_j$ be as in Lemma~\ref{LEM:RR}, and fix extensions $\wt z_j$ of $z_j$ for $j\in\DF$. Let also $\chi_T$ be as in \eqref{XT}. In the following we abuse notations and then write $z_j = \chi(t)e^{it(-\dx^2)^\frac\al2}\gf_j$ if $j\in\GF$ and $z_j = \chi_T(t)\wt z_j$ if $j\in\DF$. First, observe that from Sobolev inequality and the embedding properties \eqref{embeddings}-\eqref{Ysb1}-\eqref{Ysb2} of $Y^{s}$ and $X^{s_1,b}$ for $s_1,b>\frac12$ and $\frac{s_1}2<s$, it is enough to get the bound \eqref{RR} with the right-hand side replaced by ${\displaystyle CT^\theta\prod_j \|z_j\|_{C_T(H^{\frac{s_1}2}\cap L^\infty)}}$. Then we have for any $0<\theta\ll1$
\begin{align}
&\Big\|\RR_{2k+1}(z_1,...,z_{2k+1})\Big\|_{X^{s_1,b-1}(T)}\notag\\
 &\qquad\le \Big\|\chi_T\RR_{2k+1}(z_1,...,z_{2k+1})\Big\|_{X^{s_1,b-1}}\notag\\
&\qquad\le CT^{1-b-\theta}\sum_{N_0}N_0^{s_1}\Big\|\sum_{(n_1,...,n_{2k+1})\in\G_{\RR_{2k+1}}(n_0)}\prod_{j=1}^{2k+1}\ft z_j(t,n_j)^{\i_j}\Big\|_{L^2_T\ell^2_{|n_0|\sim N_0}},\label{R1}
\end{align}
where here and in the sequel the notation $\ell^2_{|n_0|\sim M_0}$ means that the summation in $n_0$ is performed over the support of $\chi_{M_0}$ as in \eqref{PM}.

In view of \eqref{GRL}, and since $k\ge 1$, we can bound
\begin{align}\label{R2}
&\Big\|\sum_{(n_1,...,n_{2k+1})\in\G_{\RR_{2k+1}}(n_0)}\prod_{j_1}^{2k+1}\ft z_j(t,n_j)^{\i_j}\Big\|_{L^2_T\ell^2_{|n_0|\sim M_0}}\\ &\qquad \les k(k+1)\sup_{\ell_1\neq \ell_2}\Big\|\ft z_{2\ell_1+1}(t,n_0)\ft z_{2\ell_2+1}(t,n_0)\notag\\
&\qquad\qquad\times\sum_{\sum_{\substack{j\neq 2\ell_1+1\\j\neq 2\ell_2+1}}\i_jn_j=0}\prod_{\substack{j\neq 2\ell_1+1\\j\neq 2\ell_2+1}}\ft z_j(t,n_j)^{\i_j}\Big\|_{L^2_T\ell^2_{|n_0|\sim M_0}}\notag\\
&\qquad \les k(k+1)\sup_{\ell_1\neq \ell_2}\big\|\ft z_{2\ell_1+1}(t,n_0)\big\|_{L^{\infty}_T\ell^{\infty}_{|n_0|\sim M_0}}\big\|\ft z_{2\ell_2+1}(t,n_0)\big\|_{L^\infty_T\ell^\infty_{|n_0|\sim M_0}}\notag\\
&\qquad\qquad\times\Big\|\sum_{\sum_{\substack{j\neq 2\ell_1+1\\j\neq 2\ell_2+1}}\i_jn_j=0}\prod_{\substack{j\neq 2\ell_1+1\\j\neq 2\ell_2+1}}\ft z_j(t,n_j)^{\i_j}\Big\|_{L^2_T\ell^2_{|n_0|\sim M_0}}.
\end{align}
Putting \eqref{R2} in \eqref{R1} and using the bound $\|\ft f\|_{\ell^\infty}\le\|f\|_{L^1}\les \|f\|_{L^2}$ with Cauchy-Schwarz inequality in $M_0$, we get
\begin{align*}
&\Big\|\RR_{2k+1}(z_1,...,z_{2k+1})\Big\|_{X^{s_1,b-1}(T)}\\ 
&\le CT^{1-b-\theta}k(k+1)\sup_{\ell_1\neq \ell_2}\sum_{M_0}M_0^{s_1}\big\|\ft z_{2\ell_1+1}(t,n_0)\big\|_{L^{\infty}_T\ell^{\infty}_{|n_0|\sim M_0}}\big\|\ft z_{2\ell_2+1}(t,n_0)\big\|_{L^\infty_T\ell^\infty_{|n_0|\sim M_0}}\\
&\qquad\qquad\times\Big\|\sum_{\sum_{\substack{j\neq 2\ell_1+1\\j\neq 2\ell_2+1}}\i_jn_j=0}\prod_{\substack{j\neq 2\ell_1+1\\j\neq 2\ell_2+1}}\ft z_j(t,n_j)^{\i_j}\Big\|_{L^2_T\ell^2_{|n_0|\sim M_0}}\\
&\les T^{1-b-\theta}k(k+1)\sup_{\ell_1\neq \ell_2}\big\|z_{2\ell_1+1}\big\|_{L^{\infty}_TH^{\frac{s_1}2}_x}\big\|z_{2\ell_2+1}\big\|_{L^\infty_TH^{\frac{s_1}2}_x}\Big\|\prod_{\substack{j\neq 2\ell_1+1\\j\neq2\ell_2+1}}z_j^{\i_j}\Big\|_{L^2_{T,x}}\\
&\les T^{\frac32-b-\theta} k(k+1)\prod_{j=1}^{2k+1}\|z_j\|_{C_T(H^{\frac{s_1}2}_x\cap L^{\infty}_x)}.
\end{align*}
This proves Lemma \ref{LEM:RR}.
\end{proof}

Next, we turn to the proof of the estimate for the non-resonant terms.
\begin{proof}[Proof of Lemma~\ref{LEM:NN}]
To prove \eqref{NN}, we define extensions (still denoted $z_j$) as in the proof of Lemma~\ref{LEM:RR}, and then we start by decomposing 
\begin{align*}
z_j = \sum_{K_j,M_j}z_{K_j,M_j}=\sum_{K_j,M_j}\P_{K_j,M_j}z_j.
\end{align*}
 Then we have
\begin{align}
&\Big\|\chi_T(t)\NN_{2k+1}(z_1,...,z_{2k+1})\Big\|_{X^{s,b-1}}\notag\\
&\qquad\le CT^{\theta}\sum_{K_1,...,K_{2k+1}}\sum_{M_1,...,M_{2k+1}}\Big\{\sum_{K_0,M_0}K_0^{2(b+\theta-1)}M_0^{2s_1}\notag\\
&\qquad\qquad\qquad\Big\|\P_{K_0,M_0}\NN_{2k+1}(z_{K_1,M_1},...,z_{K_{2k+1},M_{2k+1}})\Big\|_{L^2_t,x}^2\Big\}^{\frac12}.\label{N1}
\end{align}

In the following, we write $M^{(0)}\ge M^{(1)}\ge ...\ge M^{(2k+1)}$ to denote the decreasing rearrangement of $M_0,...,M_{2k+1}$, and similarly for $K^{(j)}$. Note that in order to have $\P_{K_0,M_0}\NN_{2k+1}(z_{K_1,M_1},...,z_{K_{2k+1},M_{2k+1}})$ non zero, this yields the condition $M^{(1)}\le M^{(0)} \le CkM^{(1)}$. 

In order to prove \eqref{N1}, we first state a general estimate: using the definition \eqref{v} of $\NN_{2k+1}$, H\"older's inequality and an estimate similar to \eqref{R2}, we have for any $j_{M},j_K\in \GF\cup\DF$ (recall that $k\ge 1$ so that $\GF\cup\DF\neq\{1\}$) and any $\zz_1,...,\zz_{2k+1}$
\begin{align}
&\Big\|\P_{\le M_0}\NN_{2k+1}(\zz_{1},...,\zz_{2k+1})\Big\|_{L^2_t,x}\notag\\
&=\Big\|\P_{\le M_0}\Big\{\prod_{j=1}^{2k+1}\zz_{j}^{\i_j}-\sum_{\ell=1}^{k+1}\zz_{2\ell-1}^+\A\big[\prod_{j\neq 2\ell-1}\zz_{j}^{\i_j}\big]+\RR_{2k+1}(\zz_1,...,\zz_{2k+1})\Big\}\Big\|_{L^2_{t,x}}\notag\\
&\le \min\Big\{\big\|\zz_{j_{M}}\big\|_{L^2_{t,x}}\prod_{j\neq j_{M}}\big\|\zz_j\big\|_{L^{\infty}_{t,x}}~,~\big\|\zz_{j_{M}}\big\|_{L^\infty_tL^2_x}\big\|\zz_{j_K}\big\|_{L^2_tL^\infty_x}\prod_{j\neq j_{M},j_K}\big\|\zz_j\big\|_{L^{\infty}_{t,x}}\Big\}\notag\\
&\qquad + \sum_{\ell=1}^{k+1}\min\Big\{\big\|\zz_{2\ell-1}\big\|_{L^2_{t,x}}\Big\|\prod_{j\neq 2\ell-1}\zz_j^{\i_j}\Big\|_{L^{\infty}_t L^1_x}~,~\big\|\zz_{2\ell-1}\big\|_{L^{\infty}_tL^2_x}\Big\|\prod_{j\neq 2\ell-1}\zz_j^{\i_j}\Big\|_{L^2_t L^1_x}\Big\} \notag\\
&\qquad + \sum_{\ell_1\neq \ell_2}\min\Big\{\big\|\P_{\le M_0}\zz_{2\ell_1-1}\big\|_{L^{\infty}_tL^2_x}\big\|\P_{\le M_0}\zz_{2\ell_2-1}\big\|_{L^{\infty}_tL^2_x}\Big\|\P_{\le M_0}\Big[\prod_{\substack{j\neq 2\ell_1-1\\j\neq2\ell_2-1}}\zz_{j}^{\i_j}\Big]\Big\|_{L^2_{t,x}}~,\notag\\
&\qquad\qquad \big\|\P_{\le M_0}\zz_{2\ell_1-1}\big\|_{L^2_{t,x}}\big\|\P_{\le M_0}\zz_{2\ell_2-1}\big\|_{L^{\infty}_tL^2_x}\Big\|\P_{\le M_0}\Big[\prod_{\substack{j\neq 2\ell_1-1\\j\neq2\ell_2-1}}\zz_{j}^{\i_j}\Big]\Big\|_{L^\infty_{t}L^2_x}\Big\}.
\label{N2}
\end{align}

This estimate will allow us to prove \eqref{NN} after exploiting some multilinear smoothing effect provided by the relation between the $M_j$'s and the $K_j$'s. Indeed, the conditions $\sum_{j=0}^{2k+1}\i_jn_j=0=\sum_{j=0}^{2k+1}\i_j\tau_j$ in the definition of $\F\NN_{2k+1}$ along with $|\tau_j-|n_j|^\al|\sim K_j$ on the support of $\F z_{K_j,M_j}$ imply that
\begin{align}\label{Res}
\Big|\sum_{j=0}^{2k+1}\i_j|n_j|^\al\Big|=\Big|\sum_{j=0}^{2k+1}\i_j(\tau_j-|n_j|^\al)\Big| \les k K^{(0)}.
\end{align} 
We then perform a case-by-case analysis to exploit the lower bound \eqref{Res} and estimate the norm on the right-hand side of \eqref{N1}. 

\bigskip\textbf{Case (1): If $M^{(0)}\les k M_{j_{\max}}$ for some $j_{\max}\in\mathfrak{D}$.} 

\noi Then putting \eqref{N1} and \eqref{N2} together with $\zz_{j_M}=\sum_{M_{j_{\max}}\ge ck^{-1}M_0}z_{K_{j_{\max}},M_{j_{\max}}}$, and using Bernstein's inequality with the Fourier support of $z_{K_j,M_j}$, we have
\begin{align*}
&\Big\|\chi_T(t)\NN_{2k+1}(z_1,...,z_{2k+1})\Big\|_{X^{s,b-1}}\\
&\le CT^{\frac12+\theta}k(k+1)\sum_{K_j}\sum_{\underset{j\neq j_{\max}}{M_j}}\Big\{\sum_{K_0,M_0}K_0^{2(b+\theta-1)}M_{0}^{2s_1}\Big\|\sum_{M_{j_{\max}}\ge ck^{-1}M_0}z_{K_{j_{\max}},M_{j_{\max}}}\Big\|_{L^{\infty}_tL^2_x}^2\\
&\qquad\qquad\times\prod_{j\neq j_{\max}}\big\|z_{K_j,M_j}\big\|_{L^{\infty}_{t,x}}^2\Big\}^{\frac12}\\
&\le C^kT^{\frac12+\theta}k(k+1)\sum_{K_j}\sum_{\underset{j\neq j_{\max}}{M_j}}\Big\{\sum_{M_{j_{\max}}}\sum_{K_0}\sum_{M_0\le CkM_{j_{\max}}}K_0^{2(b+\theta-1)}M_{0}^{2s_1}K_{j_{\max}}\big\|z_{K_{j_{\max}},M_{j_{\max}}}\big\|_{L^2_{t,x}}^2\\
&\qquad\qquad\times\prod_{j\in\mathfrak{G}}\big\|z_{K_j,M_j}\big\|_{L^{\infty}_{t,x}}^2\prod_{\underset{j\neq j_{\max}}{j\in\mathfrak{D}}}K_jM_j\big\|z_{K_j,M_j}\big\|_{L^2_{t,x}}^2\Big\}^{\frac12}\\
&\le  C^kT^{\frac12+\theta}k^{1+2s_1}(k+1)\|z_{j_{\max}}\|_{X^{s_1,b}}\prod_{j\in\mathfrak{G}}\big\|\mathfrak{g}_j\big\|_{Y^{s}}\prod_{\substack{j\in\mathfrak{D}\\j\neq j_{\max}}}\big\|z_j\big\|_{X^{s_1,b}},
\end{align*}
where we used that $s>0$, $\frac12<s_1,b<1$ and $0<\theta\ll1$ to sum on $K_j,M_j$. In the last estimate we also used the embedding \eqref{embeddings} 
This is enough for \eqref{NN}.

\bigskip\textbf{Case (2): If $M^{(0)}\gg kM_j$ for all $j\in\mathfrak{D}$, and $M^{(0)}\les k(M^{(2)})^{\frac{\al}{\al-1}}$.}

\noi In this case we have $M^{(\ell)}=M_{j_\ell}$, $\ell=0,1,2$, where  $j_0,j_1\in \mathfrak{G}\cup\{0\}$ and $j_2\in\{0,...,2k+1\}$. Indeed note that from $k\ge 1$ and the assumption $M_j\ll M^{(0)}$ for all $j\in\DF$ we have $\GF\neq \emptyset$, and $\#\GF=1$ implies $j_0=0$ or $j_1=0$. Without loss of generality we can assume that $j_1,j_2\neq 0$ (as $j_0=0$ is allowed). Using again \eqref{N2}, and after summing in $K_0,M_0$ keeping in mind that $M_{j_2}\le M_{j_1}\le M_{j_0}\les kM_{j_2}^{\frac\al{\al-1}}$ and $M_0\le M_{j_0}\sim M_{j_1}$, we then get the bound
\begin{align*}
&\Big\|\chi_T(t)\NN_{2k+1}(z_1,...,z_{2k+1})\Big\|_{X^{s,b-1}}\\
&\le C^kT^{\frac12+\theta}k(k+1)\sum_{K_j,M_j}(kM_{j_1})^s(kM_{j_2}^{\frac{\al}{\al-1}})^{s_1-s}\big\|z_{K_{j_2},M_{j_2}}\big\|_{L^{\infty}_TL^2_x}\prod_{j\neq j_2}\big\|z_{K_j,M_j}\big\|_{L^{\infty}_{T,x}}\\
&\le C^kT^{\frac12+\theta}k^{1+s_1}(k+1)\big\|z_{j_2}\|_{L^{\infty}_TH^{\frac\al{\al-1}(s_1-s)}_x}\big\|\mathfrak{g}_{j_1}\big\|_{Y^{s}}\prod_{j\neq j_1,j_2}\big\|z_j\big\|_{X^{s_1,b}}.
\end{align*}
 This is enough for \eqref{NN}, since $\frac\al{\al-1}(s_1-s)<s$. 

%

\bigskip\textbf{Case (3): If $ M^{(0)}\gg k\max\{M_j,(M^{(2)})^\frac\al{\al-1}\}$ for all $j\in\mathfrak{D}$.}

\noi In this case we have $\{M^{(0)},M^{(1)}\}=\{M_0,M_{j_{\max}}\}$ for some $j_{\max}\in\mathfrak{G}$.\\

\noi\textbf{Subcase 3.1: if $j_{\max}$ is even.} Then from the lower bound \eqref{Res} on the resonance function and that $\i_j=(-1)^{j+1}$ it holds
\begin{align*}
K^{(0)} \gtrsim k^{-1} \Big|\sum_{j=0}^{2k+1}\i_j|n_j|^\al\Big| = k^{-1}\Big||n_0|^\al + |n_{j_{\max}}|^\al +O(M^{(2)})^\al\Big| \gtrsim k^{-1}M_0^\al,
\end{align*}
since in this regime $M_0\sim M_{j_{\max}}\sim M^{(0)} \gg (M^{(2)})^\frac\al{\al-1} \ge M^{(2)}$. Thus we can use the gain of $(K^{(0)})^{-\frac12+}$ to recover $\frac{\al}2-$ spatial derivatives. 

Indeed, in the case $K^{(0)}=K_0$, we have from \eqref{N1}-\eqref{N2} after summing in $M_0$ and $K_0\gtrsim M_{j_{\max}}^\al$ and recalling that $j_{\max}$ is even:
\begin{align}
&\Big\|\chi_T(t)\NN_{2k+1}(z_1,...,z_{2k+1})\Big\|_{X^{s,b-1}}\notag\\
&\le C^kT^{\theta}k^{1-b-2\theta}\sum_{K_j,M_j}M_{j_{\max}}^{s_1-\al(1-b-2\theta)}\Big\{\big\|z_{K_{j_{\max}},M_{j_{\max}}}\big\|_{L^2_{t,x}}\prod_{j\neq j_{\max}}\big\|z_{K_j,M_j}\big\|_{L^{\infty}_{t,x}}\notag\\
&\qquad + \sum_{\ell=1}^{k+1}\big\|z_{K_{2\ell-1},M_{2\ell-1}}\big\|_{L^\infty_tL^2_x}\Big\|\prod_{j\neq 2\ell-1}z_{K_j,M_j}^{\i_j}\Big\|_{L^{2}_t L^1_x} \notag\\
&\qquad + \sum_{\ell_1\neq \ell_2}\big\|z_{K_{2\ell_1-1},M_{2\ell-1}}\big\|_{L^{\infty}_tL^2_x}\big\|z_{K_{2\ell_2-1},M_{2\ell-1}}\big\|_{L^{\infty}_tL^2_x}\Big\|\prod_{j\neq 2\ell_1-1,2\ell_2-1}z_{K_j,M_j}^{\i_j}\Big\|_{L^2_{t,x}}\Big\}\notag\\
&\le C^kT^{\frac12+\theta}(k+1)k^{2-b-2\theta}\|\gf_{j_{\max}}\|_{Y^{s}}\prod_{\substack{j\in\GF\\j\neq j_{\max}}}\|\gf_j\|_{Y^{s}}\prod_{j\in\DF}\|z_j\|_{X^{s_1,b}},\label{K0}
\end{align}
where we used that $s_1-\al(1-b-2\theta) = s_1-\al(\frac12-\wt\eps-2\theta)<s=\frac{\al-1}2-\eps$ for $0<\theta,\wt\eps,\eps \ll 1$ small enough (recall that $\al>1$). This is enough for \eqref{NN}.

The case $K^{(0)}=K_{j_{\max}}$ is handled similarly (again, $j_{\max}$ is even):
\begin{align}
&\Big\|\chi_T(t)\NN_{2k+1}(z_1,...,z_{2k+1})\Big\|_{X^{s,b-1}}\notag\\
&\le C^kT^{\theta}k^\frac12\sum_{K_j,M_j}M_{j_{\max}}^{s_1-\frac\al2}K_{j_{\max}}^\frac12\Big\{\big\|z_{K_{j_{\max}},M_{j_{\max}}}\big\|_{L^2_{t,x}}\prod_{j\neq j_{\max}}\big\|z_{K_j,M_j}\big\|_{L^{\infty}_{t,x}}\notag\\
&\qquad + \sum_{\ell=1}^{k+1}\big\|z_{K_{2\ell-1},M_{2\ell-1}}\big\|_{L^\infty_{t}L^2_x}\Big\|\prod_{j\neq 2\ell-1}z_{K_j,M_j}^{\i_j}\Big\|_{L^2_t L^1_x} \notag\\
&\qquad + \sum_{\ell_1\neq \ell_2}\big\|z_{K_{2\ell_1-1},M_{2\ell-1}}\big\|_{L^{\infty}_tL^2_x}\big\|z_{K_{2\ell_2-1},M_{2\ell-1}}\big\|_{L^{\infty}_tL^2_x}\Big\|\prod_{j\neq 2\ell_1-1,2\ell_2-1}z_{K_j,M_j}^{\i_j}\Big\|_{L^2_{t,x}}\Big\}\notag\\
&\le C^kT^{\theta}(k+1)k^\frac32\|\gf_{j_{\max}}\|_{Y^s}\prod_{\substack{j\in\GF\\j\neq j_{\max}}}\|\gf_j\|_{Y^{s}}\prod_{j\in\DF}\|z_j\|_{X^{s_1,b}},\label{Kmax}
\end{align}
since as above $s_1-\frac\al2 < s$, and $b>\frac12$. 

When $K^{(0)}=K_{j_K}$ for some $j_K\in\GF\cup\DF$, $j_K\neq j_{\max}$, we use again \eqref{N1}-\eqref{N2} to get the bound
\begin{align}
&\Big\|\chi_T(t)\NN_{2k+1}(z_1,...,z_{2k+1})\Big\|_{X^{s,b-1}}\notag\\
&\le C^kT^{\theta}\sum_{K_j,M_j}M_{j_{\max}}^{s_1-\frac\al2}K_{j_K}^\frac12\Big\{\big\|z_{K_{j_{\max}},M_{j_{\max}}}\big\|_{L^\infty_{t}L^2_x}\big\|z_{K_{j_{K}},M_{j_{K}}}\big\|_{L^2_{t}L^\infty_x}\prod_{j\neq j_{\max},j_K}\big\|z_{K_j,M_j}\big\|_{L^{\infty}_{t,x}}\notag\\
&\qquad + \sum_{\ell=1}^{k+1}\big\|z_{K_{2\ell-1},M_{2\ell-1}}\big\|_{L^\infty_tL^2_x}\Big\|\prod_{j\neq 2\ell-1}z_{K_j,M_j}^{\i_j}\Big\|_{L^2_t L^1_x} \notag\\
&\qquad + \sum_{\ell_1\neq \ell_2}\big\|z_{K_{2\ell_1-1},M_{2\ell-1}}\big\|_{L^{\infty}_tL^2_x}\big\|z_{K_{2\ell_2-1},M_{2\ell-1}}\big\|_{L^{\infty}_tL^2_x}\Big\|\prod_{j\neq 2\ell_1-1,2\ell_2-1}z_{K_j,M_j}^{\i_j}\Big\|_{L^2_{t,x}}\Big\}\notag\\
&\le C^kT^{\theta}(k+1)k\|\gf_{j_{\max}}\|_{Y^{s}}\|z_{j_K}\|_{Y^{s}+X^{s,b}}\prod_{\substack{j\in\GF\\j\neq j_{\max},j_K}}\|\gf_j\|_{Y^s}\prod_{\substack{j\in\DF\\j\neq j_K}}\|z_j\|_{X^{s_1,b}}.\label{KK}
\end{align}
This is enough for \eqref{NN}.\\

\noi\textbf{Subcase 3.2: if $j_{\max}$ is odd.} Then, the definition \eqref{QQ} of $\NN_{2k+1}$, in particular the definition \eqref{GN} of $\G_{\NN_{2k+1}}(n_0)$, the conditions $kM^{(1)}\gtrsim M^{(0)}\gg k(M^{(2)})^\frac\al{\al-1} \ge M^{(2)}$ and $\sum_{j=0}^{2k+1}\i_jn_j=0$ with $|n_j|\sim M_j$ and $\i_j=(-1)^{j+1}$ imply that $|n_{j_{\max}}-n_0|=O(M^{(2)})$, thus $n_{j_{\max}}$ and $n_0$ must have the same sign. Thus the lower bound \eqref{Res} along with the condition $n_j\neq n_0$ for odd $j$ in the definition of $\G_{\NN_{2k+1}}(n_0)$ imply that
\begin{align}
kK^{(0)}\gtrsim \Big|\sum_{j=0}^{2k+1}\i_j|n_j|^\al\Big|&\ge \big||n_0|^\al-|n_{j_{\max}}|^\al\big|-O\big(k(M^{(2)})^\al\big)\notag\\
& \ge \big||n_0|^{\al-1}+|n_{j_{\max}}|^{\al-1}\big|-O\big(k(M^{(2)})^\al\big) \gtrsim M_{j_{\max}}^{\al-1}\label{Res1}
\end{align}
since we are in the case $M_{j_{\max}}\sim M^{(0)}\gg k(M^{(2)})^\frac{\al}{\al-1}$.

Then we proceed similarly as in the previous subcase, except that now $j_{\max}=2\ell_{\max}-1$ is odd. If $K^{(0)}=K_0$, we can use \eqref{Res1} with \eqref{N1}-\eqref{N2} as in \eqref{K0} to get
\begin{align*}
&\Big\|\chi_T(t)\NN_{2k+1}(z_1,...,z_{2k+1})\Big\|_{X^{s,b-1}}\\
&\le C^kT^{\theta}k^{1-b-2\theta}\sum_{K_j,M_j}M_{j_{\max}}^{s_1-(\al-1)(1-b-2\theta)}\Big\{\big\|z_{K_{j_{\max}},M_{j_{\max}}}\big\|_{L^2_{t,x}}\prod_{j\neq j_{\max}}\big\|z_{K_j,M_j}\big\|_{L^{\infty}_{t,x}}\\
&\qquad + \|z_{K_{j_{\max}},M_{j_{\max}}}\|_{L^2_{t,x}}\Big\|\prod_{j\neq j_{\max}}z_{K_j,M_j}^{\i_j}\Big\|_{L^\infty_t L^1_x}\\
&\qquad+\sum_{\ell\neq \ell_{\max}}\big\|z_{K_{2\ell-1},M_{2\ell-1}}\big\|_{L^\infty_tL^2_x}\Big\|\prod_{j\neq 2\ell-1}z_{K_j,M_j}^{\i_j}\Big\|_{L^2_t L^1_x} \\
&\qquad+\sum_{\ell_1\neq \ell_{\max}}\big\|z_{K_{2\ell_1-1},M_{2\ell_1-1}}\big\|_{L^{\infty}_tL^2_x}\big\|z_{K_{j_{\max}},M_{j_{\max}}}\big\|_{L^2_{t,x}}\Big\|\prod_{j\neq 2\ell_1-1,j_{\max}}z_{K_j,M_j}^{\i_j}\Big\|_{L^\infty_{t}L^2_x}\\
&\qquad + \sum_{\ell_1\neq \ell_2\neq \ell_{\max}}\big\|z_{K_{2\ell_1-1},M_{2\ell-1}}\big\|_{L^{\infty}_tL^2_x}\big\|z_{K_{2\ell_2-1},M_{2\ell-1}}\big\|_{L^{\infty}_tL^2_x}\Big\|\prod_{j\neq 2\ell_1-1,2\ell_2-1}z_{K_j,M_j}^{\i_j}\Big\|_{L^2_{t,x}}\Big\}\\
&\le C^kT^{\theta}(k+1)k^{2-b-2\theta}\|\gf_{j_{\max}}\|_{Y^s}\prod_{\substack{j\in\GF\\j\neq j_{\max}}}\|\gf_j\|_{Y^s}\prod_{j\in\DF}\|z_j\|_{X^{s_1,b}},
\end{align*}
where we used that $s_1-(\al-1)(1-b-2\theta)=s_1-(\al-1)(\frac12-\wt\eps-2\theta)<s=\frac{\al-1}2-\eps$ for $0<\theta,\wt\eps,\eps\ll$ small enough (recall that $\al>\frac32$). This is enough for \eqref{NN}.

 The case $K^{(0)}=K_{j_{\max}}$ is handled similarly as in \eqref{Kmax}, except that again $j_{\max=2\ell_{\max}-1}$ is odd:
\begin{align*}
&\Big\|\chi_T(t)\NN_{2k+1}(z_1,...,z_{2k+1})\Big\|_{X^{s,b-1}}\\
&\le C^kT^{\theta}k^\frac12\sum_{K_j,M_j}M_{j_{\max}}^{s_1-\frac{\al-1}2}K_{j_{\max}}^\frac12\Big\{\big\|z_{K_{j_{\max}},M_{j_{\max}}}\big\|_{L^2_{t,x}}\prod_{j\neq j_{\max}}\big\|z_{K_j,M_j}\big\|_{L^{\infty}_{t,x}}\\
& + \|z_{K_{j_{\max}},M_{j_{\max}}}\|_{L^2_{t,x}}\Big\|\prod_{j\neq j_{\max}}z_{K_j,M_j}^{\i_j}\Big\|_{L^\infty_t L^1_x}\\
&\qquad+\sum_{\ell\neq \ell_{\max}}\big\|z_{K_{2\ell-1},M_{2\ell-1}}\big\|_{L^\infty_tL^2_x}\Big\|\prod_{j\neq 2\ell-1}z_{K_j,M_j}^{\i_j}\Big\|_{L^2_t L^1_x} \\
&\qquad+\sum_{\ell_1\neq \ell_{\max}}\big\|z_{K_{2\ell_1-1},M_{2\ell_1-1}}\big\|_{L^{\infty}_tL^2_x}\big\|z_{K_{j_{\max}},M_{j_{\max}}}\big\|_{L^2_{t,x}}\Big\|\prod_{j\neq 2\ell_1-1,j_{\max}}z_{K_j,M_j}^{\i_j}\Big\|_{L^\infty_{t}L^2_x}\\
&\qquad + \sum_{\substack{\ell_1\neq \ell_2\\\ell_1\neq \ell_{\max}\\\ell_2\neq\ell_{\max}}}\big\|z_{K_{2\ell_1-1},M_{2\ell-1}}\big\|_{L^{\infty}_tL^2_x}\big\|z_{K_{2\ell_2-1},M_{2\ell-1}}\big\|_{L^{\infty}_tL^2_x}\Big\|\prod_{\substack{j\neq 2\ell_1-1\\j\neq 2\ell_2-1}}z_{K_j,M_j}^{\i_j}\Big\|_{L^2_{t,x}}\Big\}\\
&\le C^kT^{\theta}(k+1)k^\frac32\|\gf_{j_{\max}}\|_{Y^{s}}\prod_{\substack{j\in\GF\\j\neq j_{\max}}}\|\gf_j\|_{Y^s}\prod_{j\in\DF}\|z_j\|_{X^{s_1,b}},
\end{align*}
since as above $s_1-\frac\al2 < s$, and $b>\frac12$.

 When $K^{(0)}=K_{j_K}$ for some $j_K\in\GF\cup\DF$, $j_K\neq j_{\max}$, we use again \eqref{N1}-\eqref{N2} to get the bound
\begin{align*}
&\Big\|\chi_T(t)\NN_{2k+1}(z_1,...,z_{2k+1})\Big\|_{X^{s,b-1}}\\
&\le C^kT^{\theta}k^\frac12\sum_{K_j,M_j}M_{j_{\max}}^{s_1-\frac{\al-1}2}K_{j_K}^\frac12\Big\{\big\|z_{K_{j_{\max}},M_{j_{\max}}}\big\|_{L^\infty_tL^2_x}\prod_{j\neq j_{\max}}\big\|z_{K_j,M_j}\big\|_{L^2_tL^\infty_x} \\
&\qquad+ \|z_{K_{j_{\max}},M_{j_{\max}}}\|_{L^\infty_tL^2_x}\Big\|\prod_{j\neq j_{\max}}z_{K_j,M_j}^{\i_j}\Big\|_{L^2_t L^1_x}\\
&\qquad+\sum_{\ell\neq \ell_{\max}}\big\|z_{K_{2\ell-1},M_{2\ell-1}}\big\|_{L^\infty_tL^2_x}\Big\|\prod_{j\neq 2\ell-1}z_{K_j,M_j}^{\i_j}\Big\|_{L^2_t L^1_x} \\
&\qquad+\big\|z_{K_{j_K},M_{j_K}}\big\|_{L^2_tL^2_x}\big\|z_{K_{j_{\max}},M_{j_{\max}}}\big\|_{L^\infty_tL^2_x}\Big\|\prod_{j\neq j_K,j_{\max}}z_{K_j,M_j}^{\i_j}\Big\|_{L^\infty_{t}L^2_x}\\
&\qquad+\sum_{\substack{\ell_1\neq \ell_{\max}\\2\ell_1-1\neq j_K}}\big\|z_{K_{2\ell_1-1},M_{2\ell_1-1}}\big\|_{L^{\infty}_tL^2_x}\big\|z_{K_{j_{\max}},M_{j_{\max}}}\big\|_{L^\infty_tL^2_x}\Big\|\prod_{j\neq 2\ell_1-1,j_{\max}}z_{K_j,M_j}^{\i_j}\Big\|_{L^2_{t}L^2_x}\\
&\qquad + \sum_{\substack{\ell_1\neq \ell_2\neq \ell_{\max}\\j_K\neq 2\ell_1-1,2\ell_2-1}}\big\|z_{K_{2\ell_1-1},M_{2\ell-1}}\big\|_{L^{\infty}_tL^2_x}\big\|z_{K_{2\ell_2-1},M_{2\ell-1}}\big\|_{L^{\infty}_tL^2_x}\Big\|\prod_{j\neq 2\ell_1-1,2\ell_2-1}z_{K_j,M_j}^{\i_j}\Big\|_{L^2_{t,x}}\Big\}\\
&\le C^kT^{\theta}(k+1)k^\frac32\|\gf_{j_{\max}}\|_{Y^{s}}\Big(\sum_{K_{j_K},M_{j_K}}K_{j_K}^\frac12\|z_{K_{j_K},M_{j_K}}\|_{L^2_tL^\infty_x}\Big)\prod_{\substack{j\in\GF\\j\neq j_{\max},j_K}}\|\gf_j\|_{Y^s}\prod_{\substack{j\in\DF\\j\neq j_K}}\|z_j\|_{X^{s_1,b}}.
\end{align*}
This is enough for \eqref{NN} since the last norm on $z_{j_K}$ is controlled by $\|\gf_{j_K}\|_{Y^s}$ if $j_K\in \GF$ in view of \eqref{Ysb2},  and by $\|z_{j_K}\|_{X^{s_1,b}}$ if $j_K\in\DF$ due to Sobolev inequality with \eqref{delta}. This concludes the proof of Lemma~\ref{LEM:NN}.
\end{proof}

\subsection{Almost sure global well-posedness and invariance of the Gibbs measure}\label{SUBSEC:GWP2}
As in the proof of Theorem~\ref{THM:GWP1}~(ii), we rely on Bourgain's invariant measure argument \cite{BO94,BO96} to upgrade the almost sure local well-posedness result implied by Proposition~\ref{PROP:LWP} to the almost sure global well-posedness result and invariance of Gibbs measure claimed in Theorem~\ref{THM:GWP2}. Again, we mainly follow \cite{STz1}, and we will only indicate the modifications from the argument presented there. From now on we fix $\gamma>0$, $1+\frac{\sqrt{2}}2 < \al \le 2$, $0<s<\frac{\al-1}2$, $s_1,b>\frac12$ with $s_1<\frac{2\al-1}{\al}s$ as above, and $\be\in(0,\be^\star]$ for some $\be^\star>0$ to be determined later.

We consider again the truncated dynamic \eqref{expNLSN}
\begin{align*}
i\dt u_N + (-\dx^2)^\frac\al2 u_N = 2\gamma\be \Pi_{\le N}\Big[e^{\be|\Pi_{\le N}u|^2}\Pi_{\le N}u\Big],
\end{align*}
which, by Lemma~\ref{LEM:global} is globally defined on $H^s(\T)$, and which leaves the truncated Gibbs measure $\rho_{\al,\be,N}$ in \eqref{GibbsM} invariant. 

We write again $\Phi_N(t)$ for the flow of \eqref{expNLSN} on $H^s(\T)$. Then it follows from \eqref{expNLSN} that
\begin{align*}
\Phi_N(t)(u_0) = \Pi_{\le N}\Phi_N(t)(u_0) + \Pi_{>N}\<1>(t,u_0) = \Phi_N(t)(\Pi_{\le N}u_0) + \<1>(t,\Pi_{> N}u_0).
\end{align*}
If we make the truncated gauge transform
\begin{align*}
v_N = \exp\Big(i\int_0^t\GG(\Phi_N(t')(\Pi_{\le N}u_0))dt'\Big)\Phi_N(t)(\Pi_{\le N}u_0),
\end{align*}
then $v_N$ solves
\begin{align*}
\begin{cases}
i\dt v_N+(-\dx^2)^\frac\al2 = 2\gamma\be \sum_{k\ge 1}\frac{\be^k}{k!}\Pi_{\le N}\Big\{\NN_{2k+1}(\Pi_{\le N}v_N)+\RR_{2k+1}(\Pi_{\le N}v_N)\Big\},\\
v_N(0)=\Pi_{\le N}u_0.
\end{cases}
\end{align*}
In particular $\Pi_{\le N}v_N = v_N$. Then we use the ansatz \eqref{ansatz} to expand
\begin{align*}
v_N = \<1>(\Pi_{\le N}u_0) + w_N,
\end{align*}
where $w_N$ solves \eqref{wN} with $\<1> = \<1>(\Pi_{\le N}u_0)$. In particular we also have $\Pi_{\le N}w_N=w_N$ due to \eqref{wN}-\eqref{FN}.

For $u_0 = \gf+\phi\in Y^s(\T)+H^{s_1}(\T)$, this leads us to
\begin{align*}
\Phi_N(t)u_0=\varphi_N(t)+\psi_N(t)
\end{align*}
where
\begin{align*}
\varphi_N(t)=e^{i\int_0^t\GG(\Phi_N(t')(\Pi_{\le N }u_0))dt'}\Pi_{\le N}\<1>(\gf)+\Pi_{>N}\<1>(\gf)\in Y^s
\end{align*}
and
\begin{align*}
\psi_N(t)= e^{i\int_0^t\GG(\Phi_N(t')(u_0))dt'}\big(\Pi_{\le N}\<1>(\phi)+w_N(t) \big)+\Pi_{>N}\<1>(\phi)\in H^{s_1}.
\end{align*}

We then have a result similar to Proposition~\ref{PROP:Si}.
\begin{proposition}\label{PROP:Si2}
There exists $\be^\star>0$ such that the following holds for any $\be\in (0,\be^\star]$: there exists $C>0$ such that for all $m,N\in\N$, there exists a measurable set $\Si_N^{m}\subset H^s(\T)$ such that\\
\textup{(i)} the bound
\begin{align}\label{measure-Si2}
\rho_{\al,\be,N}(H^s(\T)\setminus\Si_N^{m})\leq 2^{-m+1}
\end{align}
holds;\\
\textup{(ii)} for all $u_0\in\Si_N^{m}$ and $t\in\R$, the solution $\Phi_N(t)(u_0)$ to \eqref{expNLSNN} satisfies
\begin{align}
\big\|\Phi_N(t)(u_0)\big\|_{Y^{s}} \le C\sqrt{m(1+\log(1+|t|))};
\label{estim-Si2}
\end{align}
\textup{(iii)} there exists $c>0$ such that for every $t_0$, every $m\ge 1$ and $N\in\N$,
\begin{align}\label{inclusion}
\Phi_N(t_0)(\Si_N^m)\subset \Si_N^{Dm(1+\log_2(1+|t_0|))}.
\end{align}
\end{proposition}

Before moving on to the proof of Proposition~\ref{PROP:Si2}, recall the following lemma from \cite{STz1}.

\begin{lemma}[Lemma 6.7 in \cite{STz1}]\label{decomposition1}
	Assume that $u_0\in Y^s(\T)+H^{s_1}(\T)$ such that for some $R>0, \delta>0,$
	$$ \|u_0\|_{Y^s+H^{s_1}}\leq R, \quad \|\Pi_{>N}u_0\|_{Y^s+H^{s_1}}\leq N^{-\delta}R.
	$$
	Then there exists $\gf\in Y^s, \phi\in H^{s_1}$, such that
	$$ \|\gf\|_{\Y^s}+\|\phi\|_{H^{s_1}}\leq 2A_0(R+1), \quad \|\Pi_{>N}\gf\|_{Y^s}+\|\Pi_{>N}\phi\|_{H^{s_1}}\leq N^{-\delta}A_0(R+1),
	$$
	where $A_0>0$ is a uniform constant.
	\end{lemma}

\begin{proof}[Proof of Proposition~\ref{PROP:Si2}]

			For $m,j\in\N$ and $D>0$ to be chosen later, we define the set
			\begin{equation}\label{BmjN}
			B_N^{m,j}(D):=\big\{u_0\in H^{s}(\T),~ \|u_0\|_{Y^s+H^{s_1}}\leq D\sqrt{mj}, ~\|\Pi_{>N}u_0\|_{Y^s+H^{s_1}}\leq N^{-\delta}D\sqrt{mj}  \}
			\end{equation} 
			By Lemma \ref{decomposition1}, for $u_0\in B_N^{m,j}(D)$, there exists a decomposition
			$ u_0=\gf+\phi $, such that
			$$
			 \|\gf\|_{Y^s}+\|\phi\|_{H^{s_1}}\leq 2A_0D\sqrt{mj},
			\quad  \|\Pi_{>N}\gf\|_{Y^s}+\|\Pi_{>N}\phi\|_{H^{s_1}}\leq A_0N^{-\dl}\s_\al D\sqrt{mj},
			$$
	for some small $\dl>0$ in view of \eqref{tailY} and $|\s_{\al-2s-2\eps_0}-\s_{\al,2s-2\eps_0,N}|^\frac12\les N^{-\dl}\s_\al^\frac12$.

			From Proposition~\ref{PROP:LWP}, the local time of existence is 
			$$ \tau_{m,j}=c(4\gamma\be A_0^2D^2mj)^{-\theta^{-1}}e^{-\be C_0\theta^{-1}4A_0^2D^2mj}
			$$
			for some $0<\theta\ll 1$, and it holds
			\begin{equation*}\label{boundeness}
			\sup_{|t|\leq \tau_{m,j}}\big(\|\varphi_N(t)\|_{Y^s}+\|\psi_N(t)\|_{H^{s_1}} \big)\leq 4A_0D\sqrt{mj},
			\end{equation*}
			by using that the $H^{s_1}$-norm is invariant by the linear flow, the property \eqref{Ysb3} in Lemma~\ref{LEM:Y}, and $\Pi_{>N}w_N=0$. Therefore, for $|t|\leq \tau_{m,j}$
			\begin{equation*}\label{boundeness1}
			\|\Phi_N(t)\phi\|_{Y^s+H^{s_1} }\leq 4A_0D\sqrt{mj}.
			\end{equation*}
Since
$$ \Pi_{>N}\Phi_N(t)(u_0)=\Pi_{>N}\<1>(\gf+\phi),
$$ 	
from \eqref{Ysb3} again, we obtain that
\begin{align*}
\|\Pi_{>N}\Phi_N(t)(u_0)\|_{Y^s+H^{s_1}}\leq A_1\|\Pi_{>N}u_0\|_{Y^s+H^{s_1}}\leq A_1N^{-\delta}\s_\al D\sqrt{mj}.
\end{align*}	

Since $\|u_0\|_{Y^s+H^{s_1}}\leq \|u_0\|_{Y^s}$, we deduce from Lemma~\ref{LEM:init} and~\ref{LEM:devY} that
$$ \rho_{\al,\be,N}(H^{s}(\T)\setminus B_N^{m,k}(D))\leq e^{-c\s_\al^{-1}D^{2}mj}.$$
Next, we set
\begin{equation}\label{dataset0}
\Si_N^{m,j}(D)=\bigcap_{|r|\leq \frac{2^j}{\tau_{m,j}}}\Phi_N(-r\tau_{m,j})\big(B_N^{m,j}(D)\big),
\end{equation}
and from the invariance of $\rho_{\al,\be,N}$ under the flow $\Phi_N(t)$, we have
\begin{align*}
\rho_{\al,\be,N}\big(H^{s}(\T)\setminus \Si_N^{m,j}(D) \big)&\leq \sum_{|r|\leq \frac{2^j}{\tau_{m,j}}} \rho_{\al,\be,N}\big(H^{s}(\T)\setminus \Phi_N(-r\tau_{m,j})\big(B_N^{m,j}(D) \big) \big)\\
 &\le \frac{2^{j+2}}{\tau_{m,j}}\rho_{\al,\be,N}\big(H^{s}(\T)\setminus B_N^{m,j}(D) \big)\intertext{Using that $d\rho_{\al,\be,N}(u)=\ZZ_{\al,\be,N}^{-1}e^{-\gamma V_\be(\Pi_{\le N}u)}d\mu_\al(u)$ with $\gamma, V_\be\ge 0$, we continue with}
  &\les \frac{2^{j+2}}{\tau_{m,j}}\mu_\al\big(H^{s}(\T)\setminus B_N^{m,j}(D) \big)\intertext{Then, Lemma~\ref{LEM:devY} with $\|\cdot\|_{Y^s}\le \|\cdot\|_{Y^s+H^{s_1}}$ and the definition of $\tau_{m,j}$ give the bound}
&\le  \frac{2^{j+2}}{c}(4\gamma\beta A_0^2D^2mj)^{\theta^{-1}}e^{\be C_0\theta^{-1}4A_0^2D^2mj}e^{-c\s_\al^{-1}D^{2}mj} \le 2^{mj},
\end{align*}
provided that $D$ is chosen large enough and $$0<\be^\star \ll \theta(C_0A_0^2)^{-1}\s_\al.$$ Now we define the desired data set by
\begin{equation}\label{dataset1}
\Si_N^m:=\bigcap_{j\geq 1}\Si_N^{m,j}(D).
\end{equation}
It is clear that $\rho_N\big(H^s(\T)\setminus\Si_N^m\big)\leq 2^{-m+1}$. The remaining properties \eqref{estim-Si2}-\eqref{inclusion} are then proved exactly as in \cite[Proposition 6.6]{STz1}.

		\end{proof}
		We finally define
		\begin{align*}
	\Sigma^m = \limsup_{N\to\infty}\Si_N^m
	\end{align*}
		and
		\begin{align*}
		\Si = \bigcup_{m\in\N}\Si^m.
		\end{align*}
		
	As in Subsection~\ref{SUBSEC:GWP}, these definitions and \eqref{measure-Si2} ensure that $\Si$ is of full $\rho_{\al,\be}$-measure.
		
\noi Moreover, as above and as in \cite{BTT1,STz1}, the proof of Theorem~\ref{THM:GWP2}~(ii) relies on the following approximation result, which is similar to \cite[Proposition 6.4]{STz1}.

\begin{proposition}\label{enhanced-localconvergence}
Let $(u_{0,n})\subset Y^s+H^{s_1}$, $u_0\in Y^s+H^{s_1}$ satisfying
	$$ \|u_{0,n}\|_{Y^s+H^{s_1}}+\|u_0\|_{Y^s+H^{s_1}}\leq R,\quad \lim_{n\rightarrow\infty} \|u_{0,n}-u_0\|_{Y^s+H^{s_1}}=0.
	$$
	Let $N_p\to+\infty$ be a subsequence of $\N$. 
	Then there exist $c,C>0, \theta>0$, such that for all $t\in [-\tau_R,\tau_R]$ with $\tau_R=c(\gamma \be R^2)^{-\theta^{-1}}e^{-\be C\theta^{-1} R^2}$, we have
	\begin{align}\label{convergence:iterated1}
	\lim_{p\rightarrow\infty}\|\Phi_{N_p}(t)(u_{0,p})-\Phi(t)(u_0) \|_{Y^s+H^{s_1}}=0.
	\end{align}
\end{proposition}

\begin{proof}
By \cite[Lemma 6.3]{STz1}, there exist sequences $(\gf_{p})_{p\in\N}\subset Y^s(\T), (\phi_{p})_{p\in\N}\subset H^{s_1}(\T)$, and $\gf\in Y^s(\T), \phi\in H^{s_1}(\T)$, such that
\begin{align}\label{decomposition}
 u_{0,p}=\gf_{p}+\phi_{p}, \quad u_0=\gf+\phi,
\end{align}
	and
\begin{align}\label{bd-g-phi-p}	
	\|\gf_{p}\|_{Y^s}+\|\phi_{p}\|_{H^{s_1}}\leq R+2, \quad \|\gf\|_{Y^s}+\|\phi\|_{H^{s_1}}\leq R+2
\end{align}
	and
\begin{align}\label{lim-g-phi-p}
\lim_{p\rightarrow\infty}\big(\|\gf_{p}-\gf\|_{Y^s}+\|\phi_{p}-\phi\|_{H^{s_1}} \big)=0.
\end{align}
	Moreover, we have
\begin{align}\label{tail-phi} 
\lim_{p\rightarrow\infty}\|\Pi_{>N_p}\phi\|_{H^{s_1}}=0,\text{ and } \lim_{p\rightarrow\infty}\|\Pi_{>N_p}\phi_{p}\|_{H^{s_1}}=0
\end{align}
the last convergence being due to $\|\Pi_{>N_p}\phi_{p}\|_{H^{s_1}}\leq \|\Pi_{>N_p}\phi\|_{H^{s_1}}+\|\Pi_{>N_p}(\phi_{p}-\phi)\|_{H^{s_1}}$.

Now we claim the following convergence result, similar to that of Proposition 3.3 in \cite{STz1}.
\begin{lemma}\label{LEM:cv}
Let $R\geq 1$ and $(\gf_{p})\subset Y^s(\T)$, $\gf\in Y^s(\T)$ satisfying $\|\gf_{p}\|_{Y^s}\le R$ and $\|\gf\|_{Ys}\le R$. Let also $(\phi_p)\subset H^{s_1}(\T)$ satisfying $\|\phi_p\|_{H^{s_1}}\leq 2R$. Let $N_p\rightarrow\infty$ be a subsequence of $\mathbb{N}$. Assume moreover that
\begin{align}\label{cv-g-phi}
 \lim_{p\rightarrow\infty}\big(\|\gf_p-\gf\|_{Y^s}+\|\phi_p-\phi\|_{H^{s_1}}\big)=0.
 \end{align}
Then we have for all $t\in[-\tau_R,\tau_R]$
\begin{align}\label{decomp-phip}
\Phi_{N_p}(t)(\gf_p+\phi_p)=e^{i\int_0^t\GG\big(\Pi_{\le N_p}\<1>(\gf_p+\phi_p)+w_p\big)(t')dt'} \big(\Pi_{\le N_p}\<1>(\gf_p+\phi_p)+w_p\big)+\Pi_{>N_p}\<1>(\gf),
\end{align}
and
\begin{align}\label{decomp-phi}
\Phi(t)(\gf+\phi)=e^{i\int_0^t\GG\big(\<1>(\gf+\phi)+w\big)(t')dt'}\big(\<1>(\gf+\phi)+w\big),
\end{align}
with
\begin{align}\label{cv-w}
\lim_{p\to\infty}\|w_p-w\|_{X^{s_1,b}(\tau_R)}=0.
\end{align}
\end{lemma}
\noi The proof of Lemma~\ref{LEM:cv} is an adaptation of the proofs of Lemma~\ref{LEM:approx} and of Proposition~\ref{PROP:LWP}, and we postpone it to the end of the section in order to finish that of Proposition~\ref{enhanced-localconvergence}.

Let then $(w_p)$ be given by Lemma~\ref{LEM:cv}. In particular from \eqref{cv-w} and \eqref{embeddings} it holds
	\begin{align}\label{cv-wp} 
	\lim_{p\rightarrow\infty}\|w_p(t)-w(t)\|_{C_{\tau_R}H^{s_1}}=0.
	\end{align}
To prove the convergence of $\Phi_{N_p}$, we thus need to study the convergence of the gauge. Define 
\begin{align*}
b_p(t)=e^{i\int_0^t\GG\big(\Pi_{\le N_p}\<1>(\gf_p+\phi_p)+w_p\big)(t')dt'}
\end{align*}
and
\begin{align*}
b(t)=e^{i\int_0^t\GG\big(\<1>(\gf+\phi)+w\big)(t')dt'}.
\end{align*} 
Then, similarly to \eqref{bd-gauge}, we have for all $t\in[-\tau_R,\tau_R]$
\begin{align}
&\big|b_p(t)-b(t)\big|\notag\\
& \les \int_0^t\Big|\GG\big(\Pi_{\le N_p}\<1>(\gf_p+\phi_p)+w_p\big)-\GG\big(\<1>(\gf+\phi)+w\big)\Big|(t')dt'\\
&\les \big(\big\|\Pi_{\le N_p}\gf_p-\gf\big\|_{Y^s} + \|\Pi_{\le N_p}\phi_p-\phi\|_{H^{s_1}}+\|w_p-w\|_{X^{s_1,b}(\tau_R)}\big) \exp\Big\{\be C \big(\|\Pi_{\le N_p}\gf_p\|_{Y^s}\notag\\
&\qquad+\|\gf\|_{Y^s}+\|\Pi_{\le N_p}\phi_p\|_{H^{s_1}}+\|\phi\|_{H^{s_1}}+\|w_p\|_{X^{s_1,b}(\tau_R)}+\|w\|_{X^{s_1,b}(\tau_R)}\big)^2\Big\}\notag\\
&\les \big(\big\|\Pi_{\le N_p}\gf_p-\gf\big\|_{Y^s} + \|\Pi_{\le N_p}\phi_p-\phi\|_{H^{s_1}}+\|w_p-w\|_{X^{s_1,b}(\tau_R)}\big)e^{\be C R^2}\notag\\
& \too 0\label{cv-b}
\end{align} 
as $p\to \infty$, in view of the bounds \eqref{bd-g-phi-p}-\eqref{lim-g-phi-p} with \eqref{normN}.

Thus, we can decompose
	$$ \Phi_{N_p}(t)u_{0,p}-\Phi(t)u_0=\varphi_p(t)+\psi_p(t),
	$$
	where
	\begin{align}
	\varphi_p(t)&=\big(b_p(t)-b(t) \big)\Pi_{\le N_p}\<1>(\gf_p)+\big(1-b(t)\big)\Pi_{>N_p}\<1>(\gf)+\Pi_{>N_p}\<1>\big(\gf_p-\gf\big)\notag\\
	&\qquad+	b(t)\Pi_{\le N_p}\<1>(\gf_p-\gf),\label{phip}
	\end{align}
	and
	\begin{align*}
	\psi_p(t)=&b_p(t)\big(\Pi_{\le N_p}\<1>(\phi_p-\phi)+w_p(t)-w(t) \big)+(b_p(t)-b(t))\big(\Pi_{\le N_p}\<1>\phi+w(t)\big)\\
	+&(1-b(t))\Pi_{>N_p}\<1>(\phi)+\Pi_{>N_p}\<1>(\phi_p-\phi).	
	\end{align*}
The convergence \eqref{cv-wp} ensures that $\psi_k(t)\rightarrow 0$ in $H^{s_1}(\T)$, for any $|t|\le \tau_R$. 

It remains to show the convergence $\varphi_p(t)\rightarrow 0$ as $p\to\infty$ in $Y^s(\T)$ for all $|t|\leq \tau_R$. From \eqref{normN} and \eqref{cv-b} we have 
	$$ \lim_{p\rightarrow\infty}\|\big(b_p(t)-b(t)\big)\Pi_{\le N_p}\<1>(\gf_p) \|_{Y^s}=0.
	$$ 
The last part $\Pi_{>N_p}\<1>\big(\gf_p-\gf\big) +	b(t)\Pi_{\le N_p}\<1>(\gf_p-\gf)$ in \eqref{phip} converges to 0 in $Y^s(\T)$ by \eqref{normN} and \eqref{lim-g-phi-p}. For the remaining part $\big(1-b(t)\big)\Pi_{>N_p}\<1>(\gf)$ of \eqref{phip}, the convergence for the $H^s$ part of the $Y^s$ norm is clear, and its $Z^s$ norm goes to 0 by dominated convergence through the exact same argument as in the proof of Proposition 6.4 in \cite{STz1}.

\end{proof}
We now present the proof of the approximation result in Lemma~\ref{LEM:cv}.
\begin{proof}[Proof of Lemma~\ref{LEM:cv}]
The existence of $w_p,w\in X^{s_1,b}(\tau_R)$ such that \eqref{decomp-phip}-\eqref{decomp-phi} hold follows directly from Proposition~\ref{PROP:LWP}. For the convergence \eqref{cv-wp}, we have in view of the Duhamel formula \eqref{DuhamelQ}
\begin{align*}
&\|w_p-w\|_{X^{s_1,b}(\tau_R)}\\
& \les \Big\|F_{N_p}\big(\Pi_{\le N_p}\<1>(\gf_p+\phi_p)+w_p\big)-F\big(\<1>(\gf+\phi)+w\big)\Big\|_{X^{s_1,b-1}(\tau_R)}\notag\\
&\les \sum_{k\ge 1}\frac{\be^k}{k!}\sum_{m_+=0}^{k+1}\sum_{m_-=0}^k\binom{k+1}{m_+}\binom{k}{m_-}\big\{ \1_{\NN_{2k+1}} + \1_{\RR_{2k+1}}\big\},
\end{align*}
where
\begin{align*}
\1_{\NN_{2k+1},p}&=\Big\|\Pi_{\le N_p}\big[\NN_{2k+1}\big(\{\Pi_{\le N_p}\<1>(\gf_p),\Pi_{\le N_p}(\<1>(\phi_p)+w_p)\}_{(m_+,m_-)}\big)\big]\\
&\qquad\qquad-\NN_{2k+1}\big(\{(\<1>(\gf),\<1>(\phi)+w)\}_{(m_+,m_-)}\big)\Big\|_{X^{s_1,b}(\tau_R)}
\end{align*}
and
\begin{align*}
\1_{\RR_{2k+1},p}&=\Big\|\Pi_{\le N_p}\big[\RR_{2k+1}\big(\{\Pi_{\le N_p}\<1>(\gf_p),\Pi_{\le N_p}(\<1>(\phi_p)+w_p)\}_{(m_+,m_-)}\big)\big]\\
&\qquad\qquad-\RR_{2k+1}\big(\{\<1>(\gf),\<1>(\phi)+w\}_{(m_+,m_-)}\big)\Big\|_{X^{s_1,b}(\tau_R)}.
\end{align*}
Arguing as in the proof of \eqref{diff-Gamma} with the use of \eqref{normN}, we have by multilinearity
\begin{align*}
&\1_{\NN_{2k+1},p}\\
 &\le \Big\|\Pi_{> N_p}\NN_{2k+1}\big(\{(\<1>(\gf_p),\<1>(\phi_p)+w_p)\}_{(m_+,m_-)}\big)\Big\|_{X^{s_1,b}(\tau_R)}\\
&\qquad+\sum_{j=1}^{m_++m_-}\|\Pi_{\le N_p}\gf_p-\gf\|_{Y^{s}}\|\Pi_{\le N_p}\gf_p\|_{Y^{s}}^{m_++m_--j}\|\gf\|_{Y^{s}}^{j-1}\|\<1>(\phi_1)+w_1\|_{X^{s_1,b}(T)}^{2k+1-m_+-m_-}\notag\\
&\qquad+ \sum_{j=m_++m_-+1}^{2k+1}\|\gf\|_{Y^{s}}^{m_++m_-}\|\Pi_{\le N_p}(\<1>(\phi_p)+w_p)-(\<1>(\phi)+w)\|_{X^{s_1,b}(T)}\\
&\hspace{4cm}\cdot\|\Pi_{\le N_p}\<1>(\phi_p)+w_p\|_{X^{s_1,b}(T)}^{2k+1-j}\|\<1>(\phi)+w\|_{X^{s_1,b}(T)}^{j-m_+-m_--1}
\end{align*}
Thus, summing in $m_+,m_-,k$, we get
\begin{align*}
&\sum_{k\ge 1}\frac{\be^k}{k!}\sum_{m_+=0}^{k+1}\sum_{m_-=0}^k\binom{k+1}{m_+}\binom{k}{m_-} \1_{\NN_{2k+1},p}\\
&\les \II_{\NN,p}+ C|\gamma|\be \tau_R^\theta \big(\|\Pi_{\le N_p}\gf_p-\gf\|_{Y^{s}(T)}+\|\Pi_{\le N_p}\phi_p-\phi\|_{H^{s_1}}+\|w_p-w\|_{X^{s_1,b}(T)}\big)e^{\be C R^2}\notag\\
&\le C\II_{\NN} +C|\gamma|\be \tau_R^\theta \big(\|\Pi_{\le N_p}\gf_p-\gf\|_{Y^{s}(T)}+\|\Pi_{\le N_p}\phi_p-\phi\|_{H^{s_1}}\big)e^{\be C R^2}\\
&\qquad\qquad+\frac14\|w_1-w_2\|_{X^{s_1,b}(T)},
\end{align*}
where the last estimate follows from the definition of $\tau_R$, and where
\begin{align*}
&\II_{\NN,p}\\
&=\sum_{k\ge 1}\frac{\be^k}{k!}\sum_{m_+=0}^{k+1}\sum_{m_-=0}^k\binom{k+1}{m_+}\binom{k}{m_-}\Big\|\Pi_{> N_p}\NN_{2k+1}\big(\{(\<1>(\gf_p),\<1>(\phi_p)+w_p)\}_{(m_+,m_-)}\big)\Big\|_{X^{s_1,b}(\tau_R)}.
\end{align*}
In particular, again from Lemma~\ref{LEM:NN} as in the proof of Proposition~\ref{PROP:LWP}, and using \eqref{normN}, we have that
\begin{align*}
\II_{\NN,p}&\les C\tau_R^\theta\sum_{k\ge 1}\frac{\be^k}{k!}\sum_{m_+=0}^{k+1}\sum_{m_-=0}^k\binom{k+1}{m_+}\binom{k}{m_-}\big\|\gf_p\big\|_{Y^s}^{m_++m_-}\\
&\qquad\qquad\qquad\qquad\times\big(\big\|\phi_p\big\|_{H^{s_1}}+\big\|w_p\big\|_{X^{s_1,b}(\tau_R)}\big)^{2k+1-m_+-m_-}\\
&\les C\tau_R^\theta e^{\be CR^2},
\end{align*}
uniformly in $p$, by \eqref{bd-g-phi-p} and \eqref{bdw}. Thus by dominated convergence, it holds $\II_{\NN,p}\too 0$ as $p\to\infty$. Similar estimates for $\1_{\RR_{2k+1}}$ finally lead to
\begin{align}
\|w_p-w\|_{X^{s_1,b}(\tau_R)}&\le  C\II_{\NN,p}+C\II_{\RR,p}\notag\\
&\qquad +C|\gamma|\be \tau_R^\theta \big(\|\Pi_{\le N_p}\gf_p-\gf\|_{Y^{s}(T)}+\|\Pi_{\le N_p}\phi_p-\phi\|_{H^{s_1}}\big)e^{\be C R^2}\notag\\
&\qquad\qquad+\frac12\|w_1-w_2\|_{X^{s_1,b}(T)}.\label{bd-cv-wp}
\end{align}
The convergences $\II_{\NN,p},\II_{\RR,p}\too 0$ and \eqref{cv-g-phi} with \eqref{bd-cv-wp} finally show \eqref{cv-wp}. This proves Lemma~\ref{LEM:cv}.
\end{proof}

\end{document}